\documentclass[aop]{imsart}
\RequirePackage{amsthm,amsmath,amsfonts,amssymb}
\RequirePackage{array}
\RequirePackage{graphicx}
\RequirePackage{subcaption}
\RequirePackage{float}
\RequirePackage{placeins}
\RequirePackage{makecell}
\RequirePackage{dsfont}
\RequirePackage[numbers,sort&compress]{natbib}
\RequirePackage[colorlinks,citecolor=blue,urlcolor=blue,linkcolor=blue]{hyperref}
\RequirePackage{mathtools}
\RequirePackage[table]{xcolor}
\RequirePackage{booktabs}
\RequirePackage[overload]{empheq}
\usepackage{comment}

\startlocaldefs
\numberwithin{equation}{section}
\theoremstyle{plain}
\newtheorem{lemma}{Lemma}[section]
\newtheorem{theorem}[lemma]{Theorem}
\newtheorem{proposition}[lemma]{Proposition}
\newtheorem{corollary}[lemma]{Corollary}
\theoremstyle{definition}

\newtheorem{remark}[lemma]{Remark}

\mathtoolsset{showonlyrefs}
\allowdisplaybreaks
\newcommand{\cip}{\stackrel{\P}{\rightarrow}}
\newcommand{\cas}{\stackrel{\rm a.s.}{\rightarrow}}
\newcommand{\eid}{\stackrel{\rm d}{=}}
\newcommand{\cid}{\stackrel{\rm d}{\rightarrow}}
\newcommand{\slv}{\stackrel{\text{sl.v.}}{=}}
\newcommand{\limn}{\lim_{n \to \infty}}
\newcommand{\nto}{n\to\infty}

\newcommand{\E}{{\mathbb E}}
\newcommand{\N}{\mathbb{N}}
\newcommand{\R}{\mathbb{R}}

\newcommand{\tr}{\operatorname{tr}}
\newcommand{\diag}{\operatorname{diag}}
\newcommand{\Var}{\operatorname{Var}}
\newcommand{\norm}[1]{\|#1\|}
\newcommand{\dint}{\,\mathrm{d}}
\newcommand{\e}{\mathrm e}
\newcommand{\lhs}{left-hand side}
\newcommand{\rhs}{right-hand side}
\newcommand{\MP}{Mar\v cenko--Pastur }
\newcommand{\vep}{\varepsilon}

\newcommand{\la}{\lambda}
\newcommand{\tx}{\tilde{\mathbf{x}}}
\newcommand{\X}{\mathbf X}
\newcommand{\Y}{\mathbf Y}
\newcommand{\bfx}{\mathbf x}
\newcommand{\bfA}{\mathbf A}

\newcommand{\bfI}{\mathbf I}
\newcommand{\bfR}{\mathbf R}
\newcommand{\bfS}{\mathbf S}

\newcommand{\Tone}{\widetilde{T}_1}
\newcommand{\tbeta}{\widetilde{\beta}}
\newcommand{\etalim}{\eta_{\alpha,c_\delta}}
\newcommand{\nulim}{\nu_{\alpha,c_\delta}}

\newcommand{\barr}{\begin{array}}
\newcommand{\earr}{\end{array}}

\makeatletter
\let\save@mathaccent\mathaccent
\newcommand*\if@single[3]{\setbox0\hbox{${\mathaccent"0362{#1}}^H$}\setbox2\hbox{${\mathaccent"0362{\kern0pt#1}}^H$}\ifdim\ht0=\ht2 #3\else #2\fi}
\newcommand*\rel@kern[1]{\kern#1\dimexpr\macc@kerna}
\newcommand*\widebar[1]{\@ifnextchar^{{\wide@bar{#1}{0}}}{\wide@bar{#1}{1}}}
\newcommand*\wide@bar[2]{\if@single{#1}{\wide@bar@{#1}{#2}{1}}{\wide@bar@{#1}{#2}{2}}}
\newcommand*\wide@bar@[3]{\begingroup\def\mathaccent##1##2{\let\mathaccent\save@mathaccent\if#32 \let\macc@nucleus\first@char \fi\setbox\z@\hbox{$\macc@style{\macc@nucleus}_{}$}\setbox\tw@\hbox{$\macc@style{\macc@nucleus}{}_{}$}\dimen@\wd\tw@\advance\dimen@-\wd\z@\divide\dimen@ 3\@tempdima\wd\tw@\advance\@tempdima-\scriptspace\divide\@tempdima 10\advance\dimen@-\@tempdima\ifdim\dimen@>\z@ \dimen@0pt\fi\rel@kern{0.6}\kern-\dimen@\if#31\overline{\rel@kern{-0.6}\kern\dimen@\macc@nucleus\rel@kern{0.4}\kern\dimen@}\advance\dimen@0.4\dimexpr\macc@kerna\let\final@kern#2\ifdim\dimen@<\z@ \let\final@kern1\fi\if\final@kern1 \kern-\dimen@\fi\else\overline{\rel@kern{-0.6}\kern\dimen@#1}\fi}\macc@depth\@ne\let\math@bgroup\@empty \let\math@egroup\macc@set@skewchar\mathsurround\z@ \frozen@everymath{\mathgroup\macc@group\relax}\macc@set@skewchar\relax\let\mathaccentV\macc@nested@a\if#31\macc@nested@a\relax111{#1}\else\def\gobble@till@marker##1\endmarker{}\futurelet\first@char\gobble@till@marker#1\endmarker\ifcat\noexpand\first@char A\else\def\first@char{}\fi\macc@nested@a\relax111{\first@char}\fi\endgroup}
\makeatother
\renewcommand{\bar}{\widebar}
\renewcommand{\P }{{\mathbb P}}
\endlocaldefs

\begin{document}
\begin{frontmatter}
\title{Non-Gaussian fluctuations for traces of squared sample correlation matrices in high dimensions}
\runtitle{Phase transitions for sample correlation matrices}
\begin{aug}
\author[A]{\fnms{Johannes}~\snm{Heiny}\ead[label=e1]{heiny@kth.se}}
\author[A]{\fnms{Xuechun}~\snm{Hu}\ead[label=e2]{xuechunh@kth.se}}
\author[B]{\fnms{Felix}~\snm{Seo}\ead[label=e3]{felix.seo@handelsbanken.se}}
\address[A]{Department of Mathematics, KTH Royal Institute of Technology\printead[presep={,\ }]{e1,e2}}
\address[B]{Svenska Handelsbanken\printead[presep={,\ }]{e3}}
\end{aug}
\begin{abstract}
We provide limit theory for the trace of the squared sample correlation matrix $\mathbf R$, constructed from $n$ observations of a $p$-dimensional random vector with iid components. If the entries have finite fourth moment and $p$ and $n$ grow proportionally, it is known that $\tr(\mathbf R^2)$ satisfies a central limit theorem (CLT) and the centering and scaling sequences are universal in the sense that they do not depend on the entry distribution. Under symmetry and regular variation assumption with index $\alpha$ and any growth rate of the dimension, we prove that the universal CLT remains valid for $\alpha >3$. For $\alpha<3$, we identify a critical dimension growth at which the fluctuations of $\tr(\mathbf R^2)$ become non-Gaussian. Moreover, if the dimension $p$ grows faster and $\alpha\le 3$ we establish a non-universal CLT with norming sequences depending on the value of $\alpha$.  Our findings are illustrated in a simulation study.
\end{abstract}
\begin{keyword}[class=MSC]\kwd{Primary 60B20}\kwd{Secondary 60F05}\kwd{60F10}\kwd{60G10}\kwd{60G55}\kwd{60G70}\end{keyword}
\begin{keyword}\kwd{Sample correlation matrix}\kwd{high dimension}\kwd{linear spectral statistics}\kwd{non-Gaussian limit} \kwd{heavy tails}\end{keyword}
\end{frontmatter}

\section{Introduction}
Measuring the dependence between random variables has always been a fundamental task in statistics. Starting with the early works of  Pearson \cite{Pearson1920}, Kendall \cite{Kendall1938}, Hoeffding \cite{Hoeffding1948} and Blum \cite{blum1961}, several measures of dependence or association have been introduced and analyzed by numerous authors. An outstanding role is played by Pearson's correlation coefficient, a measure of the linear dependency of two random variables, about which  most students learn early on in their studies. Motivated by its importance for statistical inference and estimation, many works are devoted to its stochastic properties in different frameworks. For example, in time series analysis, the notion of correlation plays a vital role in multivariate statistical analysis for  parameter estimation, goodness-of-fit tests, change-point detection, etc.; see for example the classical monographs \cite{brockwell:davis:1991,priestley:1981}.

With the rapid advancements of data collection devices, many modern fields such as biological engineering, telecommunications and finance require the analysis of high-dimensional data sets where the dimension $p$ and the sample size $n$ are of comparable magnitude. As a result traditional results from multivariate analysis, which rely on the assumption that the dimension remains fixed and thus is negligible compared to the sample size, are typically not applicable in high-dimensional regimes. Driven by such challenges, random matrix theory - as outlined in the monographs \cite{bai:silverstein:2010, yao:zheng:bai:2015} - aims to provide a deeper understanding of differences that arise when $p$ is assumed to grow with $n$. A standard assumption is that the ratio $p/n$ approaches some positive constant. It is worth mentioning that a regime where $p=\sqrt{n}$ might lead to completely different asymptotic theory than (say) $p=n$. In practical applications, however, $p/n$ is always some positive number and it is therefore non-trivial to distinguish between various regimes.

\subsection{Model and assumptions}
Consider a $p$-dimensional population $\bfx=(X_1,\ldots,X_p)^{\top}\in\R^p$, where the components $X_i$ are independent and identically distributed (iid), non-degenerate random variables with mean zero. For a sample $\bfx_1,\ldots,\bfx_n$ from the population we construct the data matrix $\X=\X_n=(\bfx_1,\ldots,\bfx_n)=(X_{ij})_{1\le    i\le p; 1\le j  \le n}$, the sample covariance matrix $\bfS=\bfS_n =n^{-1} \X\X^{\top}$ and the sample correlation matrix $\bfR$,
\begin{equation}\label{eq:defRY}
    \bfR =\bfR_n =\{\diag(\bfS_n)\}^{-1/2}\, \bfS_n\{\diag(\bfS_n)\}^{-1/2}= \Y \Y^{\top}\,. 
\end{equation}
Here the standardized  matrix $\Y=\Y_n=(Y_{ij})_{1\le i\le p; 1\le j  \le n}$ for the sample correlation matrix has entries 
\begin{equation}\label{def:R}
    Y_{ij}=Y_{ij}^{(n)}=\frac{X_{ij}}{\sqrt{X_{i1}^2+\cdots+X_{in}^2}}\,,
\end{equation}
which depend on $n$. Throughout the paper, we often suppress the dependence on $n$ in our notation. Since $Y_{ij}$ is invariant with respect to a scaling of the $X_{ij}$'s, we will assume without loss of generality that $\E[X_{11}^2]=1$ whenever $\E[X_{11}^2]$ is finite. In this paper, we will often assume that $|X_{11}|$ has a regularly varying tail with index $\alpha>0$, that is 
\begin{equation}\label{eq:regvar}
    \P(|X_{11}|>x)=   x^{-\alpha}\, L(x)\,,\qquad x>0\,,
\end{equation}
for a function $L$ that is slowly varying at infinity. Thus, regularly varying distributions possess power-law tails and moments of $|X_{11}|$ of higher order than $\alpha$ are infinite. Typical examples include the Pareto distribution with parameter $\alpha$ and the $t$-distribution with $\alpha$ degrees of freedom. In addition, we assume that the distribution of $X_{11}$ is symmetric, that is, $X_{11} \eid -X_{11}$. We consider the high-dimensional regime
\begin{equation*}
    p=p_n \to \infty, \qquad \text{as} \quad \nto\,. 
\end{equation*} 

\subsection{Background} 
Since the pioneering works \cite{marchenko:pastur:1967, wigner:1955, wigner:1957}, 
the limiting eigenvalue distribution of various types of random matrices received a lot of interest in the literature: For example, \cite{silverstein1995analysis, fleermann:heiny:2023} on sample covariance matrices, \cite{bryc2006spectral, catalano:fleermann:2024, fleermann:kirsch:2021} on Hankel, Markov, Toeplitz and band matrices, \cite{bose2009limiting} on circulant type matrices, \cite{wang:yao:2016} on auto-covariance matrices, \cite{li:wang:yao:2022} on spatial-sign covariance matrices, \cite{li_et_al_2023} on distance covariance matrices, just to name a few. Many popular test statistics, such as the likelihood ratio statistic for testing independence of a normal population, can be expressed as a function of the eigenvalues of the sample covariance matrix $\bfS$ or the sample correlation matrix $\bfR$. For a function $f:\R\to \R$ and a random matrix ${\bf A}$ with $p$ real eigenvalues $\la_{1}( {\bf A} ) \ge \cdots \ge\la_{p}(  {\bf A})$, we call $\sum_{i=1}^p f(\lambda_i(\bfA))$ a linear spectral statistic of $\bfA$. In the proportional regime $p/n\to \gamma\in (0,\infty)$, the spectral properties of the sample covariance matrix $\bfS$ have been well studied in random matrix theory since \cite{marchenko:pastur:1967}, where it is shown that the empirical distribution of the eigenvalues $\lambda_i(\bfS)$ converges weakly to the \MP~law. Subsequently, several  ground-breaking results such as the convergence of the largest eigenvalue $\lambda_1(\bf S)$ and the smallest eigenvalue $\lambda_p(\bfS)$ to the edges of  the \MP law \citep{BaiYin88a,tikhomirov:2015}, asymptotic normality of linear spectral statistics of $\bfS$ \citep{BS04},  or its edge universality towards the Tracy-Widom law \citep{johnstone:2001,Peche2012,PillaiYin2014} were established. Apart from the convergence of $\lambda_p(\bfS)$ all those results require a finite fourth moment $\E[X_{11}^4]$. In case of infinite fourth moments, the theory for the eigenvalues and eigenvectors of $\bfS$ is quite different from the aforementioned \MP~theory. A detailed account on the developments in the heavy-tailed case can be found in \cite{davis:heiny:mikosch:xie:2016,heiny:mikosch:2017:iid,basrak:heiny:jung:2020,auffinger:arous:peche:2009}.

For the sample correlation matrix $\bfR=\{\diag(\bfS_n)\}^{-1/2} \bfS_n \{\diag(\bfS_n)\}^{-1/2}$, the situation gets more complicated because of the specific nonlinear dependence structure caused by the normalization $\{\diag(\bfS_n)\}^{-1/2}$,  which makes the analysis of this random matrix quite challenging. As a consequence, the study of the high-dimensional sample correlation matrix is more recent and somewhat limited. Under finite fourth moments, many results about sample correlation matrices can be reduced to the covariance case via a comparison of their spectra (see \cite{heiny:2022}). Assuming the proportional regime, Lemma~2 in \cite{bai:yin:1993} asserts that $\E[X_{11}^4]<\infty$ is equivalent to $\norm{\diag(\bfS_n)-\bfI} \cas 0\,, n \to \infty$, where $\|\cdot \|$ is the spectral norm and $\bfI$ the identity matrix; see also \cite[Theorem 1.2]{heiny:2022} for a similar result in the dependent case. Therefore, under finite fourth moment the normalization $\{\diag(\bfS_n)\}^{-1/2}$ in \eqref{eq:defRY} can be replaced with $\bfI$ and consequently $\max_i |\lambda_i(\bfR)-\lambda_i(\bfS_n)| \le \norm{\bfR-\bfS_n}$ converges to zero almost surely as $\nto$. Using this comparison, Jiang \cite{jiang:2004} (see also \cite{elkaroui:2009, heiny:mikosch:2017:corr}) showed that the \MP~law  is still valid for the sample correlation matrix $\bfR$. The limiting eigenvalue distribution of $\bfR$ beyond the common fourth moment condition exhibits a phase transition at $\alpha=2$. We refer to \cite{bai:zhou:2008, doernemann:heiny:2025,  heiny:yao:2020}, where the proportional regime is treated under mild moment assumptions.
\smallskip

It turns out that a  modification of the above comparison trick can be used to obtain CLTs for linear spectral statistics of $\bfR$ from CLTs for linear spectral statistics of $\bfS$ \cite{Gao2017, yin:zheng:zou:2023, yin:li:tian:zheng:2022}. To the best of our knowledge, the first result under infinite fourth moment concerns the function $f=\log$ for which $\sum_{i=1}^p \log(\lambda_i(\bfR)) = \log \det \bfR$, the $\log$-determinant of the sample correlation matrix \cite{heiny:parolya:2024, li:logdet:2026}. For $\alpha>3$, it was proved that $\log \det \bfR$ satisfies a universal CLT in the proportional regime. Furthermore, Li et al.~\cite{li:pan:xie:wang:2024} provided CLTs for linear spectral statistics of $\bfR$ in the proportional regime when $\alpha>3$, while Jiang and Pham \cite{jiang:pham:2025uniformity} proved a CLT for the Bingham test for $\alpha\in(0,2)$. In all of those works, the asymptotic distribution of the linear spectral statistics is Gaussian. By allowing for arbitrary growth rates of the dimension $p$, we obtain both Gaussian and non-Gaussian limits in the present paper.
\smallskip

This work was motivated by the functions $f_k(x):=x^k$, $k\ge 2$, for which 
\begin{equation*}
    \sum_{i=1}^p f_k(\lambda_i(\bfR))=\tr(\bfR^k)\,.
\end{equation*}
(Note that the case $k=1$ is degenerate since $\tr(\bfR)=p$ is non-random.) CLTs for linear spectral statistics of $\bfR$ for more general functions $f$ can be obtained by approximating $f$ through polynomials, that is, linear combinations of $f_k$'s. For technical reasons and for the sake of clarity, we restrict ourselves to the case $k=2$.\footnote{An extension to general $k$ is a topic for future research.}

\subsection{Contributions} 
The novel contributions of this paper are outlined below. 
\begin{itemize}
\item We provide CLTs for $\tr(\bfR^2)$ under general growth rates of $p$ relative to $n$ and investigate the influence of the tail index $\alpha$. If $p\asymp n^{\delta}$ for some $\delta>0$, we determine a region for $(\alpha, \delta)$ where $\tr(\bfR^2)$ satisfies a CLT. For any pair $(\alpha, \delta)$ outside the closure of this region, we prove that moment convergence fails.
\item On the two regions' boundary, we identify the limiting fluctuations of $\tr(\bfR^2)$. We show that the boundary limit is moment-determinate and non-Gaussian.
\item At $\alpha=3$ (corresponding to the boundary of finite and infinite third moment), we discover a transition in the variance of $\tr(\bfR^2)$. As a consequence, if $\alpha>3$, no restriction on the growth of $p$ is required for the validity of the CLT. 
\item To the best of our knowledge, this work is the first that establishes a non-Gaussian limit for a linear spectral statistic of $\bfR$ in the case of infinite third moment $\E|X_{11}|^3=\infty$.
\end{itemize}

\subsection{Outline of main results}
The main results of this paper are a universal CLT for $\tr(\bfR^2)$ when $\alpha>3$, a non-universal CLT for $\tr(\bfR^2)$ when $\alpha\leq 3$, and the limiting distribution at the phase transition boundary. Consider a sequence $\sigma_n^2$  satisfying $\sigma_n^2\sim \Var(\tr(\bfR^2))$, as $\nto$, and define the critical dimension growth by 
\begin{equation}\label{eq:pboundary:intro}
    p_{\text{crit}} :=
    \begin{cases}
        n^{1/2} &\text{if}\ \alpha\in (0,2)\,, \\
        n^{(\alpha-1)/2} \frac{1}{L(n^{1/2})}  &\text{if}\ \alpha\in (2,3)\,.
    \end{cases}    
\end{equation}
The main findings of this paper concern a curious phase transition for $(\tr(\bfR^2)-\E[\tr(\bfR^2)])/\sigma_n$, which can be summarized (in slightly simplified fashion) as follows. As $\nto$, it holds
\begin{subequations} \label{eq:summaryresults}
    \begin{align}[left ={\dfrac{\tr(\bfR^2)-\E[\tr(\bfR^2)]}{\sigma_n} \cid  \empheqlbrace}]
    & \frac{\sum_{i=1}^\infty \xi_i-c_{\delta}^2}{\sqrt{2}\, c_\delta\, c_2(\alpha)} && \text{if}\ \alpha\in (0,3)\backslash\{2\} \text{ and } \frac{p}{p_{\text{crit}}} \to c_{\delta}, \label{poissontype}\\
    & N(0,1) && \text{if}\ \alpha\in (0,3)\backslash\{2\} \text{ and } \frac{p}{p_{\text{crit}}} \to \infty\,, \label{nonuniversal}\\
    & N(0,1)&& \text{if}\ \alpha>3\,, \label{universal}
    \end{align}
\end{subequations}
where $c_{\delta}\in (0,\infty)$ and $(\xi_i)_{i\ge 1}$ are the atoms of a Poisson point process on $(0,2]$ with intensity measure $\nulim$ defined in \eqref{eq:nu-unified}, and the constant $c_2(\alpha)$ is given in \eqref{eq:cm-def}. We will later refer to the dimension growth in \eqref{poissontype} as the phase transition boundary. In this case, the limiting distribution is of {\em Poisson process type} and centered (Theorem~\ref{thm:mainresult}). To the best of our knowledge, such a distribution is completely new in the context of linear spectral statistics of sample correlation or covariance matrices. Equation \eqref{nonuniversal} shows that for $\alpha\in(0,3)$ the dimension $p$ needs to grow faster than $p_{\text{crit}}$ in order to obtain a Gaussian limit. While the expectation $\E[\tr(\bfR^2)]$ is the same for any symmetric distribution of $X_{11}$ (see \eqref{eq:mean}), the variance sequence $\sigma_n^2$ asymptotically depends on the specific value $\alpha$ if $\alpha\in (0,3)$. Due to this dependence on $\alpha$, we call the limit result in \eqref{nonuniversal} a {\em non-universal CLT}. In contrast, for $\alpha>3$, the sequence $\sigma_n^2$ asymptotically does not depend on the value $\alpha$ and therefore we call \eqref{universal} a {\em universal CLT}. The universal and non-universal CLTs are extensions of the Master's thesis \cite{seo:2024} by the third author which was defended in June 2024. Figure~\ref{fig:firstsimulation} illustrates the different limits $(\tr(\bfR^2)-\E[\tr(\bfR^2)])/\sigma_n$ in \eqref{eq:summaryresults} for $n=8000$ and various choices of $\alpha$ and $p$. In the first row, we see that for $\alpha=1.1$ a standard Gaussian density ($p=500$) and the density of the Poisson process type limit ($p=90\approx p_{\text{crit}}$) are excellent fits to the histogram. The second row shows that for $\alpha=3.5$ the Gaussian limit holds for both values of $p$.

The proofs of \eqref{poissontype}--\eqref{universal} rely on cumulants and martingale theory, respectively. A crucial first step for both approaches is the decomposition 
\begin{equation}\label{eq:dec:intro}
    \tr(\bfR^2)-\E[\tr(\bfR^2)]=T_1+T_2
\end{equation}
into two uncorrelated and centered terms. This decomposition is optimal in the sense of capturing the precise phase transition boundary.

\subsection{Structure of the paper}
This paper is structured as follows. In Section~\ref{sec:preliminaries}, we derive the decomposition \eqref{eq:dec:intro} of the trace of the squared sample correlation matrix and study properties of $T_1$ and $T_2$. Based on this decomposition, Theorem~\ref{lem:completefourth} provides precise conditions for the convergence of the fourth moment of the standardized $\tr(\bfR^2)$ to the fourth moment of a standard normal variable. This leads us to the critical dimension growth $p_{\text{crit}}$ which separates the Gaussian CLT regime from a non-Gaussian limit regime. In Section~\ref{sec:main}, we present universal and non-universal CLTs for $\tr(\bfR^2)$ (Theorem~\ref{thm:clt2}) under very general growth rates on the dimension $p$ and the tail index $\alpha$. On the phase transition boundary we calculate the full limiting moments of $T_1$. Consequently, we find a non-Gaussian limiting distribution at the phase transition boundary. We further prove that this distribution is of Poisson process type as in \eqref{poissontype}. The results are then illustrated by means of a small simulation study in Section~\ref{sec:simulation:discussion}. Section~\ref{sec:mainproof} contains the proofs of our main results, while Section~\ref{sec:proofsmoment} is devoted to the proofs of the results of Section~\ref{sec:preliminaries}. Finally, the appendix consists of facts for sums of regularly varying random variables.

\begin{figure}[!htbp]
\begin{subfigure}{.4\textwidth}
  \centering
  \includegraphics[scale = 0.4]{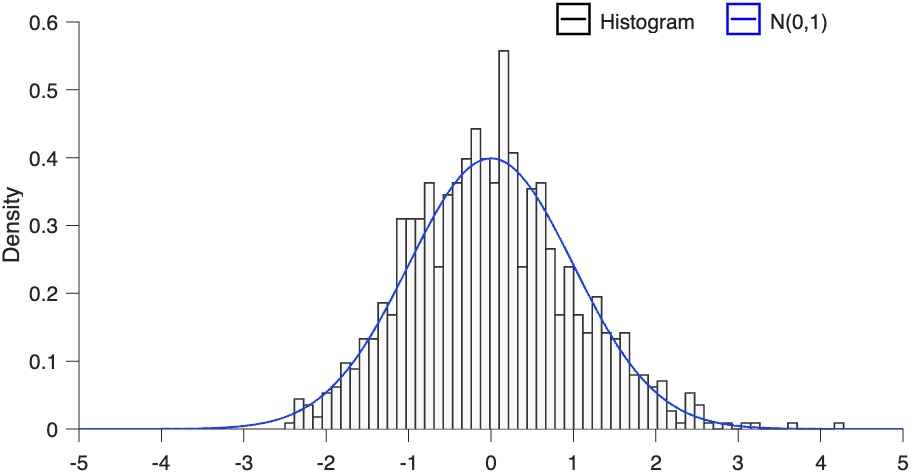}
  \caption{$\alpha = 1.1,\, p = 500 > p_{\text{crit}}$}
\end{subfigure}%
\begin{subfigure}{.4\textwidth}
  \centering
  \includegraphics[scale = 0.4]{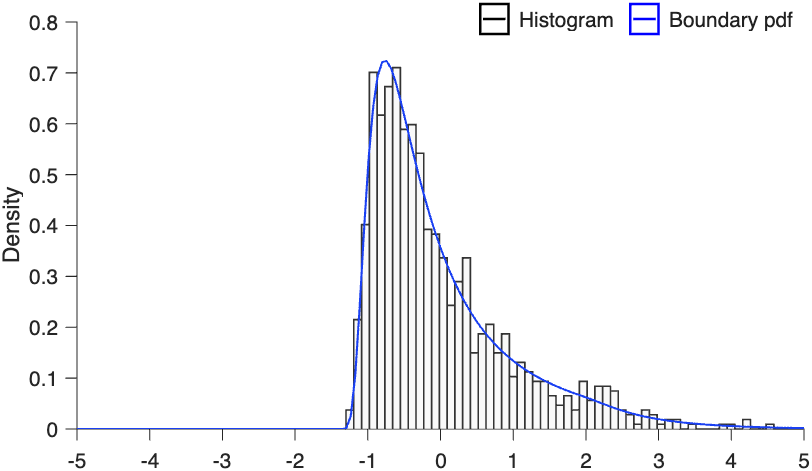}
  \caption{$\alpha = 1.1,\, p = 90\approx p_{\text{crit}}$}
\end{subfigure}
\begin{subfigure}{.4\textwidth}
  \centering
  \includegraphics[scale = 0.4]{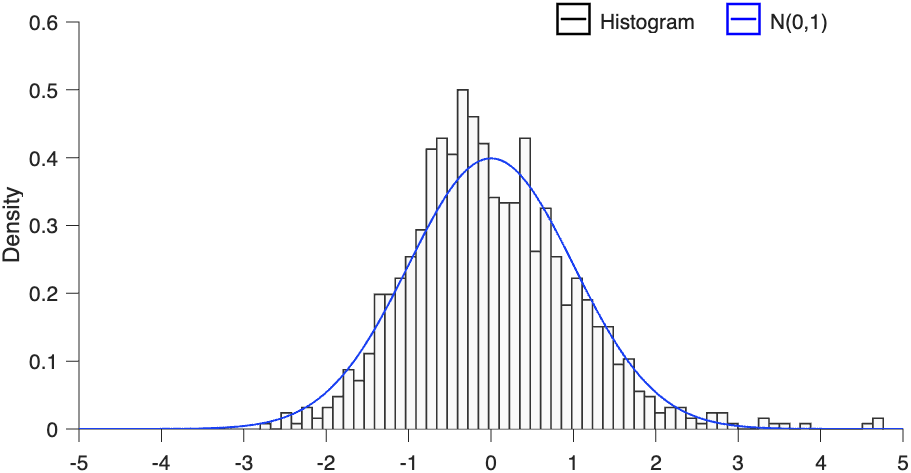}
  \caption{$\alpha = 3.5,\, p = 500 $}
\end{subfigure}%
\begin{subfigure}{.4\textwidth}
  \centering
  \includegraphics[scale = 0.4]{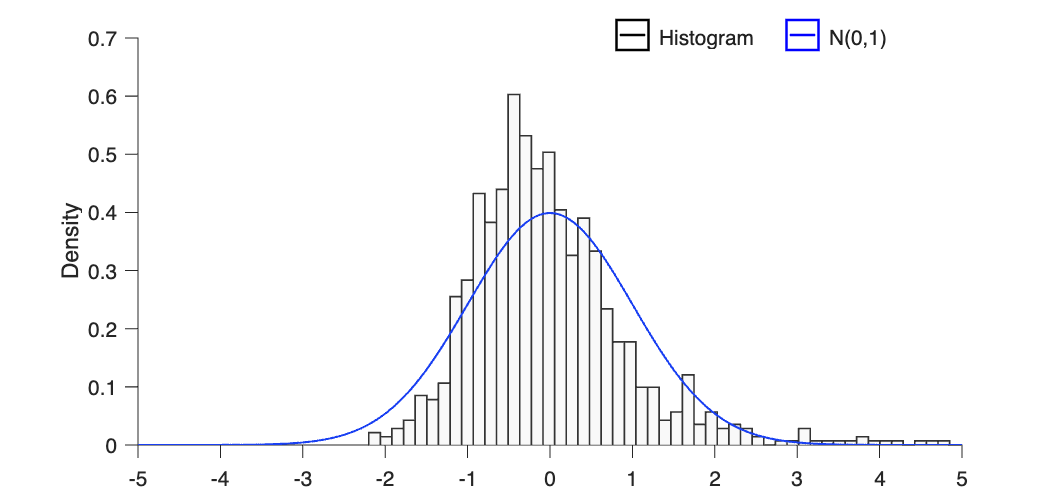}
  \caption{$\alpha = 3.5,\, p = 90$}
\end{subfigure}
\caption{Histograms of $(\tr(\bfR^2)-\mu_n)/\sigma_n$ for different values of $p$ and $\alpha$ with $n=8000$ and 1000 repetitions.}
\label{fig:firstsimulation}
\end{figure}

\subsection{Notation}
Convergence in distribution (resp.\ probability) is denoted by $\cid$ (resp.\ $\cip$), equality in distribution by $\eid$, and unless explicitly stated otherwise all limits are for $\nto$. For sequences $(a_n)_n$ and $(b_n)_n$ we write $a_n=O(b_n)$ if $a_n/b_n\leq C$ for some constant $C>0$ and every $n\in\N$, and $a_n=o(b_n)$ if $\lim_{n\to\infty} a_n/b_n=0$. Additionally, we use the notation $a_n\sim b_n$ if $\lim_{n\to\infty} a_n/b_n=1$, and $a_n = \omega(b_n)$ if $\lim_{n \to \infty} a_n / b_n = \infty$. We write $a_n \lesssim b_n$ if there exists a positive constant $C$ not depending on $n$ such that $a_n \le C\, b_n$ for sufficiently large $n$. The notation $a_n\slv b_n$ means that $a_n=b_n \ell(n)$ for some function $\ell$ that is slowly varying (at infinity). A function $\ell:(0,\infty)\to(0,\infty)$ is said to be slowly varying (at infinity) if $\lim_{x\to \infty} \ell(tx)/\ell(x)=1$ for any $t>0$.

\section{Preliminaries}\label{sec:preliminaries}
The main motivation of this work is to investigate the fluctuations of the trace of powers of the sample correlation matrix. That is, we aim to prove limit theorems for
\begin{equation}\label{eq:mk}
    \tr (\bfR^k) = \sum_{i_1,\ldots,i_k=1}^p  \sum_{t_1,\ldots,t_k=1}^n Y_{i_1t_1} Y_{i_2t_1} Y_{i_2t_2} \cdots Y_{i_kt_k} Y_{i_{k+1}t_k}, \qquad k\ge 2. 
\end{equation}
(Here the convention $i_{k+1}=i_1$ is used.) It is worth mentioning that $\tr(\bfR)=p$ since all diagonal elements of $\bfR$ are one. In this paper, we will focus on $\tr(\bfR^2)$, but our approach naturally extends to traces of higher powers of $\bfR$. 

Unless explicitly stated otherwise, the $X_{it}$ are iid and symmetric throughout this paper, which implies that the $Y_{it}$ are symmetric as well. The following properties of the matrix $\Y=(Y_{it})$ will be repeatedly used in this paper.
\begin{enumerate}
\item[(1)] By symmetry of the entry distribution we have for $s\le n$ that $\E [Y_{i1}^{m_1}\cdots Y_{is}^{m_s}]=0$ if at least one exponent $m_j\in\N$ is odd.
\item[(2)] $\Y$ has independent rows.
\item[(3)] By definition, $\sum_{t=1}^n Y^2_{it}=1$ for each row $i$.  
\end{enumerate}
The first step is to study $\tr(\bfR^2)$ for a wide range of distributions of $X_{11}$ and growth rates of the dimension $p$ relative to the sample size $n$. But before we delve deeper into this we need to introduce some important notation and results about the moments of $Y_{ij}$'s that will be crucial in our analysis.

\subsection{Mixed moments of the self-normalized entries}
For all positive integers $k_1,\ldots, k_r$, define 
\begin{equation*}
    \beta_{2k_1,\ldots, {2k_r}}:=\E[Y_{11}^{2k_1}Y_{12}^{2k_2} \cdots Y_{1r}^{2k_r}],    
\end{equation*}
where we recall the definition of $Y_{ij}$ from \eqref{def:R}. Since $\beta_{2k_1,\ldots, {2k_r}}=\beta_{2k_{\pi(1)},\ldots, 2k_{\pi(r)}}$ for any permutation $\pi$ on $\{1,\ldots,r\}$ we will write the indices in decreasing order, e.g., instead of $\beta_{2,4}$ we prefer writing $\beta_{4,2}$. In view of $\sum_{t=1}^n Y^2_{it}=1$, one has $\beta_2=n^{-1}$ which motivates the notation 
\begin{equation*}
    \tbeta_{2k_1,\ldots, {2k_r}}:=\E [ (Y_{11}^2-n^{-1})^{k_1} \cdots (Y_{1r}^2-n^{-1})^{k_r} ]\,
\end{equation*}
for the mixed centered moments. 

For positive integers $k_1,\ldots,k_r$, set $k=k_1+\cdots+k_r$ and $N_1=\#\{1\le i\le r: k_i=1\}$, and let $\Gamma(\cdot)$ denote the gamma function. Next, we introduce a set of positive $\alpha$-dependent constants which will be used throughout the paper:
\begin{equation*}
    c_{k_1,\ldots, k_r}(\alpha):=
    \begin{cases}
        \displaystyle
        \frac{\big(\alpha/2\big)^{r-1} \Gamma(r) \prod_{j=1}^r \Gamma(k_j-\alpha/2)}{ \big(\Gamma(1-\alpha/2)\big)^r\, \Gamma(k)}\,, & \alpha\in (0,2),\\
        \displaystyle
        \frac{(\alpha/2)^{r-N_1}\Gamma(N_1(1-\alpha/2)+ r\alpha/2) \, \prod_{i:k_i\ge 2} \Gamma(k_i-\alpha/2)}{\Gamma(k)}\,, & \alpha\in [2,4).
    \end{cases}
\end{equation*}
If $r=1$ and $k\ge 2$, this expression simplifies to
\begin{equation}\label{eq:cm-def}
    c_k(\alpha):=
    \begin{cases}
        \displaystyle
        \frac{\Gamma(k-\alpha/2)}{\Gamma(1-\alpha/2)\Gamma(k)}\,,& \alpha\in (0,2),\\
        \displaystyle
        \frac{\Gamma(1+\alpha/2)\Gamma(k-\alpha/2)}{\Gamma(k)}\,,& \alpha\in [2,4).
    \end{cases} 
\end{equation}
The following key lemmas reveal the asymptotic behavior of $\beta_{2k_1,\ldots, {2k_r}}$ and $\tbeta_{2k_1,\ldots, {2k_r}}$. 

\begin{lemma}\label{lem:allmoments}
Define the $Y_{ij}$'s as in \eqref{def:R} and let $L$ be a slowly varying function (at infinity).  
\begin{itemize}
\item[(a)] If $\alpha\in (0,2)$ and $\P(|X_{11}|>x)=x^{-\alpha} L(x)$ for $x>0$, then it holds
\begin{equation}\label{moment}
    \lim_{\nto} n^r \beta_{2k_1,\ldots, {2k_r}} = c_{k_1,\ldots, k_r}(\alpha).
\end{equation}
In particular, we have
\begin{equation}\label{highestmoment}
    \lim_{\nto} n \beta_{2k} = c_k(\alpha)\,,\qquad k\ge 2\,.
\end{equation}
\item[(b)] If $\alpha\in [2,4)$, $\E[X_{11}^2]=1$ and $\P(|X_{11}|>x)=x^{-\alpha} L(x)$ for $x>0$, then it holds
\begin{equation}\label{moment24}
    \lim_{\nto} \frac{n^{N_1(1-\alpha/2)+ r\alpha/2}}{L^{r-N_1}(n^{1/2})}  \beta_{2k_1,\ldots, {2k_r}} =   c_{k_1,\ldots, k_r}(\alpha)\,.
\end{equation}
In particular, we have
\begin{equation}\label{highestmoment24}
    \lim_{\nto}\frac{n^{\alpha/2}}{L(n^{1/2})}\beta_{2k} = c_k(\alpha)\,,\qquad k\ge 2\,.
\end{equation}
\item[(c)] Assume $\P(|X_{11}|>x)=x^{-\alpha} L(x)$. If $\{\alpha=2$ and $\E[X_{11}^2]=\infty\}$ or $\{\alpha=4, \E[X_{11}^2]=1$ and $\E[X_{11}^4]=\infty\}$, then \eqref{moment24} remains valid if we multiply its \lhs~ with some slowly varying function (that depends on $L$ and $k_1,\ldots,k_r$).
\item[(d)]If $\E[X_{11}^{2\max_i k_i}]<\infty$ and $\E[X_{11}^2]=1$, then it holds
\begin{equation*}
    \lim_{\nto} n^{k} \beta_{2k_1,\ldots, {2k_r}} = \prod_{i=1}^r \E[X_{11}^{2 k_i}]\,.
\end{equation*}
\end{itemize}
\end{lemma}

\begin{lemma}\label{lem:tbeta}
Let $\alpha\in (0,4)$ and assume that $\P(|X_{11}|>x)=x^{-\alpha} L(x)$ for $x>0$, where $L$ is a slowly varying function. If $\E[X_{11}^2]$ is finite, we additionally assume that $\E[X_{11}^2]=1$. Define the $Y_{ij}$'s as in \eqref{def:R} and consider integers $k_1,\ldots,k_r\ge 1$. Then it holds, as $\nto$, 
\begin{equation}\label{eq:tbeta}
    \tbeta_{2k_1,\ldots, {2k_r}} \left\{
        \begin{array}{ll}
            \sim \beta_{2k_1,\ldots, {2k_r}}\,, & \mbox{if } \min(k_1, \ldots,k_r)\ge 2\,, \\
            = O(\beta_{2k_1,\ldots, {2k_r}}) \,, & \mbox{if } \min(k_1, \ldots,k_r)=1\,.
        \end{array}\right.
\end{equation}
Moreover, for every $r\ge2$ we have
$\widetilde\beta_{2k_1,\ldots,2k_r}
=o\big(\widetilde\beta_{2(k_1+\cdots+k_r)}\big)$.
\end{lemma}

\subsection{Decomposition of $\tr(\bfR^2)$ into $T_1$ and $T_2$}
Now we are ready to calculate the mean and variance of $\tr(\bfR^2)$. From \eqref{eq:mk} it is simple to see that 
\begin{equation} \label{eq:derivmean} 
    \begin{aligned} 
    \tr(\bfR^2) 
    &= \sum_{i_1, i_2 = 1}^p \sum_{t_1, t_2 = 1}^n Y_{i_1 t_1} Y_{i_1 t_2} Y_{i_2 t_1} Y_{i_2 t_2} \\ 
    &= \sum_{i = 1}^p \sum_{t_1, t_2 = 1}^n Y_{i t_1}^2 Y_{i t_2}^2 + \sum_{\substack{i_1,i_2 = 1 \\ i_1 \neq i_2}}^p \sum_{t_1,t_2 = 1}^n Y_{i_1 t_1} Y_{i_1 t_2} Y_{i_2 t_1} Y_{i_2 t_2} \\ 
    &= p + \sum_{\substack{i_1,i_2 = 1 \\ i_1 \neq i_2}}^p \sum_{t = 1}^n Y_{i_1 t}^2 Y_{i_2 t}^2 + \sum_{\substack{i_1,i_2 = 1 \\ i_1 \neq i_2}}^p \sum_{\substack{t_1,t_2 = 1 \\ t_1 \neq t_2}}^n Y_{i_1 t_1} Y_{i_1 t_2} Y_{i_2 t_1} Y_{i_2 t_2}\,, 
    \end{aligned} 
\end{equation}
where the property $Y_{i1}^2 + \dots + Y_{in}^2 = 1$ was used for the last equality. Since $\E[Y_{it}^2]=1/n$, we deduce that
\begin{equation} \label{eq:mean}
    \mu_n:= \E[\tr(\bfR^2)] = p + \frac{p(p - 1)}{n}.
\end{equation}
An interesting observation is that the mean does not depend on the distribution of $X_{11}$. In view of the identity $\sum_{t = 1}^n \big(Y_{i_1 t}^2 Y_{i_2 t}^2 - \frac{1}{n^2}\big)
    =\sum_{t = 1}^n \big(Y_{i_1 t}^2 - \frac{1}{n} \big) \big(Y_{i_2 t}^2 - \frac{1}{n} \big)$,
we have derived  the nice decomposition 
\begin{equation}\label{eq:T_1+T_2_decomp}
    \tr(\bfR^2)-\E[\tr(\bfR^2)] = T_1 + T_2\,,
\end{equation}
where 
\begin{equation}\label{eq:def_T1_T2}
    \begin{aligned}
    T_1 &:=
    \sum_{\substack{i_1,i_2 = 1 \\ i_1 \neq i_2}}^p \sum_{t = 1}^n
    \Big(Y_{i_1 t}^2 - \frac{1}{n} \Big)
    \Big(Y_{i_2 t}^2 - \frac{1}{n} \Big),\\
    T_2 &:=
    \sum_{\substack{i_1,i_2 = 1 \\ i_1 \neq i_2}}^p
    \sum_{\substack{t_1,t_2 = 1 \\ t_1 \neq t_2}}^n
    Y_{i_1 t_1} Y_{i_1 t_2} Y_{i_2 t_1} Y_{i_2 t_2}.
    \end{aligned}
\end{equation}
are two sums of centered and uncorrelated random variables.
Using \eqref{eq:T_1+T_2_decomp}, the variance of $\tr(\bfR^2)$ is given by
\begin{equation}\label{eq:sdgf234}
    \Var(\tr(\bfR^2)) = \E[T_1^2] + \E[T_2^2],
\end{equation}
since $\E[T_1 T_2]=0$ as it only contains moments of odd powers of $Y_{it}$'s. The next lemma gives $\E[T_1^2]$ and $\E[T_2^2]$ in terms of $\beta_4=\E[Y_{11}^4]$. 
\begin{lemma}\label{lem:variance}
For any symmetric distribution of $X_{11}$ and $p=p_n\to \infty$,  it holds 
\begin{equation}
    \begin{aligned}
    \E[T_1^2] 
    &=
    \frac{2p(p-1)n^2}{n - 1} \Big(\beta_4 - \frac{1}{n^2}\Big)^2
    \sim 2p^2n \Big(\beta_4 - \frac{1}{n^2}\Big)^2\,,\\
    \E[T_2^2]
    &=   \frac{ 4p(p-1)n}{n-1}\Big(\frac{1}{n}-\beta_4\Big)^2 \sim 4 p^2 \Big(\frac{1}{n}-\beta_4\Big)^2\,, \qquad \nto\,.
    \end{aligned}
\end{equation}
\end{lemma}
From \eqref{eq:sdgf234} and Lemma~\ref{lem:variance} we deduce that 
\begin{equation}\label{eq:var} 
    \Var(\tr(\bfR^2))= 2np(p-1)\bigg(\frac{n}{n-1} \bigg(\beta_4 - \frac{1}{n^2}\bigg)^2 + \frac{2}{n-1}\bigg(\frac{1}{n}-\beta_4\bigg)^2\bigg).   
\end{equation}
\begin{remark}\label{rem:finitefourth}
From a theoretical point of view, it is more interesting to discuss and interpret the findings of this work for distributions with infinite fourth moments. To this end, we typically impose the regular variation assumption \eqref{eq:regvar} with index $\alpha\in(0,4)$. We would like to mention that the case of finite fourth moment is much simpler from a technical point of view since it only requires part (d) of Lemma~\ref{lem:allmoments}, whereas the regular variation setup with $\alpha\in(0,4)$ contains fascinating transitions since the orders of the $\beta$'s depend on $\alpha$ as showcased in parts (a) and (b) of Lemma~\ref{lem:allmoments}.
\end{remark}
Roughly speaking, the variance in \eqref{eq:var} is essentially determined by $\E[T_1^2]$ if $\E|X_{11}|^3= \infty$ and by $\E[T_2^2]$ otherwise. This means that the variance undergoes a transition at $\alpha=3$. To make this point more precise, we will start by analyzing the formulas in Lemma~\ref{lem:variance} for $\alpha \in (2,4)$. In this case, Lemma~\ref{lem:allmoments} asserts that $\beta_4\sim n^{-\alpha/2} L(n^{1/2}) c_2(\alpha)$, as $\nto$.
In combination with Lemma~\ref{lem:variance}, we deduce that,  as $\nto$, 
\begin{equation}\label{eq:dsf245d} 
    \Var(\tr(\bfR^2)) \sim 
    \begin{cases} 
        \E[T_1^2] \sim 2p^2 n \beta_4^2 &\quad \text{if}\ \alpha < 3, \\ 
        \E[T_2^2] \sim 4p^2 n^{-2} &\quad \text{if}\ \alpha > 3. 
    \end{cases} 
\end{equation}

The next lemma provides more detailed information about the asymptotic behavior of $\Var(\tr(\bfR^2))$. 
It shows that
\begin{equation}\label{eq:defsigma}
    \sigma_n^2:=2p^2 n \big(\beta_4^2+2 n^{-3}\big)
\end{equation} 
 is asymptotically equivalent to $\Var(\tr(\bfR^2))$, and thus efficiently captures the effect of $p,n$ and the distribution of $X_{11}$. 
\begin{lemma}\label{lem:asymp_variance}
For any symmetric distribution of $X_{11}$ and $p=p_n\to \infty$,  it holds 
\begin{equation*}
    \Var(\tr(\bfR^2)) \sim \sigma_n^2\,, \qquad \nto\,.
\end{equation*} 
Moreover, $\sigma_n^2$ satisfies 
\begin{equation*}
    \sigma_n^2 \sim 
    \begin{cases}
        4p^2 n^{-2} &\text{if}\ \E[X_{11}^4]<\infty \text{ or } \alpha \in (3, 4),\\
        2p^2n^{-2}(L^2(n^{1/2})c_2^2(\alpha) + 2) &\text{if}\ \alpha = 3, \\
        2p^2n^{1-\alpha} L^2(n^{1/2}) c_2^2(\alpha) &\text{if}\ \alpha \in [2, 3)\ \text{and}\ \E[X_{11}^2] = 1, \\
        2p^2 n^{-1} (1 - \alpha / 2)^2 &\text{if}\ \alpha \in (0, 2),
    \end{cases}
\end{equation*}  
where $c_2(\alpha) =\Gamma(1+ \alpha/2) \,\Gamma(2 - \alpha/2)$.
\end{lemma} 
In view of \eqref{eq:dsf245d}, it follows from Markov's inequality that
\begin{equation}\label{remark:var_asymp}
  \frac{\tr(\bfR^2) -\E[\tr(\bfR^2)]}{\sigma_n} = 
    \begin{cases}
        T_1/\sigma_n +o_{\P}(1) &\quad \text{if}\ \alpha < 3,  \\
        T_2/\sigma_n +o_{\P}(1) &\quad \text{if}\ \alpha > 3.
    \end{cases}
\end{equation} 

\subsection{Fourth moment analysis across the whole regime}
Throughout this subsection, we assume that the distribution of $X_{11}$ is symmetric and regularly varying with index $\alpha\in(0,4)$, unless explicitly stated otherwise.
In order to get a first idea for  which combinations of $\alpha,p$ and $n$ we might have a CLT for $(\tr(\bfR^2)-\E[\tr(\bfR^2)])/\sigma_n \cid N(0,1)$, we compute the fourth moment of these random variables and check when they converge to $3$, the fourth moment of a standard normal variable. 

In view of \eqref{remark:var_asymp}, it suffices to study $T_1$ for $\alpha\le 3$ and $T_2$ for $\alpha\ge 3$. Careful combinatorial considerations yield the following important technical result.
\begin{lemma}\label{lem:fourth_mom_asymp}
For $\alpha \in (0, 4)$ it holds 
\begin{equation*}
    \E[T_1^4] \sim 12 \beta_4^4 n^2 p^4 + 8 \beta_8^2 n p^2\,,\qquad \nto\,.
\end{equation*}
For  $\alpha \in [3, 4)$  it holds 
\begin{equation*}
    \E[T_2^4]\sim 48 \,n^{-4} p^4\,,\qquad \nto\,.
\end{equation*}
\end{lemma}
In order to capture the interplay between $p,n$ and $\alpha$, we introduce the function $\delta^*$,
\begin{equation}
    \delta^*(\alpha):=
    \begin{cases}
        1/2  &\quad\text{if}\ \alpha \in (0,2), \\
        (\alpha-1) / 2  &\quad\text{if}\ \alpha \in [2,3], \label{eq:delta_star} \\
        (5 - \alpha)/2 &\quad \text{if}\ \alpha \in (3, 4). \\ 
    \end{cases}
\end{equation}
A condition of the form $p = \omega(n^\delta)$ for some $\delta > \delta^*(\alpha)$ will turn out to play an important role in convergence of the fourth moment and in the CLT for $\tr(\bfR^2)$.
The next theorem sheds light on the convergence of the fourth moment.
\begin{theorem}\label{lem:completefourth}
    For $\alpha \in (0, 3) \cup (3, 4)$ and assuming that $p = \omega(n^\delta)$ for some $\delta > \delta^*(\alpha)$, we have
\begin{equation}\label{eq:fourth_mom_equal_3}
   \lim_{\nto} \frac{\E\big[(\tr(\bfR^2)-\E[\tr(\bfR^2)])^4\big]}{\Var(\tr(\bfR^2))^2} = 3\,.
\end{equation} 
 Moreover, if  $p = o(n^\delta)$ for some $\delta < \delta^*(\alpha)$, then it holds 
\begin{equation*}
    \lim_{\nto} \frac{\E\big[(\tr(\bfR^2)-\E[\tr(\bfR^2)])^4\big]}{\Var(\tr(\bfR^2))^2}=\infty\,.
\end{equation*}
\end{theorem}

\begin{figure}
    \centering
    \includegraphics[width=0.7\linewidth]{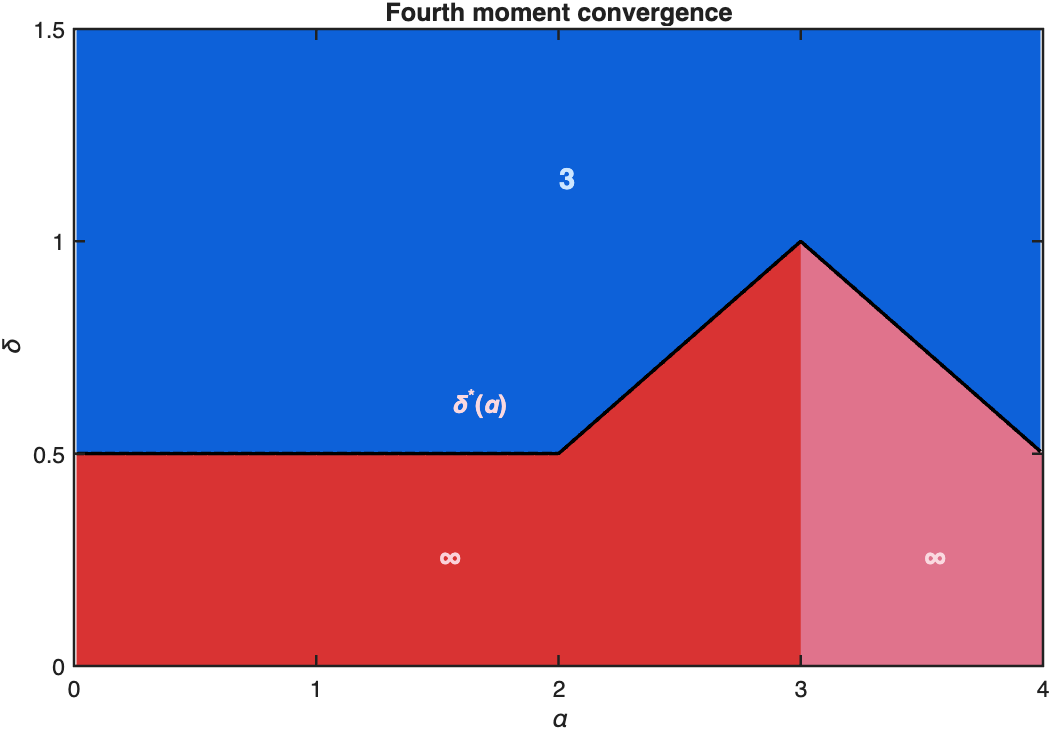}
    \caption{Phase diagram for the convergence of the fourth moment of $(\tr(\bfR^2)-\E[\tr(\bfR^2)])/\sigma_n$ in the $(\alpha,\delta)$-plane.   The fourth moment converges to $3$ in the blue region, whereas it diverges to $\infty$ in the red and pink regions. The boundary between those regions is given by the function $\delta^*$ defined in \eqref{eq:delta_star}.}
    \label{fig:4moment}
\end{figure}

Figure~\ref{fig:4moment} illustrates the results of Theorem~\ref{lem:completefourth}, which reveals a transition in the convergence of moments at $p=n^{\delta^*(\alpha)}$. The limit in the case $p=n^{\delta^*(\alpha)}$ depends in a delicate way on certain slowly varying functions\footnote{In our proofs, we often use the Potter bounds which guarantee that $\limn \ell(n)/n^{\vep}\to 0$ and $\limn \ell(n) n^{\vep}\to \infty$ for any slowly varying function $\ell$ and $\vep>0$ (see for instance \cite{bingham:goldie:teugels:1987}).} that can be expressed in terms of the function $L$ in $\P(|X_{11}|>x)=x^{-\alpha} L(x)$. It should be pointed out that the requirement $\delta>\delta^*(\alpha)$ for \eqref{eq:fourth_mom_equal_3} is specifically tailored to the convergence of the fourth moment. In general, for higher moments a slightly different condition will be required.  The next result shows that we can obtain a larger region for $p$ by only considering the leading term (in the sense of convergence in distribution) of the decomposition $T_1+T_2$. Recall that $T_1/\sigma_n =o_{\P}(1)$ if $\alpha>3$, and $T_2/\sigma_n =o_{\P}(1)$ if $\alpha<3$. Interestingly, in the case $\alpha \in (3,4)$, the crucial condition $\delta > \delta^*(\alpha)$ can be dropped if one focuses on the leading term only, as the following result shows.
\begin{proposition}\label{lem:E[T_i^4]/E[T_i^2]^2}
If $\alpha \in (0,3]$, then  
\begin{equation*}
    \lim_{\nto} \frac{\E[T_1^4]}{\E[T_1^2]^2} =
    \begin{cases}
        3 &\text{if}\ p = \omega(n^\delta) \text{ for some } \delta > \delta^*(\alpha)\,, \\
       \infty  &\text{if}\ p = o(n^\delta) \text{ for some } \delta < \delta^*(\alpha)\,,
    \end{cases}    
\end{equation*}
with $\delta^*(\alpha)$ defined in \eqref{eq:delta_star}. If $\alpha \in [3, 4)$, we have
\begin{equation}\label{eq:momt2}
    \lim_{\nto} \frac{\E[T_2^4]}{\E[T_2^2]^2} =3\,.
\end{equation}
\end{proposition}
Following the lines of the proof of Proposition~\ref{lem:E[T_i^4]/E[T_i^2]^2}, one can show that \eqref{eq:momt2} remains valid if the regular variation assumption is replaced by $\E[X_{11}^4]<\infty$. The difference between the conditions on $\delta$ in Theorem~\ref{lem:completefourth} and Proposition~\ref{lem:E[T_i^4]/E[T_i^2]^2}, respectively, is due to the fact that the $o_{\P}(1)$-term $T_1/\sigma_n$ might have a diverging fourth moment for $\alpha \in (3,4)$ which the following lemma shows.
\begin{lemma}\label{lem:E(T_i^4)/var^2_0}
For $\alpha \in (3, 4)$ we have
\begin{equation*}
    \lim_{\nto} \frac{\E[T_1^4] }{\Var(\tr(\bfR^2))^2}=
    \begin{cases}
        0 &\text{if}\ p = \omega(n^\delta) \text{ for some } \delta > \delta^*(\alpha)\,, \\
       \infty  &\text{if}\ p = o(n^\delta) \text{ for some } \delta < \delta^*(\alpha)\,.
    \end{cases}
\end{equation*}
For $\alpha \in (0, 3)$ it holds
\begin{equation*}
    \lim_{\nto} \frac{\E[T_2^4] }{\Var(\tr(\bfR^2))^2}=0.   
\end{equation*}
\end{lemma}

\subsection{Characterization of the phase transition boundary}
Combining Proposition~\ref{lem:E[T_i^4]/E[T_i^2]^2} and Lemma~\ref{lem:E(T_i^4)/var^2_0} yields that the pink region in Figure~\ref{fig:4moment} would turn blue if only the leading order term $T_2/\sigma_n$ was considered. This means that the phase transition boundary for the fourth moment is essentially described by the function $\delta^*(\alpha)$ for $\alpha\in(0,3)$. Recalling \eqref{remark:var_asymp}, it suffices to consider 
\begin{equation*}
    \Tone:=T_1/\sqrt{\E[T_1^2]}\,.
\end{equation*}
By Proposition~\ref{lem:E[T_i^4]/E[T_i^2]^2}, we know that for $\alpha\in(0,3)$ the quantity ${\E[\Tone^4]}$ either converges to $3$ or diverges to infinity if suitable conditions for the growth of $p$ are satisfied. It turns out that the setting where $p$ is neither $\omega(n^\delta)$ for some $\delta > \delta^*(\alpha)$ nor $o(n^\delta)$ for some $\delta < \delta^*(\alpha)$ is more subtle and the possible limits of ${\E[\Tone^4]}$ are not only $3$ or $\infty$, but also any constant greater than $3$. An example where all three limits can occur is $p=n^{\delta^*(\alpha)}\ell(n)$ with different slowly varying functions $\ell$. 

We therefore define the {\it phase transition boundary} as the critical regime, where
\begin{equation}\label{eq:boundary-def}
    \lim_{n\to\infty}\E[\Tone^4] = 3+C_\alpha \qquad\text{for some } C_\alpha\in(0,\infty).
\end{equation}
The main novelty of this paper is the finding that in this regime the fluctuations of $\tr(\bfR^2)$ become of Poisson process type as we will see in Section~\ref{sec:poissonlimit}.
An equivalent characterization of the phase transition boundary is through a precise growth condition on $p$ which is more intuitive (but slightly less compact) than \eqref{eq:boundary-def}.
The following lemma makes this precise. For technical reasons we will not consider the case $\alpha=2$.
\begin{lemma}\label{lem:boundary-scaling-equivalence}
Let $\alpha\in(0,3)\backslash \{2\}$. Then the boundary condition \eqref{eq:boundary-def} holds if and only if the following scaling assumptions are satisfied.
\begin{enumerate}
\item[$(i)$] If $\alpha\in(0,2)$, then \eqref{eq:boundary-def} is equivalent to the existence of a constant $c_\delta\in(0,\infty)$ such that
\begin{equation*}
    \lim_{n\to\infty}\frac{p}{n^{\delta^*(\alpha)}}
    = \lim_{n\to\infty}\frac{p}{n^{1/2}}
    = c_\delta.
\end{equation*}
The constant $c_{\delta}$ is linked to $C_{\alpha}$ through
\begin{equation*}
    C_\alpha
    = \frac{2}{c_\delta^2}\,\frac{c_4(\alpha)^2}{c_2(\alpha)^4}
    = \frac{2}{c_\delta^2} \bigg[\frac{\bigl(3-\frac{\alpha}{2}\bigr)\bigl(2-\frac{\alpha}{2}\bigr)}{6\bigl(1-\frac{\alpha}{2}\bigr)}\bigg]^2.
\end{equation*}
\item[$(ii)$] If $\alpha\in(2,3)$, then \eqref{eq:boundary-def} is equivalent to the existence of a constant $c_\delta\in(0,\infty)$ such that
\begin{equation*}
    \lim_{n\to\infty}\frac{p\,L(n^{1/2})}{n^{\delta^*(\alpha)}}
    =\lim_{n\to\infty}\frac{p\,L(n^{1/2})}{n^{(\alpha-1)/2}}
    =c_\delta.
\end{equation*}
The constant $c_{\delta}$ is linked to $C_{\alpha}$ through
\begin{equation*}
    C_\alpha
    =\frac{2}{c_\delta^2}\,\frac{c_4(\alpha)^2}{c_2(\alpha)^4}
    =\frac{2}{c_\delta^2}\bigg[\frac{\bigl(3-\frac{\alpha}{2}\bigr)\bigl(2-\frac{\alpha}{2}\bigr)}
    {6 \Gamma(1+\frac{\alpha}{2})\Gamma(2-\frac{\alpha}{2})}\bigg]^2.
\end{equation*}
\end{enumerate}
\end{lemma}
We conclude from Lemma~\ref{lem:boundary-scaling-equivalence} that the phase transition boundary \eqref{eq:boundary-def} is equivalent to 
\begin{equation*}
    p\sim
    \begin{cases}
        c_\delta \,n^{\delta^*(\alpha)} &\text{if}\ \alpha\in (0,2)\,, \\
       c_\delta \,n^{\delta^*(\alpha)} \frac{1}{L(n^{1/2})}  &\text{if}\ \alpha\in (2,3)\,,
    \end{cases}    
\qquad \nto\,,
\end{equation*}
for some positive constant $c_\delta$. In particular, this reveals that for $\alpha\in (0,2)$ the phase transition boundary is not affected by the slowly varying function $L$, whereas $L$ is relevant for $\alpha\in (2,3)$. 

\section{Main results}\label{sec:main}\setcounter{equation}{0}
Now we are ready to present the main results of this paper: a universal CLT for $\tr(\bfR^2)$ when $\alpha>3$, a non-universal CLT for $\tr(\bfR^2)$ when $\alpha\leq 3$, and the limiting distribution at the phase transition boundary. Again, we assume that the distribution of $X_{11}$ is symmetric and regularly varying with index $\alpha\in(0,4)$. From \eqref{eq:mean} and \eqref{eq:defsigma} recall that
\begin{equation*}
    \mu_n=  p + \frac{p(p - 1)}{n} \quad  \text{ and } \quad \sigma_n^2=2p^2 n \big(\beta_4^2+2 n^{-3}\big)\,.
\end{equation*} 

\subsection{Universal and non-universal CLTs}
\begin{theorem}\label{thm:clt2}
Let $\alpha \in (0,4)$ and $p=p_n\to\infty$. If $\alpha \in(0,3]$, additionally assume $p = \omega(n^\delta)$ for some $\delta > \delta^*(\alpha)$ with $\delta^*(\alpha)$ defined in \eqref{eq:delta_star}. Then, as $n\to\infty$, $\tr(\bfR^2)$ satisfies 
\begin{equation} \label{clt}
    \frac{\tr(\bfR^2)-\mu_n}{\sigma_n} \cid N(0,1)\,.   
\end{equation}
\end{theorem}
Theorem~\ref{thm:clt2} reveals the dependence on $\alpha$ in the variance but also that we need a minimal growth rate for $p$ in the case $\alpha\in (0,3)$. This is solely because of $T_1$ giving rise to high order terms in the martingale differences in Lemmas~\ref{lem:martingale_E[M_j,k^4]_to_0} and~\ref{lem:M_jk^2_cip_1} that are essential in the proof of Theorem~\ref{thm:clt2} which can be found in Section~\ref{sec:mainproof}. What is interesting is that we acquire weaker conditions for convergence in distribution to the standard normal distribution than for convergence of the centralized fourth moment divided by the squared variance; see Theorem~\ref{lem:completefourth} for details.  The extra condition on $p$ when $\alpha \in [3, 4)$ that is present in Theorem~\ref{lem:completefourth} but not in Theorem~\ref{thm:clt2} is due to the fact that we do not need to consider $T_1$ when $\alpha > 3$ in the martingale differences, while $T_1$ needs to be considered in the fourth moment calculations. More specifically, we need convergence to zero for $\E[\Tone^4]$ when $\alpha \in [3, 4)$ as stated in Lemma~\ref{lem:E(T_i^4)/var^2_0} for \eqref{eq:fourth_mom_equal_3} to hold.

Since $\sigma_n^2\sim4p^2n^{-2}$ for $\alpha\in(3,4)$, we obtain the following corollary of Theorem~\ref{thm:clt2}. 
\begin{corollary}\label{thm:clt1}
Let $\alpha\in(3,4)$ and $p=p_n\to\infty$, then as $n\to\infty$ we have 
\begin{equation} 
    \frac{\tr(\bfR^2)-\mu_n}{2pn^{-1}} \cid N(0,1)\,.   
\end{equation}
\end{corollary}
This means that for $\alpha>3$ the normalizing sequences in the CLT are universal and $p$ may tend to infinity at arbitrary speeds. The situation is quite different for $\alpha \le 3$, when both the variance $\sigma_n^2 \slv p^2n^{1-\max\{\alpha,2\}}$ (Lemma~\ref{lem:asymp_variance}) and the minimal growth of $p$  depend on $\alpha$. Due to this dependence on $\alpha$, we say that the CLT in \eqref{clt} is non-universal for $\alpha\le 3$.

Furthermore, a careful inspection of the proof of Theorem~\ref{thm:clt2} in the case $\alpha\in (3,4)$ yields the following result.
\begin{theorem}\label{thm:clt3}
Let $X_{11}$ follow a symmetric distribution (not necessarily slowly varying) with finite fourth moment and assume $p=p_n\to\infty$. Then it holds
\begin{equation*}
    \frac{\tr(\bfR^2)-\mu_n}{2pn^{-1}} \cid N(0,1)\,, \qquad \nto\,.  
\end{equation*}
\end{theorem}
With slightly more effort it is possible to remove the symmetry assumption on $X_{11}$ in Theorem~\ref{thm:clt3} at the cost of imposing the condition $\E[|X_{11}|^k]<\infty$ for a suitably large $k\ge 4$ that might depend on the relationship between $p$ and $n$. However, this result will not be pursued in this work.
\begin{figure}
    \centering
    \includegraphics[width=1\linewidth]{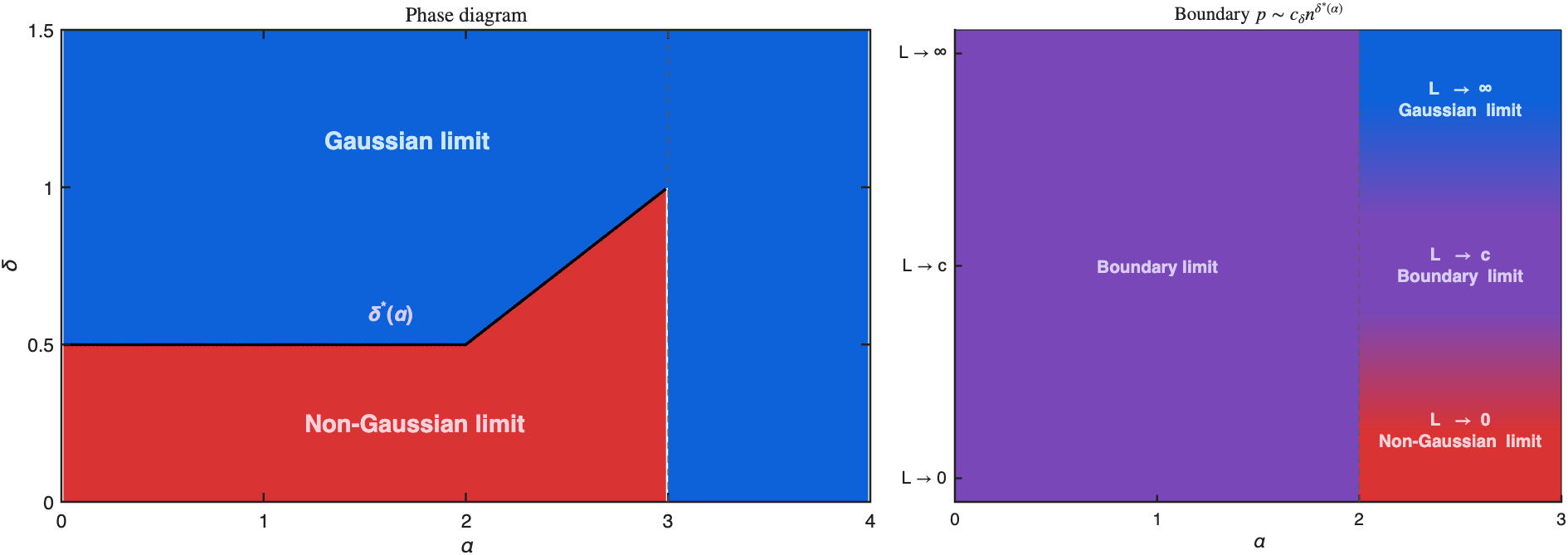}
    \caption{Phase transition and boundary behavior, where boundary limit refers to the Poisson process type limit in Theorem~\ref{thm:mainresult}.}
    \label{fig:phasetransition}
\end{figure}

\subsection{Phase transition: Poisson process limit}\label{sec:poissonlimit}
The non-universal CLT for $\alpha< 3$ requires that the dimension $p$ grows faster than on the phase transition boundary given in \eqref{eq:pboundary}. This naturally leads us to investigate the behavior on the phase transition boundary in greater detail. 

Throughout this subsection, we assume that $X_{11}$ is regularly varying with index $\alpha\in (0,3)\backslash \{2\}$ and we will work on the phase transition boundary, that is,
\begin{equation}\label{eq:pboundary}
    p\sim
    \begin{cases}
        c_\delta \,n^{\delta^*(\alpha)} &\text{if}\ \alpha\in (0,2)\,, \\
       c_\delta \,n^{\delta^*(\alpha)} \frac{1}{L(n^{1/2})}  &\text{if}\ \alpha\in (2,3)\,,
    \end{cases}    
\qquad \nto\,, \text{ for some constant } c_{\delta}\in (0,\infty)\,.
\end{equation}
By Lemmas~\ref{lem:variance} and~\ref{lem:asymp_variance}, we have $\E[T_1^2]\sim \sigma_n^2 \to 2 c_\delta^2 \,c_2(\alpha)^2$ and therefore, $T_1$ and $\tr(\bfR^2)-\mu_n$ have the same limiting distribution.

The next theorem provides the asymptotics of all moments of $T_1$ at the phase transition boundary. These moment formulas will serve as the basis for the cumulant description of the boundary limit law given afterwards.
\begin{theorem}
\label{thm:leading-piecewise}
Let $\alpha\in(0,3)\backslash \{2\}$ and assume that $p$ satisfies \eqref{eq:pboundary}. Then it holds
\begin{equation}\label{eq:ms-def}
    M_s:=\lim_{n\to\infty}\E[T_1^s]=
    \sum_{v=1}^{\lfloor s/2 \rfloor}
    c_{\delta}^{2v} \, \frac{2^{\,s-v}}{v!}\,
    \sum_{\substack{m_1,\ldots,m_v\ge2\\ m_1+\cdots +m_v=s}}
    \frac{s!}{m_1!\cdots m_v!}\,
    \prod_{\ell=1}^v c_{m_\ell}(\alpha)^2,
    \quad s\ge 1,
\end{equation}
where the coefficients $c_m(\alpha)$ are defined in \eqref{eq:cm-def}.
\end{theorem}
We also have the following result.
\begin{lemma}
\label{lem:carleman-nongauss}
For $\alpha\in (0,3)$, the moment sequence $(M_s)_{s\ge 1}$ uniquely characterizes a probability law.
\end{lemma}
As a consequence of Theorem~\ref{thm:leading-piecewise}, Lemma~\ref{lem:carleman-nongauss} and the method of moments, we obtain that the law of $T_1$ weakly converges to some law $\etalim$.
One can check that
\begin{equation*}
    \frac{M_4}{(M_2-M_1^2)^2}= 3+2\bigg[\frac{c_4(\alpha)}{c_2^2(\alpha)c_{\delta} }\bigg]^2 >3\,,
\end{equation*}
which implies that $\etalim$ has to be non-Gaussian. Our next goal is to further investigate this non-Gaussian law $\etalim$.

To this end, we need the $s$-th complete exponential Bell polynomial $B_s$, whose ordered-composition expansion is given by \cite{Comtet1974}
\begin{equation}\label{eq:bellpolynomial}
    B_s(x_1,\ldots,x_s):=\sum_{v=1}^{s}
    \frac{1}{v!}\,
    \sum_{\substack{m_1,\ldots,m_v\ge 1\\ m_1+\cdots +m_v=s}}
    \frac{s!}{m_1!\cdots m_v!}\,
    \prod_{\ell=1}^v x_{m_\ell},
    \qquad s\ge1.
\end{equation}

\begin{theorem}\label{thm:boundary-cumulant}
The moments $(M_s)_{s\ge 1}$ in \eqref{eq:ms-def} admit the representation
\begin{equation}\label{eq:ms-is-bell}
    M_s = B_s\big(\kappa_{1}^{(\alpha,c_\delta)},\kappa_{2}^{(\alpha,c_\delta)},\ldots,\kappa_{s}^{(\alpha,c_\delta)}\big), \qquad s\ge 1,
\end{equation}
where $\kappa_{1}^{(\alpha,c_\delta)}:=0$ and $\kappa_{m}^{(\alpha,c_\delta)}:=2^{m-1}\,c_{\delta}^{2}\,c_m(\alpha)^2$ for $m\ge2$. Consequently, the cumulants of $\etalim$ are given by $\big(\kappa_{s}^{(\alpha,c_\delta)}\big)_{s\ge 1}$ and the cumulant generating function of  $\etalim$ is
\begin{equation}\label{eq:Kalpha-closed-thm}
    K_{\alpha,c_\delta}(t):= \sum_{m\ge 1}\kappa_m^{(\alpha,c_\delta)}\frac{t^m}{m!}=
        \begin{cases}
            c_\delta^2\E\!\left[\frac{\e^{2tW}-1-2tW}{2W}\right] &\text{if}\ \alpha\in (0,2), \\
            c_\delta^2c_2(\alpha)^2\E\!\left[\frac{\e^{2tW'}-1-2tW'}{2(W')^2}\right]  &\text{if}\ \alpha\in [2,3),
        \end{cases}    
    \quad t\in\R\,.
\end{equation}
Here $W$ is the product of two independent $\mathrm{Beta}(1-\alpha/2,\alpha/2)$ variables for $\alpha\in (0,2)$, and $W'$ is the product of two independent $\mathrm{Beta}(2-\alpha/2,\alpha/2)$ variables for $\alpha\in [2,3)$.
\end{theorem}

Taking the limit $\alpha\downarrow0$ we obtain a centered Poisson distribution as the following corollary shows.
\begin{corollary}\label{cor:poisson}
Let $Z_{\alpha,c_\delta}$ denote a random variable with law $ \etalim$. Then it holds 
\begin{equation*}
    Z_{\alpha,c_\delta} \cid 2N_{c_{\delta}}-c_\delta^{2}, \qquad \text{ as } \alpha\downarrow 0\,,
\end{equation*}
where $N_{c_{\delta}}\sim \mathrm{Poisson}(c_\delta^{2}/2)$.
\end{corollary}

Let $\mu$ and $\mu'$ denote the laws of $2W$ and $2W'$, respectively. Define a measure $\nulim$ on $(0,2]$ by
\begin{equation}\label{eq:nu-unified}
    \nulim(dx)=
        \begin{cases}
            \displaystyle
            \frac{c_\delta^2}{x}\,\mu(dx),& \alpha\in(0,2),\\[12pt]
            \displaystyle
            \frac{2c_\delta^2c_2(\alpha)^2}{x^2}\,\mu'(dx),& \alpha\in[2,3).
        \end{cases}
\end{equation}
\begin{remark}
It is easy to see that $\nulim$ satisfies the L\'evy integrability condition
\begin{equation*}
    \int_{(0,2]}(1\wedge x^2)\nulim(dx) \le \int_{(0,2]}x^2\nulim(dx) = 2c_\delta^2c_2(\alpha)^2<\infty, \qquad \alpha\in(0,3),
\end{equation*}
and thus, $\nulim$  is a L\'evy measure.
\end{remark}

The following theorem constructs a compensated Poisson integral with law $\etalim$.
\begin{theorem}\label{thm:Poisson point process representation}
Let $N_{\alpha,c_{\delta}}=\sum_{i=1}^\infty \delta_{\xi_i}$ be a Poisson point process on $(0,2]$ with intensity measure
$\nulim$ and atoms $(\xi_i)_{i\ge 1}$.
Then the random variable 
\begin{equation}\label{eq:etalimrep}
    \sum_{i=1}^\infty \xi_i-c_{\delta}^2
\end{equation}
has law $\etalim$. Moreover, this random variable can be written as 
\begin{equation*}
    \sum_{i=1}^\infty \xi_i-c_{\delta}^2=\int_{(0,2]}x\,\widehat N_{\alpha,c_{\delta}}(dx)\,,
\end{equation*}
where $\widehat{N}_{\alpha,c_{\delta}}:=N_{\alpha,c_{\delta}}-\nulim$ is a compensated Poisson process.
\end{theorem}

We are now ready to state the main result of this subsection which is the Poisson process type limit for the trace of the squared sample correlation matrix on the phase transition boundary. It is an immediate consequence of Theorem~\ref{thm:leading-piecewise}, Lemma~\ref{lem:carleman-nongauss} and Theorem~\ref{thm:Poisson point process representation}.
\begin{theorem}[Poisson process limit on the phase transition boundary]\label{thm:mainresult}
Let $\alpha\in (0,3)\backslash \{2\}$ and assume that the dimension $p$ satisfies \eqref{eq:pboundary}. Then it holds
\begin{equation}\label{eq:trwithsigma}
    \frac{\tr(\bfR^2)-\mu_n}{\sigma_n} \cid \frac{\sum_{i=1}^\infty \xi_i-c_{\delta}^2}{\sqrt{2}\, c_\delta\, c_2(\alpha)}\,, \qquad \nto\,,
\end{equation}
where $(\xi_i)_{i\ge 1}$ are the atoms of a Poisson point process on $(0,2]$ with intensity measure
$\nulim$ defined in \eqref{eq:nu-unified}.
\end{theorem}
Note that in \eqref{eq:trwithsigma} we chose to divide by $\sigma_n$ for easier comparison with the universal and non-universal CLTs. It is worth mentioning that \eqref{eq:trwithsigma} is equivalent to 
\begin{equation*}
    \tr(\bfR^2)-\mu_n \cid \sum_{i=1}^\infty \xi_i-c_{\delta}^2\,, \qquad \nto\,.
\end{equation*}
To illustrate the boundary law in Theorem~\ref{thm:mainresult}, we use the Pareto-based symmetric model described in Section~\ref{sec:simulations}. The figure is intended as a visual aid for the Poisson process type limit; the broader simulation study is given in Section~\ref{sec:simulation:discussion}. Figure~\ref{fig:bound_histograms} provides simulation results for $\tr(\bfR^2)-\mu_n$ and we observe a good fit of both the histogram and kernel density to the density of $\etalim$ for all values of $\alpha\in (0,3)\backslash \{2\}$. The cumulants from Theorem~\ref{thm:boundary-cumulant} were used to construct a numerical approximation of the density of $\etalim$.
\begin{figure}
\begin{subfigure}{.5\textwidth}
  \centering
  \includegraphics[scale = 0.5]{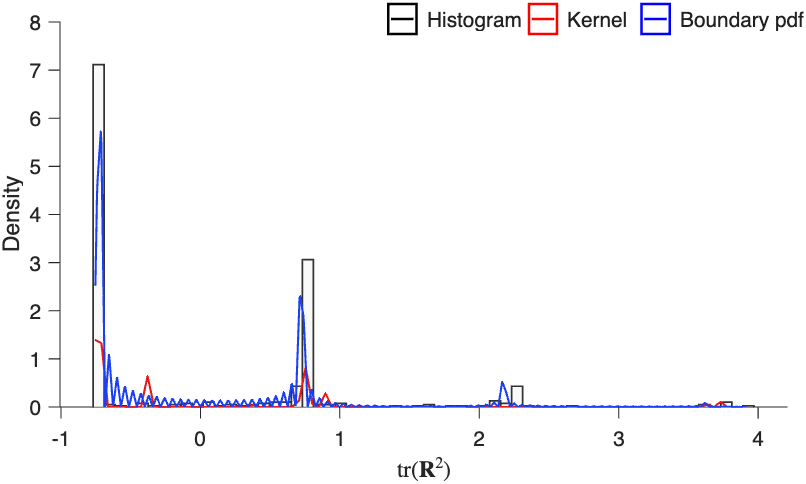}
  \caption{$\alpha = 0.05$}
\end{subfigure}%
\begin{subfigure}{.5\textwidth}
  \centering
  \includegraphics[scale = 0.5]{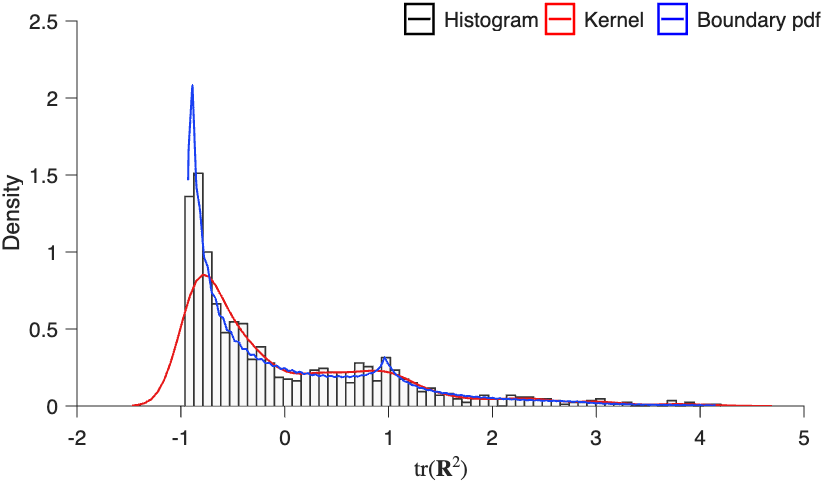}
  \caption{$\alpha = 0.5$}
\end{subfigure}
\bigskip
\begin{subfigure}{.5\textwidth}
  \centering
  \includegraphics[scale = 0.5]{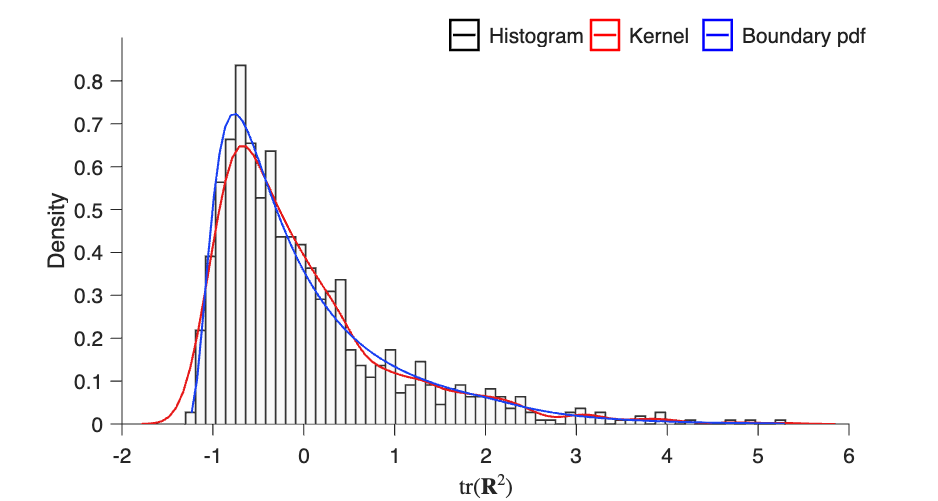}
  \caption{$\alpha = 2.2$}
\end{subfigure}%
\begin{subfigure}{.5\textwidth}
  \centering
  \includegraphics[scale = 0.5]{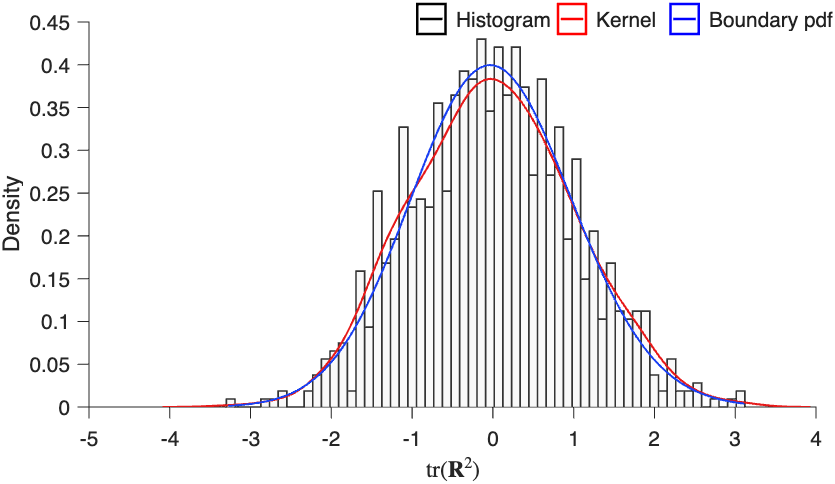}
  \caption{$\alpha = 2.8$}
\end{subfigure}
\caption{Simulations of  $\tr(\bfR^2)-\mu_n$ for different values of $\alpha$ and $p = n^{\delta^*(\alpha)}$, $n = 1000$ with $1000$ repetitions. }
\label{fig:bound_histograms}
\end{figure}

We now give a more intuitive interpretation of the Poisson point process representation in \eqref{eq:etalimrep}. In the regime $\alpha\in(0,2)$, each atom $\xi_i$ is a positive jump of size in $(0,2]$. The singular factor $x^{-1}$ means that very small positive jumps occur at high intensity, while larger jumps occur more rarely. The measure $\mu$ encodes the heavy-tail parameter $\alpha$. The random variable in \eqref{eq:etalimrep} is centered but not symmetric, which is consistent with the fact that its odd cumulants are generally nonzero.

As $\alpha\downarrow0$, the law $\mu$ of $2W$ converges weakly to $\delta_2$. In this limit, the L\'evy measure concentrates at the single positive jump size 2; more precisely,  $\nulim(dx) \Rightarrow \frac{c_\delta^2}{2}\delta_2(dx)$. Consequently, the boundary limit approaches the compensated compound Poisson variable, recovering the discrete Poisson type boundary law in the extreme heavy-tail limit $\alpha\downarrow0$ (see Corollary~\ref{cor:poisson}).

Compared to the regime $\alpha\in(0,2)$, the L\'evy density contains $x^{-2}$ instead of $x^{-1}$ in the regime $\alpha\in[2,3)$, meaning even higher activity of very small jumps. This enhances more Gaussian-like behavior in finite samples via many tiny jumps, despite the limit remaining non-Gaussian unless the higher cumulants vanish. This is also visible in Figure~\ref{fig:bound_histograms}(D), where the distribution of $\tr(\bfR^2)-\mu_n$  appears close to Gaussian in shape, despite having the non-Gaussian $\etalim$ as its limit.

\subsection*{The impact of the slowly varying function $L$}
The Poisson point process interpretation in Theorem~\ref{thm:mainresult} is formulated under the special growth regime of the dimension $p$ in \eqref{eq:pboundary}. On the one hand, we see that for $\alpha\in (0,2)$ this regime does not depend on specific properties of the slowly varying function $L$. On the other hand, for $\alpha\in (2,3)$ we have 
\begin{equation*}
    p\sim c_\delta \,n^{\delta^*(\alpha)} \frac{1}{L(n^{1/2})}\,,\qquad \nto\,,
\end{equation*}
so that the function $L$ affects the growth regime in a nontrivial way. To illustrate the effect of $L$, suppose instead that we impose the unified scaling
\begin{equation}\label{eq:unifiedscaling}
    p\sim c_{\delta}\,n^{\delta^*(\alpha)}, \qquad \nto\,, \alpha\in(2,3)\,, c_\delta>0.
\end{equation}
Under \eqref{eq:unifiedscaling}, the slowly varying factor $L(n^{1/2})$ directly affects the limiting behavior. If $L(n^{1/2})$ converges to some positive constant $q$, then the statements of Theorems~\ref{thm:leading-piecewise},~\ref{thm:boundary-cumulant} and~\ref{thm:Poisson point process representation} remain valid with $c_\delta$ replaced by $c_\delta \,q$. In contrast, if $L(n^{1/2})\to 0$ or $L(n^{1/2})\to\infty$, then the limiting behavior of $\Tone$ falls into the red or blue region, respectively, in Figure~\ref{fig:phasetransition}. The possibilities under the unified scaling \eqref{eq:unifiedscaling} are summarized in the right part of Figure~\ref{fig:phasetransition}, where boundary limit refers to the Poisson process type limit.

More precisely, for $s\ge 3$, the moment ${\E[\Tone^s]}$admits an asymptotic expansion in powers of $1/L(n^{1/2})$, with exponents of the form $s-2v$ where $v=1,\dots,\lfloor s/2\rfloor$. Hence, when $L(n^{1/2})\to0$ or $L(n^{1/2})\to\infty$, it is enough to identify the extreme powers appearing in the expansion, since all intermediate orders are negligible by comparison. In fact,
\begin{equation*}
    \E[\Tone^s]\sim
        \begin{cases}
        \displaystyle
        \frac{1}{2}\,\frac{c_s^{2}(\alpha)}{c_2^{s}(\alpha)}
        \frac{c_{\delta}^{\,s-2}}{L(n^{1/2})^{\,s-2}}
        +
        \frac{4s!}{3!\left(\frac{s-3}{2}\right)!}\,
        \frac{c_3^{2}(\alpha)}{c_2^{3}(\alpha)}
        \frac{c_{\delta}}{L(n^{1/2})},
        & \text{if } s \text{ is odd},
        \\[10pt]
        \displaystyle
        \frac{1}{2}\,\frac{c_s^{2}(\alpha)}{c_2^{s}(\alpha)}
        \frac{c_{\delta}^{\,s-2}}{L(n^{1/2})^{\,s-2}}
        +(s-1)!!,
        & \text{if } s \text{ is even}.
        \end{cases}
\end{equation*}
If $L(n^{1/2})\to 0$, then
\begin{equation}\label{eq:momTone}
    \E[\Tone^s] \sim \frac{1}{2}\,\frac{c_s^{2}(\alpha)}{c_2^{s}(\alpha)} \frac{c_{\delta}^{\,s-2}}{L(n^{1/2})^{\,s-2}}, \qquad s\ge 3,
\end{equation}
and therefore $\E[\Tone^s]\to\infty$ for $s\ge 3$.

If $L(n^{1/2})\to\infty$, then
\begin{equation*}
    \E[\Tone^s]\to
        \begin{cases}
            0, & \text{if } s \text{ is odd},\\[4pt]
            (s-1)!!, & \text{if } s \text{ is even},
        \end{cases}
\end{equation*}
which we recognize as the moments of a standard normal random variable.

Recall that the tail behavior is given by $\P(|X_{11}|>x)=x^{-\alpha}L(x)$. Heuristically, the case $L(n^{1/2})\to\infty$ is analogous to replacing $\alpha$ by $\alpha+\varepsilon$ for some $\varepsilon>0$. In this sense, the resulting behavior is similar to that at $(\alpha+\varepsilon,\delta^*(\alpha))$, which lies in the Gaussian regime (see the left part of Figure~\ref{fig:phasetransition}). On the other hand, the case $L(n^{1/2})\to0$ corresponds heuristically to replacing $\alpha$ by $\alpha-\varepsilon$, and hence leads to behavior analogous to that at $(\alpha-\varepsilon,\delta^*(\alpha))$, which lies in the non-Gaussian regime\footnote{For transparency, we clarify that in this paper we only show that the moments of $\Tone$ diverge for $s\ge 3$. Hence, we do not rigorously prove that the asymptotic distribution of $\Tone$ is non-Gaussian in the red region. Nevertheless, non-Gaussianity is strongly supported by extensive simulations.}.

Therefore, the three cases $L(n^{1/2})\to0$, $L(n^{1/2})\to q\in(0,\infty)$, $L(n^{1/2})\to\infty$ naturally describe whether the unified scaling \eqref{eq:unifiedscaling} falls to the left of, on, or to the right of the phase transition boundary \eqref{eq:pboundary}.

\section{Simulations and discussion}\label{sec:simulation:discussion}\setcounter{equation}{0}
The boundary simulations illustrating the Poisson process type limit were already presented in Figure~\ref{fig:bound_histograms}, immediately after Theorem~\ref{thm:mainresult}, where they serve as a visual interpretation of the boundary law. In the present section we give the data-generating mechanism used throughout the simulations and complement the boundary illustration with Gaussian-regime simulations, comparisons across regimes, and experiments for non-symmetric entries.

\subsection{Gaussian regimes and comparison across regimes}\label{sec:simulations}
To illustrate the role of $\alpha$ and the need of symmetry, we perform some simulations. Throughout the symmetric simulations, let $Z_1, Z_2$ be iid Pareto random variables with parameter $\alpha$, that is, $\P(Z_1 > x) = x^{-\alpha}$, $x>1$. We generate symmetric entries from the difference $Z_1-Z_2$. For $\alpha\in(0,2)$, no unit-variance normalization is applied since the second moment is infinite; this corresponds to the infinite-variance regime covered separately in Lemma~\ref{lem:asymp_variance}. For $\alpha>2$, we normalize the entries to have unit variance and set
\begin{equation*}
    X_{11} \eid \frac{Z_1 - Z_2}{K_\alpha^{1/2}},\qquad 
    K_\alpha:=\Var(Z_1 - Z_2) = \frac{2 \alpha}{(\alpha - 1)^2(\alpha - 2)}.
\end{equation*}
Using Lemma~\ref{lem:L(x)_sym_pareto}, it is straightforward to show that $Z_1-Z_2$ has a regularly varying tail with index $\alpha$. In the notation of Lemma~\ref{lem:L(x)_sym_pareto}, we let $S_2 = Z_1 + (-Z_2)$ and since, for $x>0$, $\P(Z_1 > x) = \bar F(x)$ and $\P(-Z_2 > x) = 0$, we get $c_1^+ = 1$ and $c_2^+ = 0$. We also find that $c_1^- = 0$ and $c_2^- = 1$ since $\P(Z_1 \leq -x) = 0$ and $\P(-Z_2 \leq -x) = \bar F(x)$. This satisfies \eqref{eq:cond_1_L(x)_pareto} and it is obvious that also \eqref{eq:cond_2_L(x)_pareto} holds. Now Lemma~\ref{lem:L(x)_sym_pareto} implies that $\P(Z_1 - Z_2 > x) \sim x^{-\alpha}$ and $\P(-Z_1 + Z_2 > x) \sim x^{-\alpha}$ yielding that $\P(|Z_1 - Z_2| > x) \sim 2 x^{-\alpha}$, as $x\to \infty$. Consequently, for $\alpha>2$,
\begin{equation*}
    \P(|X_{11}| > x) = \P(|Z_1 - Z_2| > x K_\alpha^{0.5}) \sim x^{-\alpha} 2 K_\alpha^{-0.5\alpha}
\end{equation*}
meaning that $L(x) \sim 2 K_\alpha^{-0.5\alpha}$. Thus, $L(n^{1/2}) \approx 2 K_\alpha^{-0.5\alpha}$ and we can approximate $\sigma_n^2$ by using Lemma~\ref{lem:asymp_variance}.

We start with simulations of the universal and non-universal CLTs in Theorem~\ref{thm:clt2}. Figure~\ref{fig:sym_histograms} displays the histograms of the normalized statistic $(\tr(\bfR^2)-\mu_n)/\sigma_n$ for different values of $\alpha$. In all four panels, the choice of $p$ places the model in the Gaussian regime. Especially, for smaller values of $\alpha$, the entries are much heavier-tailed, but the condition on $p$ still puts the statistic in the Gaussian regime covered by Theorem~\ref{thm:clt2}. We see that the histograms are close to the standard normal density. This confirms our asymptotic result that, throughout the blue regime, the centered and normalized statistic has Gaussian fluctuations across different values of the tail index.
\begin{figure}
\begin{subfigure}{.5\textwidth}
  \centering
  \includegraphics[scale = 0.4]{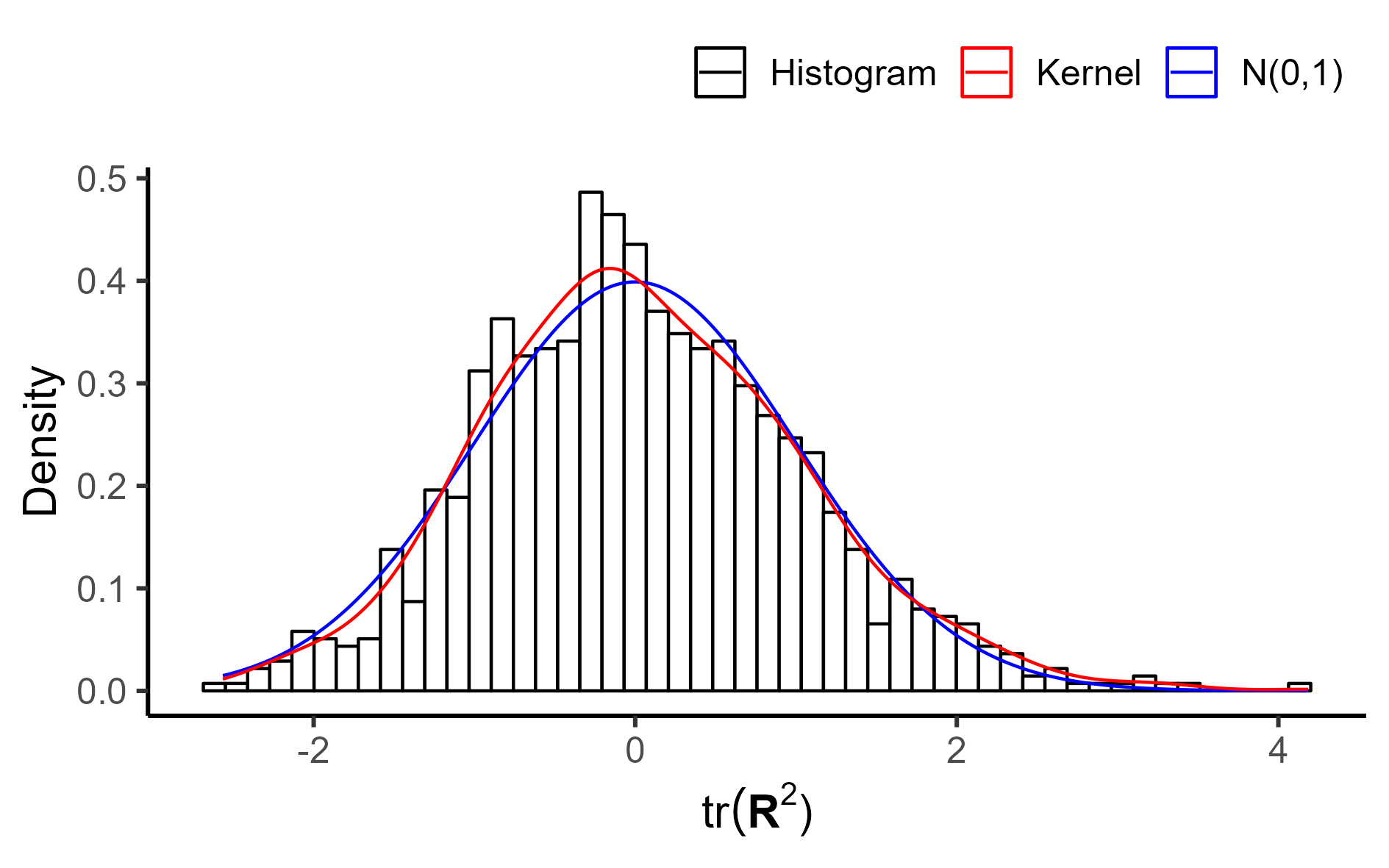}
  \caption{$\alpha = 0.5$}
\end{subfigure}%
\begin{subfigure}{.5\textwidth}
  \centering
  \includegraphics[scale = 0.4]{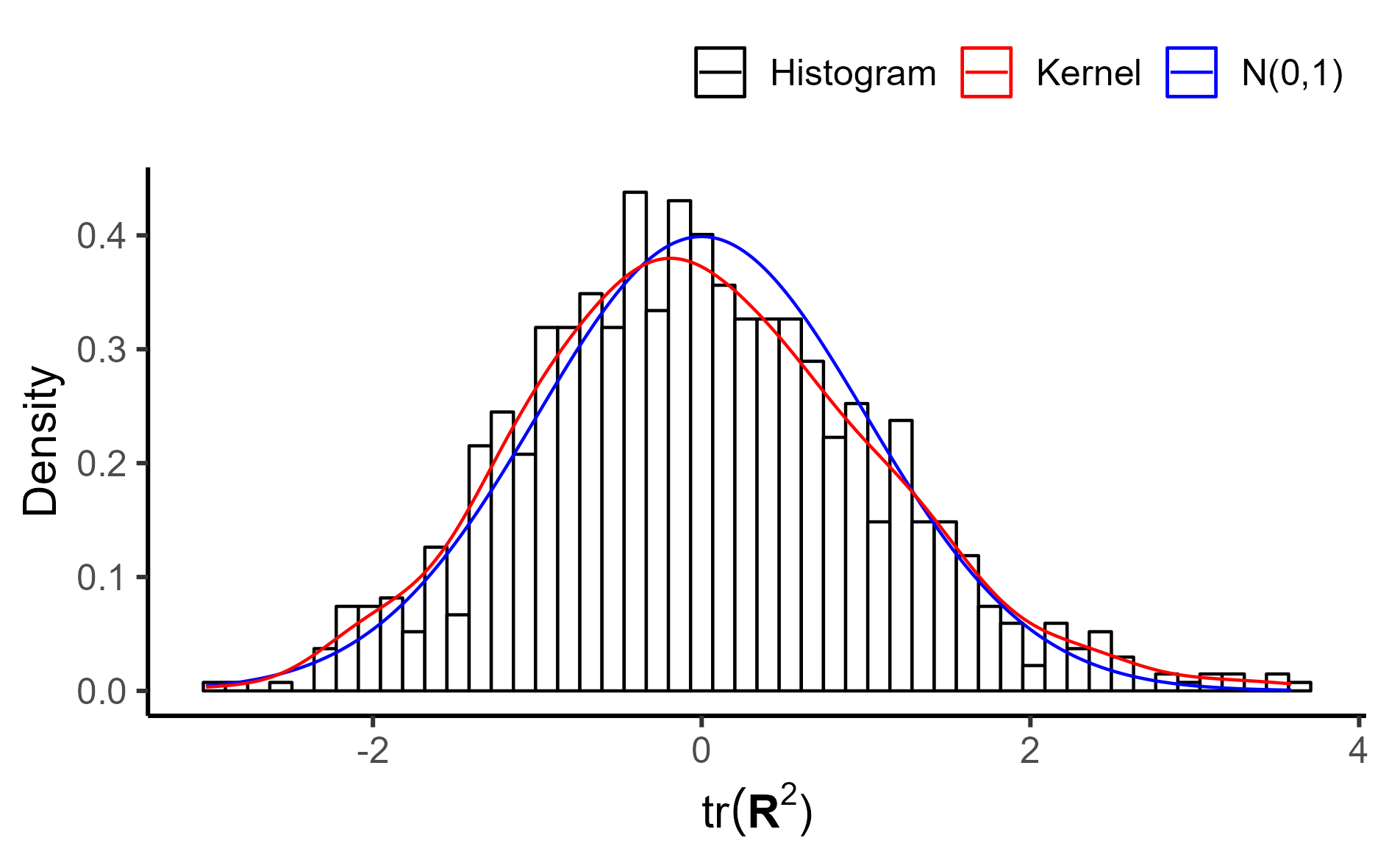}
  \caption{$\alpha = 1.1$}
\end{subfigure}
\bigskip
\begin{subfigure}{.5\textwidth}
  \centering
  \includegraphics[scale = 0.4]{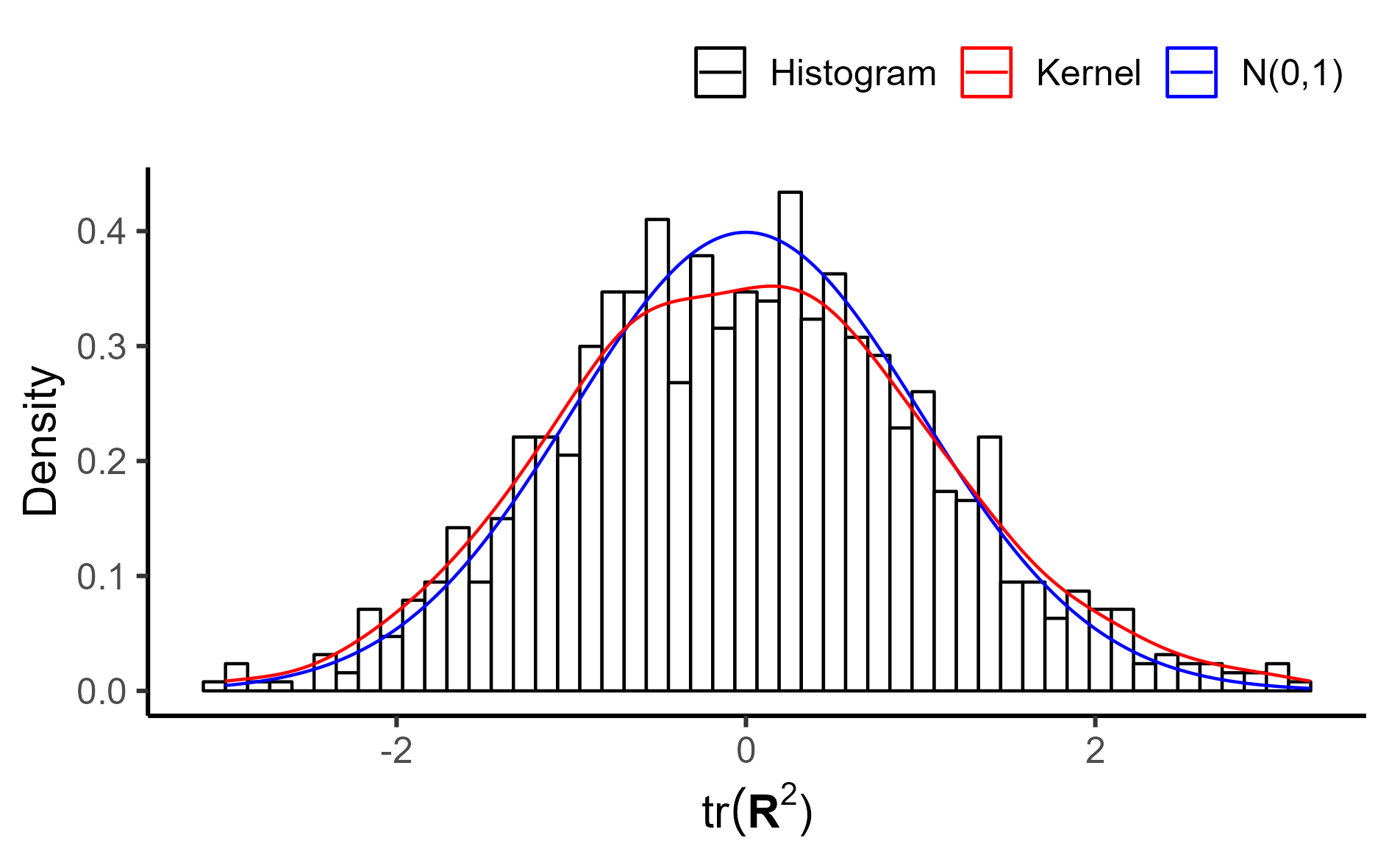}
  \caption{$\alpha = 2.8$}
\end{subfigure}%
\begin{subfigure}{.5\textwidth}
  \centering
  \includegraphics[scale = 0.4]{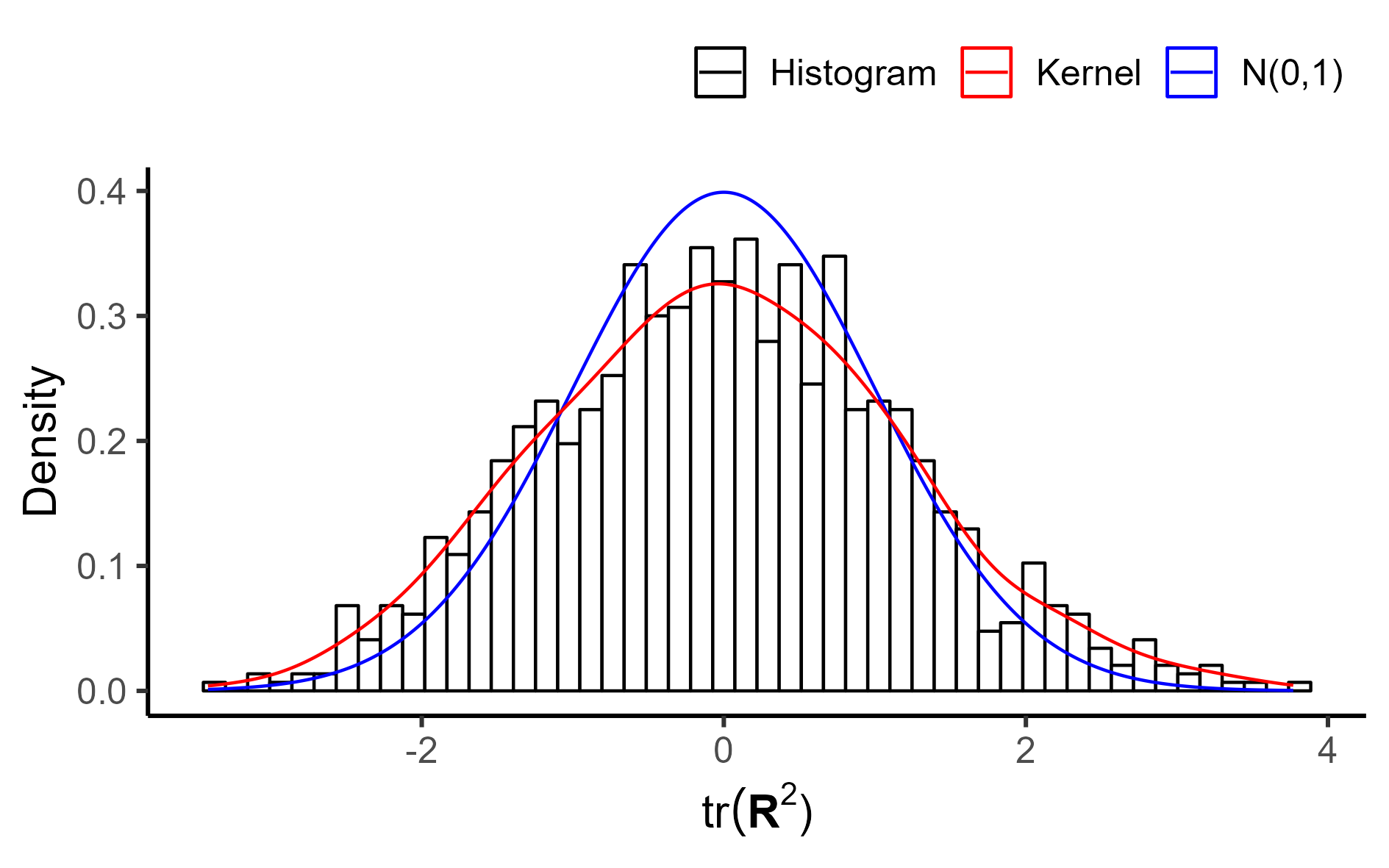}
  \caption{$\alpha = 3.9$}
\end{subfigure}
\caption{Simulations of the central limit theorem for $\tr(\bfR^2)$ for different values of $\alpha$ and $p = 400$, $n = 1000$ with 1000 repetitions.}
\label{fig:sym_histograms}
\end{figure}
To get some understanding of the behavior in the tails we present the Q-Q plots seen in Figure~\ref{fig:qqplots}. In the cases with small $\alpha$ the tails deviate the most from the straight line, which is not unexpected since the tails get heavier for smaller value of the shape parameter $\alpha$ in a Pareto distribution. The Q-Q-plots generally support the Gaussian limit in Theorem~\ref{thm:clt2}, but also indicate that the convergence rates might be slower in more heavy-tailed scenarios. 
\begin{figure}[!htbp]
\begin{subfigure}{.5\textwidth}
  \centering
  \includegraphics[scale = 0.4]{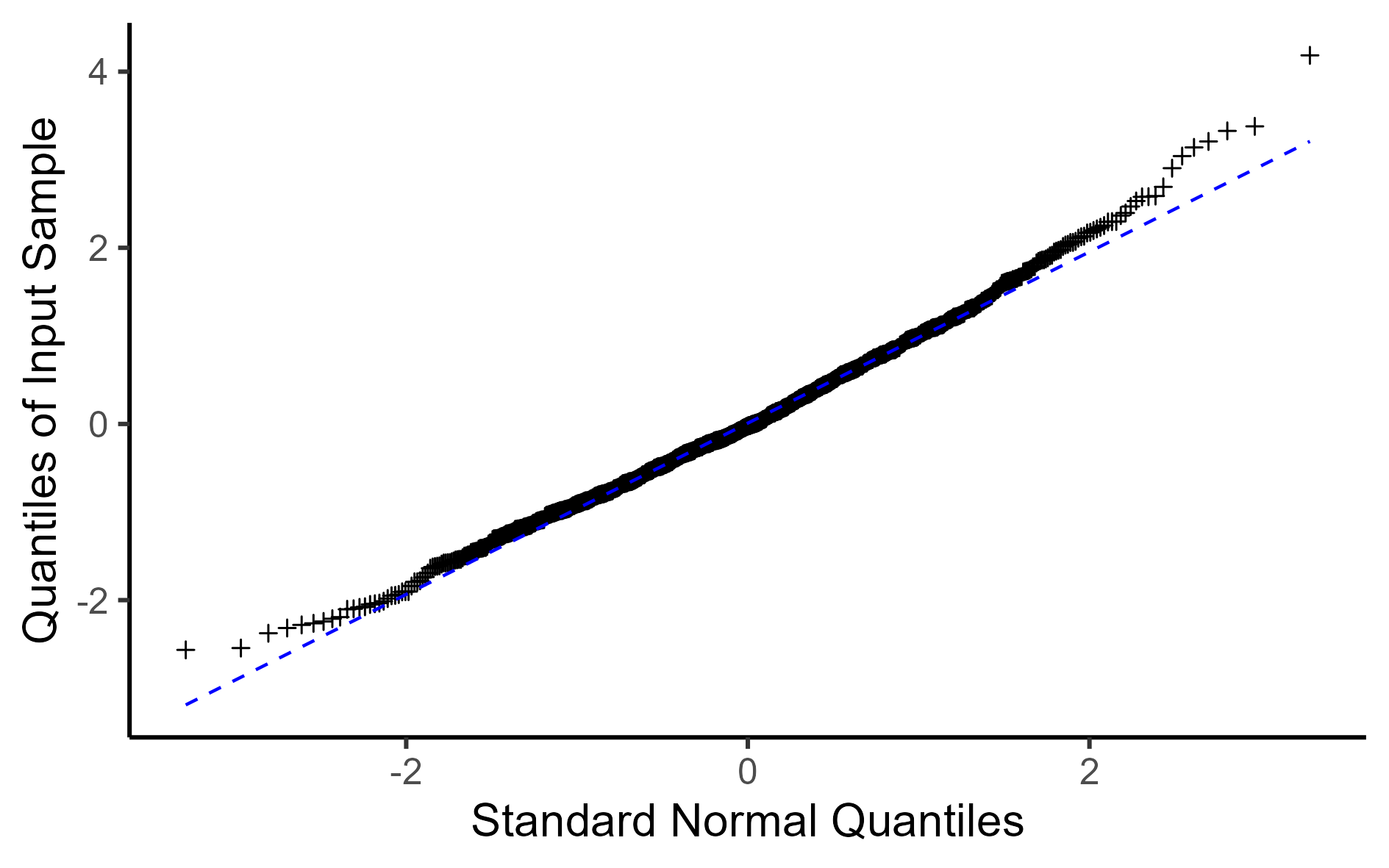}
  \caption{$\alpha = 0.5$}
\end{subfigure}%
\begin{subfigure}{.5\textwidth}
  \centering
  \includegraphics[scale = 0.4]{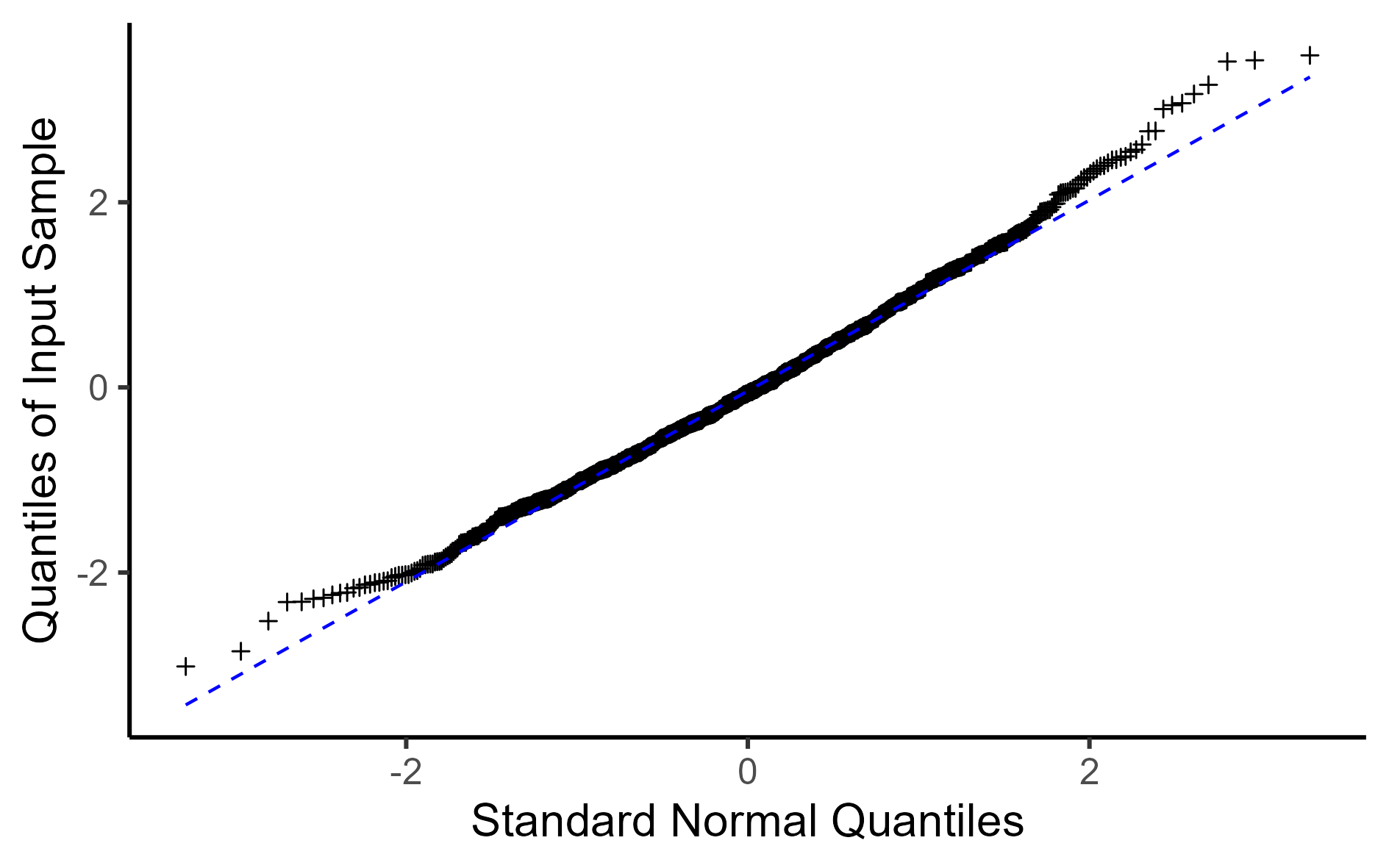}
  \caption{$\alpha = 1.1$}
\end{subfigure}
\bigskip
\begin{subfigure}{.5\textwidth}
  \centering
  \includegraphics[scale = 0.4]{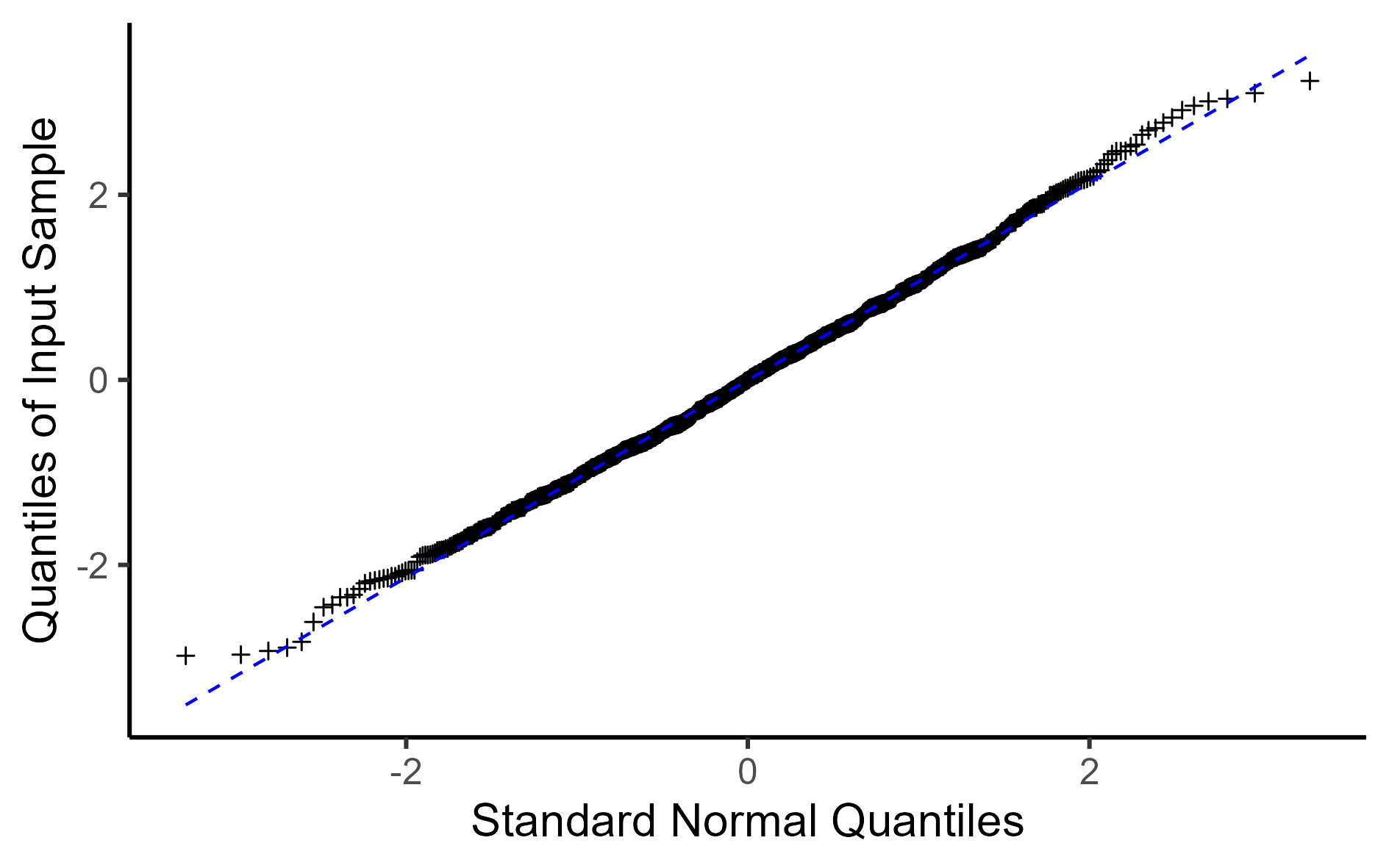}
  \caption{$\alpha = 2.8$}
\end{subfigure}%
\begin{subfigure}{.5\textwidth}
  \centering
  \includegraphics[scale = 0.4]{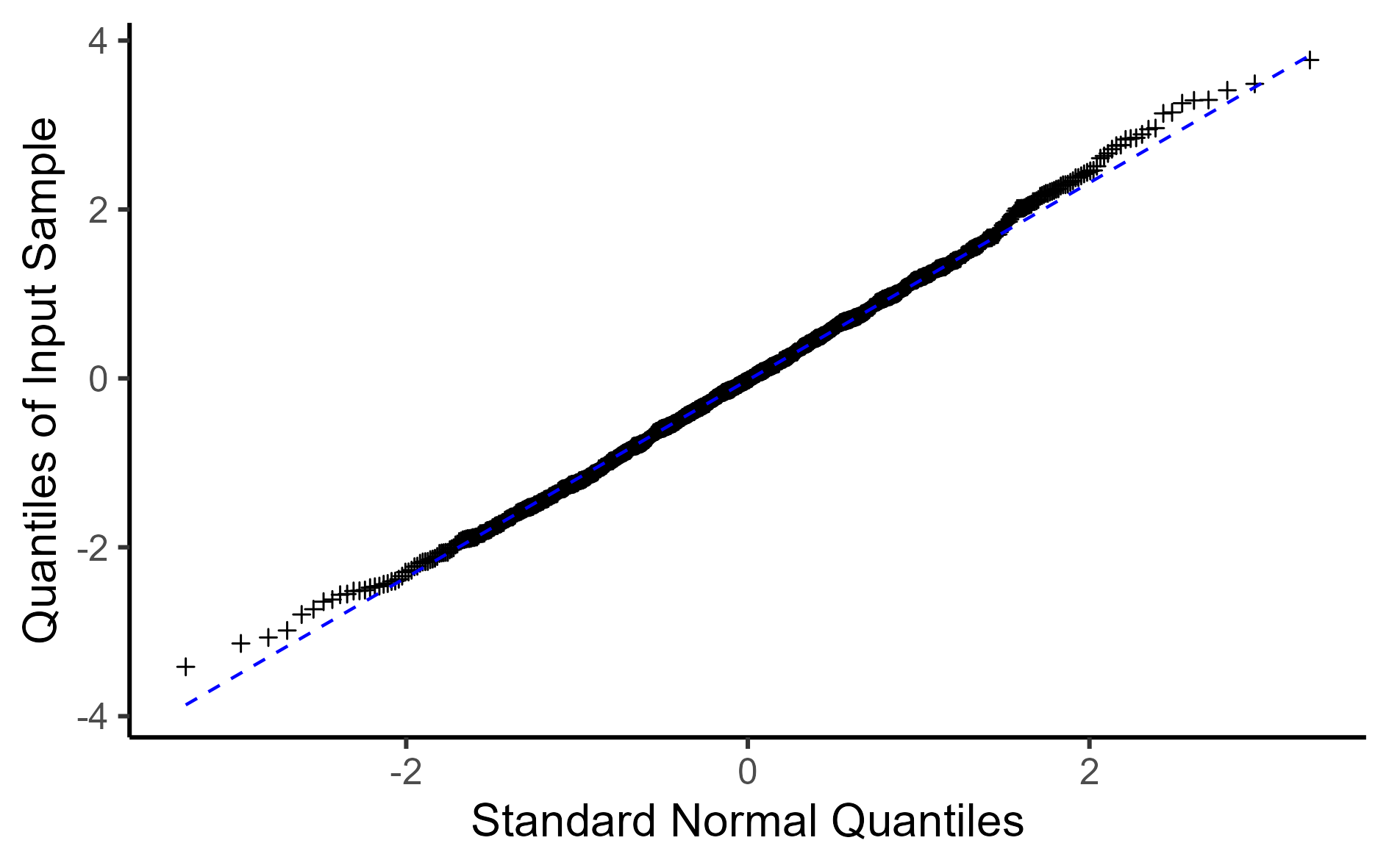}
  \caption{$\alpha = 3.9$}
\end{subfigure}
\caption{Q-Q-plots for the CLT of $\tr(\bfR^2)$ for different values of $\alpha$ and $p = 400$, $n = 1000$ with 1000 repetitions. }
\label{fig:qqplots}
\end{figure}
Recall that simulations of the Poisson process type limit (Theorem~\ref{thm:mainresult}) in the regime $p\approx n^{\delta^*(\alpha)}$ were presented in Figure~\ref{fig:bound_histograms}. Next, we compare the simulation results in the three regimes $p> n^{\delta^*(\alpha)}$, $p\approx n^{\delta^*(\alpha)}$, $p< n^{\delta^*(\alpha)}$. These regimes correspond respectively to the Gaussian regime (Theorem~\ref{thm:clt2}), the phase transition boundary, and the red region in Figure~\ref{fig:phasetransition}. Similarly to \eqref{eq:momTone}, it can be shown that, on the red region, moments of $\Tone$ of order at least three diverge to infinity. This is in contrast to the moment asymptotics on the phase transition boundary (see Theorem~\ref{thm:leading-piecewise}). The numerical results in Figure~\ref{fig:compare_histograms} are fully consistent with the theoretical predictions. When $p$ is larger than the boundary scale $n^{\delta^*(\alpha)}$, the histogram of the normalized statistic matches the standard normal distribution very well, in agreement with the CLT in Theorem~\ref{thm:clt2}. At the critical scaling $p\approx n^{\delta^*(\alpha)}$, the histogram agrees much better with the Poisson process type boundary distribution from Theorem~\ref{thm:mainresult} than with the Gaussian law. This illustrates the transition from Gaussian fluctuations to the boundary Poisson process type fluctuations. In the red region $p < n^{\delta^*(\alpha)}$, the histogram matches neither the standard normal distribution nor the boundary limit distribution, showing that the statistic has moved beyond the critical window into a different non-Gaussian regime. The characterization of the limiting distribution in this regime is a topic for future research.
\begin{figure}[!htbp]
\begin{subfigure}{.32\textwidth}
  \centering
  \includegraphics[scale = 0.3]{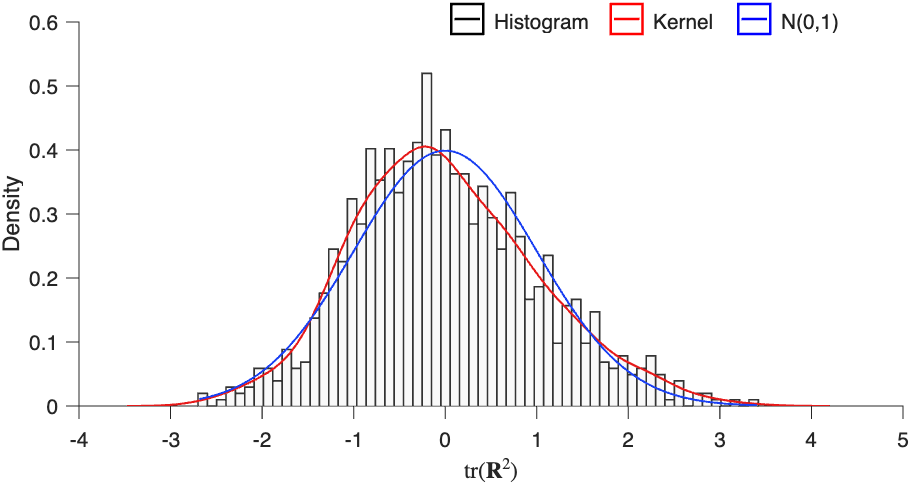}
  \caption{$p = 400 > n^{\delta^*(\alpha)}$}
\end{subfigure}%
\begin{subfigure}{.32\textwidth}
  \centering
  \includegraphics[scale = 0.3]{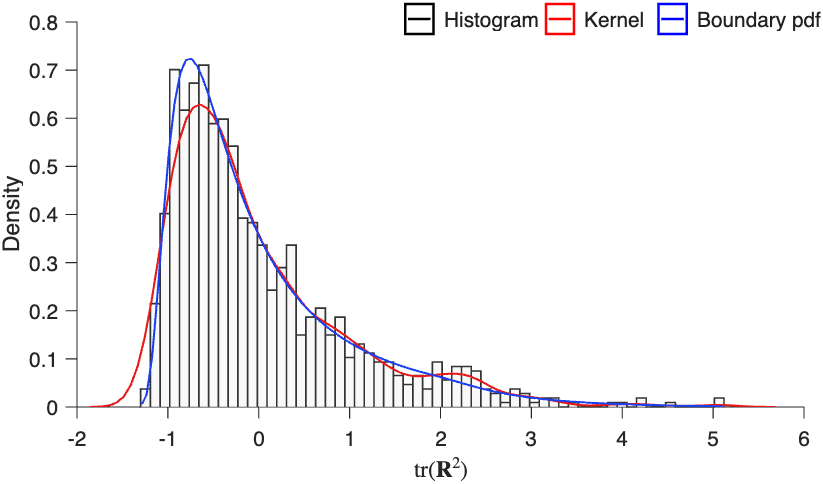}
  \caption{$p = 90\approx n^{\delta^*(\alpha)}$}
\end{subfigure}
\begin{subfigure}{.32\textwidth}
  \centering
  \includegraphics[scale = 0.3]{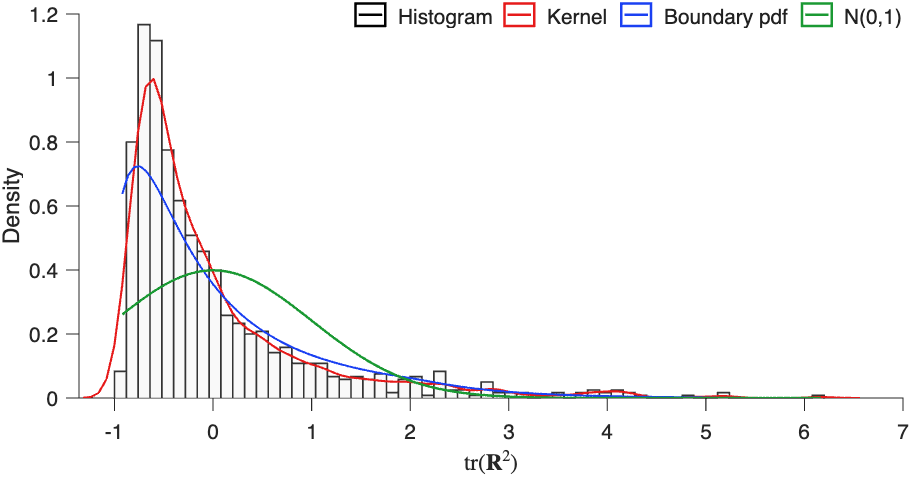}
  \caption{$p =60 < n^{\delta^*(\alpha)}$}
\end{subfigure}%
\caption{Comparison of histograms of $(\tr(\bfR^2)-\mu_n)/\sigma_n$ for different values of $p$, with $\alpha=1.1$, $n=8000$, and 1000 repetitions.}
\label{fig:compare_histograms}
\end{figure}

\subsection{Comments on the non-symmetric case}
The assumption that $X_{11}$ is symmetric is not only of significant technical  importance in our proofs, but also indispensable for the validity of the CLT presented in Theorem~\ref{thm:clt2} in general. To analyze the requirement of symmetry and the effect that specific distributions may have, we simulate from two different non-symmetric distributions. For the first case, define the entries $X_{11} \eid Z_1 - \E[Z_1]$ of the data matrix where $Z_1$ follows a Pareto distribution with  parameter $\alpha>1$. For the second case, generate instead the entries $X_{11} \eid T^2 - \E[T^2]$, where $T$ follows a t-distribution with $2\alpha$ degrees of freedom. Note that in both cases $X_{11}$ is regularly varying with index $\alpha$. The simulation results are shown in Figure~\ref{fig:nonsym_histograms}. An interesting observation is that in both cases of different non-symmetric distributions, a similar shift in the mean appears to be present for $\alpha\in\{1.1, 1.3\}$. The main consequence of symmetry violation is that the limiting distribution resembles a normal distribution but with larger variance and some considerable shift in the mean is seen in the top two rows where $\alpha$ is small. Those changes can be partially explained as follows. We shall focus on the mean.
\begin{figure}
\begin{subfigure}{.45\textwidth}
  \centering
  \includegraphics[scale = 0.4]{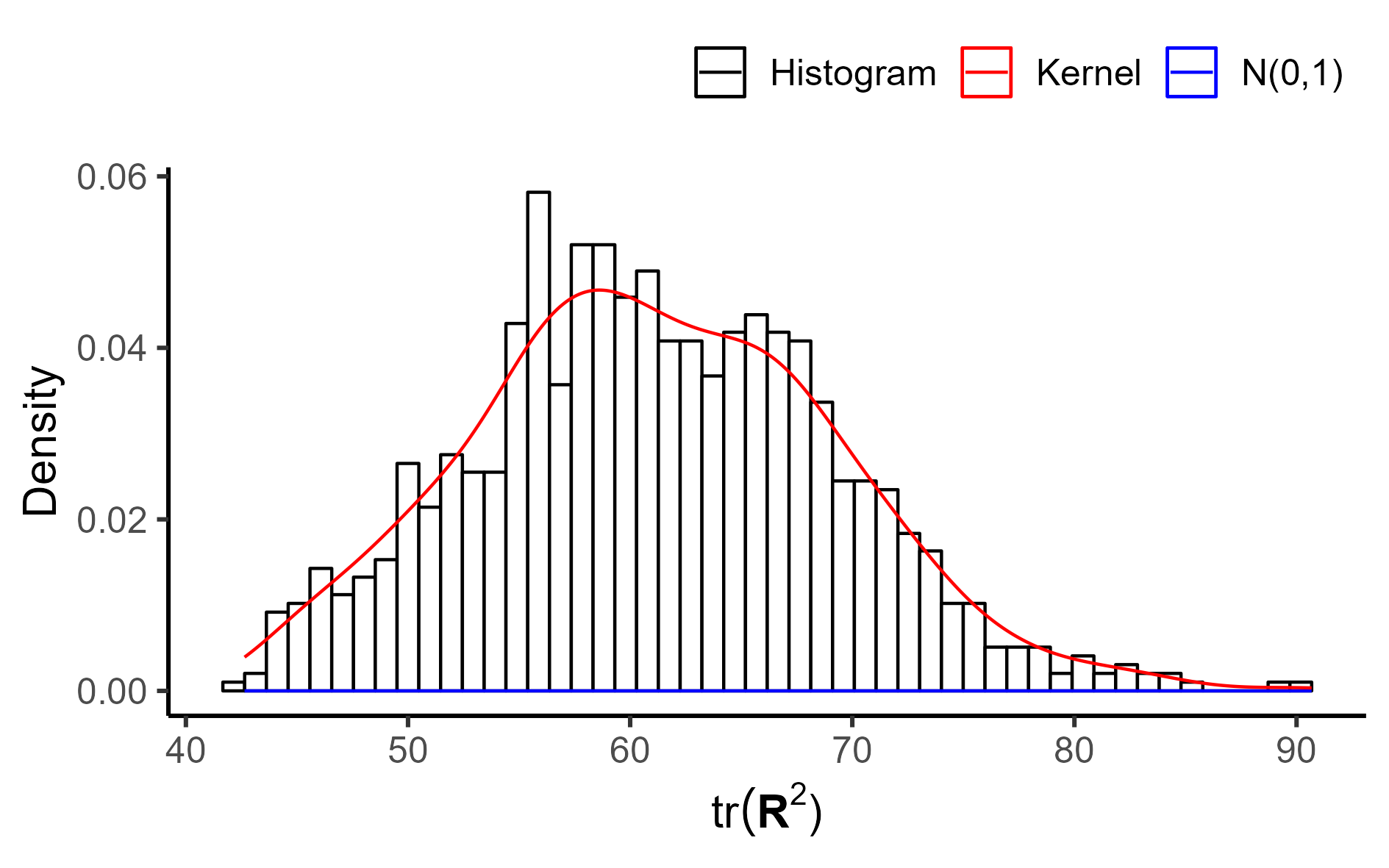}
  \caption{$\alpha = 1.1$}
\end{subfigure}%
\begin{subfigure}{.45\textwidth}
  \centering
  \includegraphics[scale = 0.4]{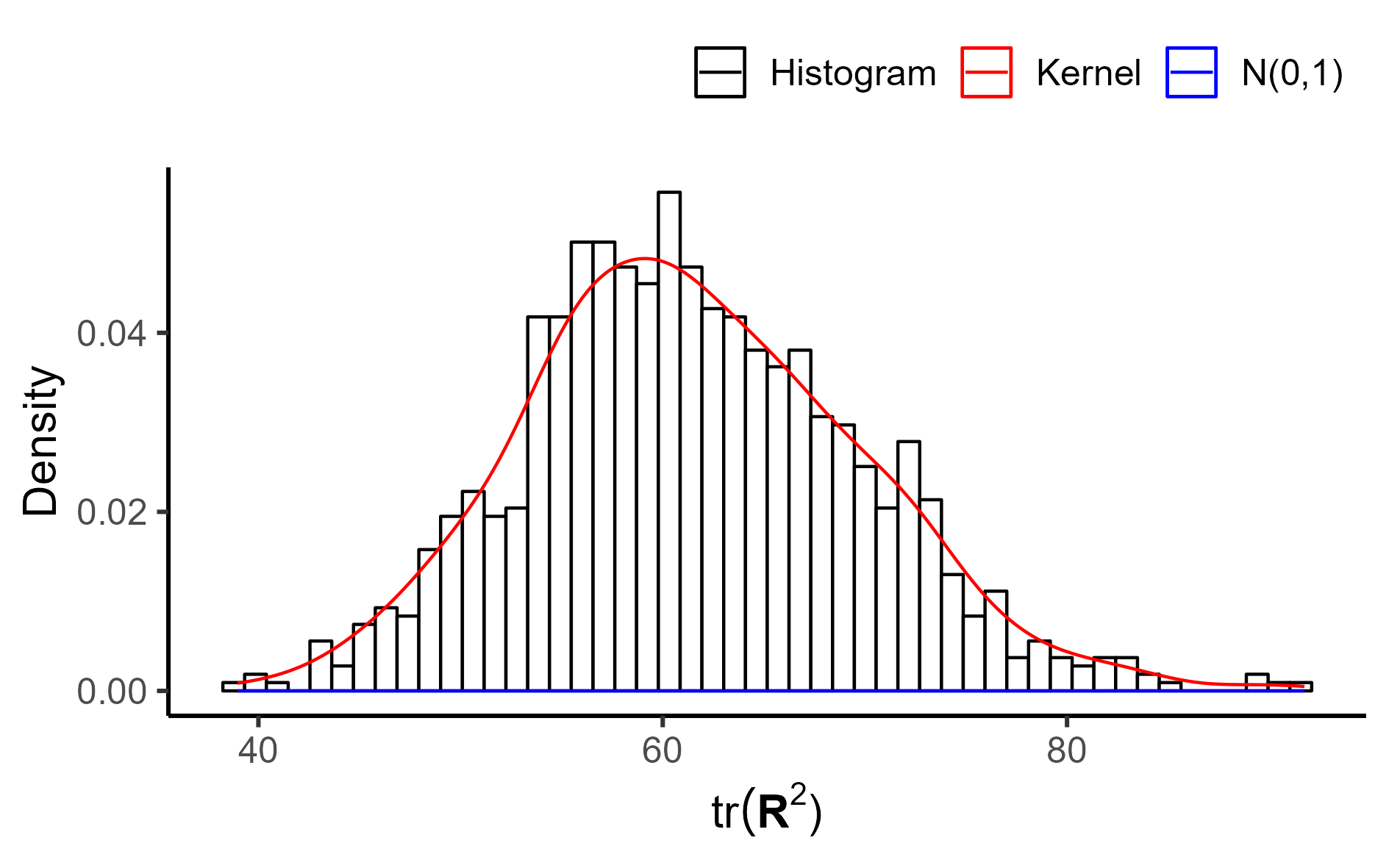}
  \caption{$\alpha = 1.1$}
\end{subfigure}
\bigskip
\begin{subfigure}{.45\textwidth}
  \centering
  \includegraphics[scale = 0.4]{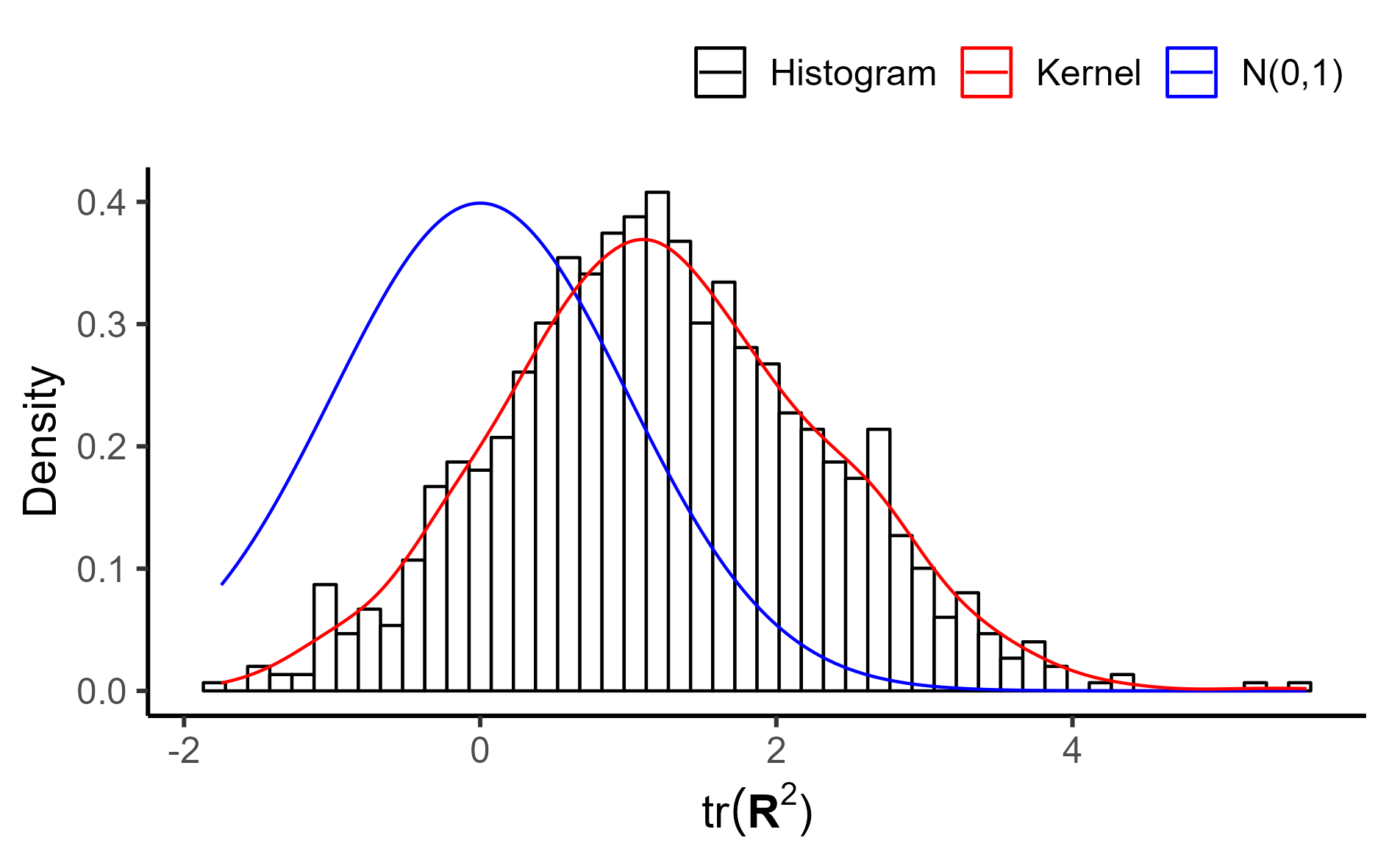}
  \caption{$\alpha = 1.3$}
\end{subfigure}%
\begin{subfigure}{.45\textwidth}
  \centering
  \includegraphics[scale = 0.4]{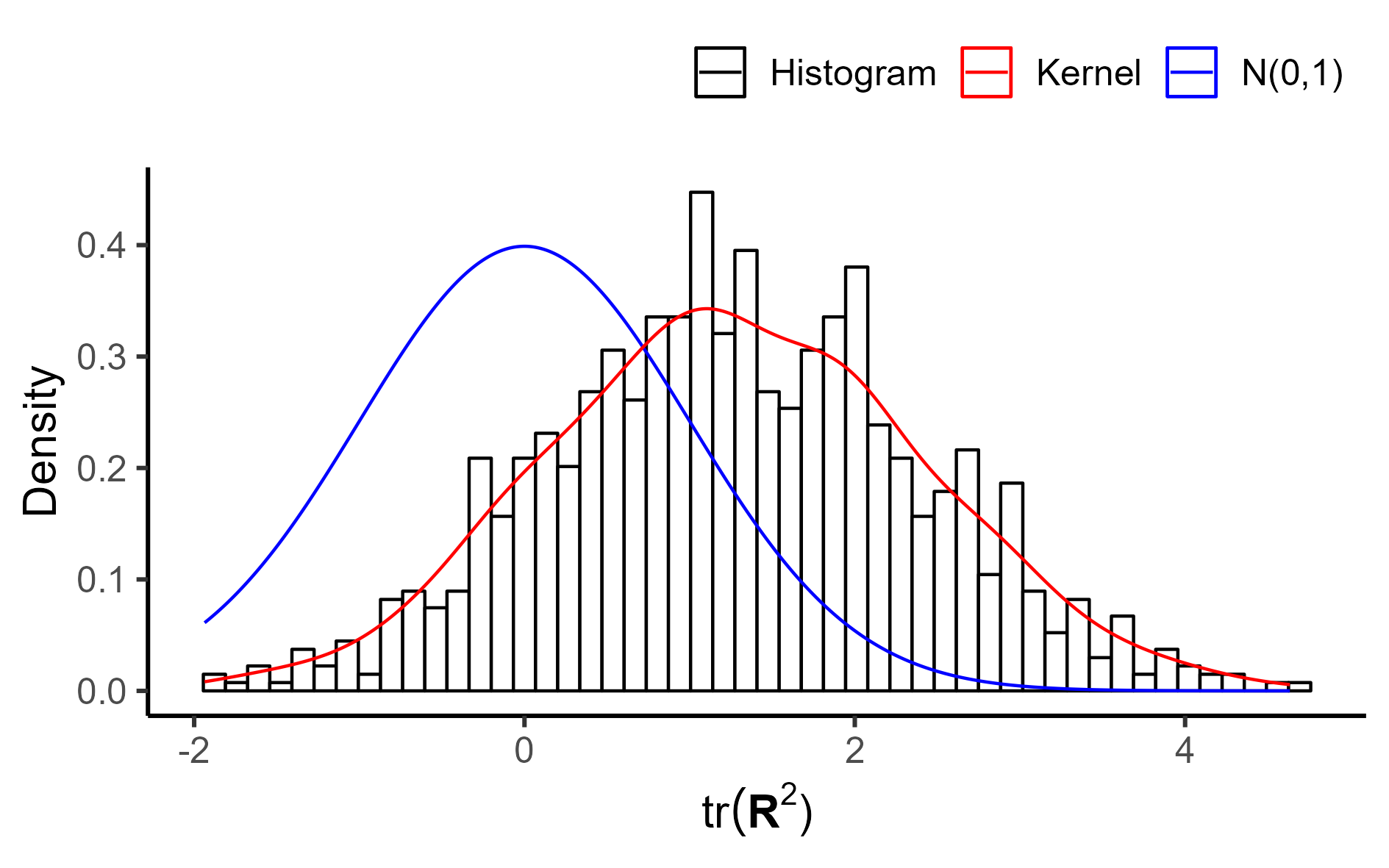}
  \caption{$\alpha = 1.3$}
\end{subfigure}
\bigskip
\begin{subfigure}{.45\textwidth}
  \centering
  \includegraphics[scale = 0.4]{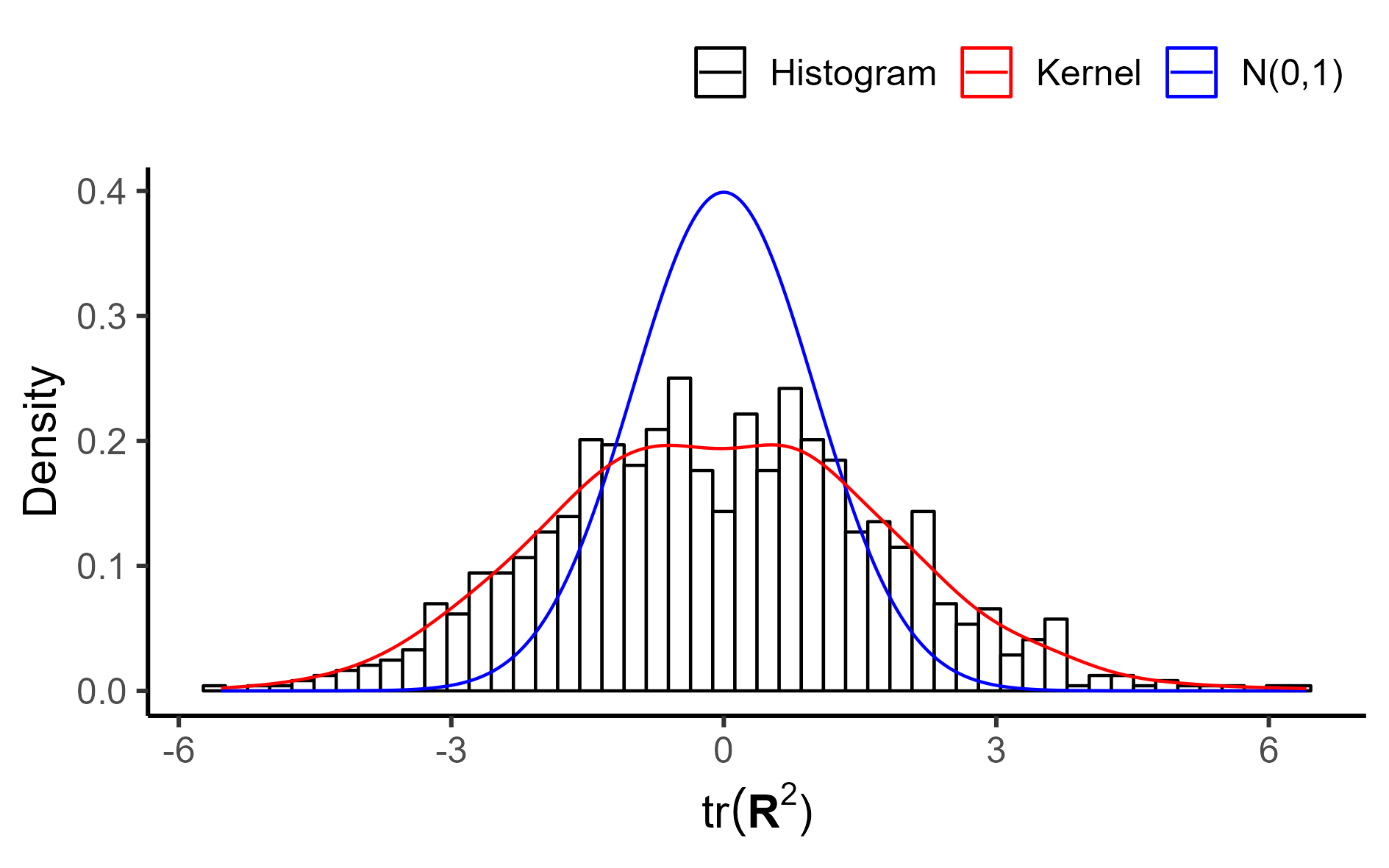}
  \caption{$\alpha = 3.2$}
\end{subfigure}%
\begin{subfigure}{.45\textwidth}
  \centering
  \includegraphics[scale = 0.4]{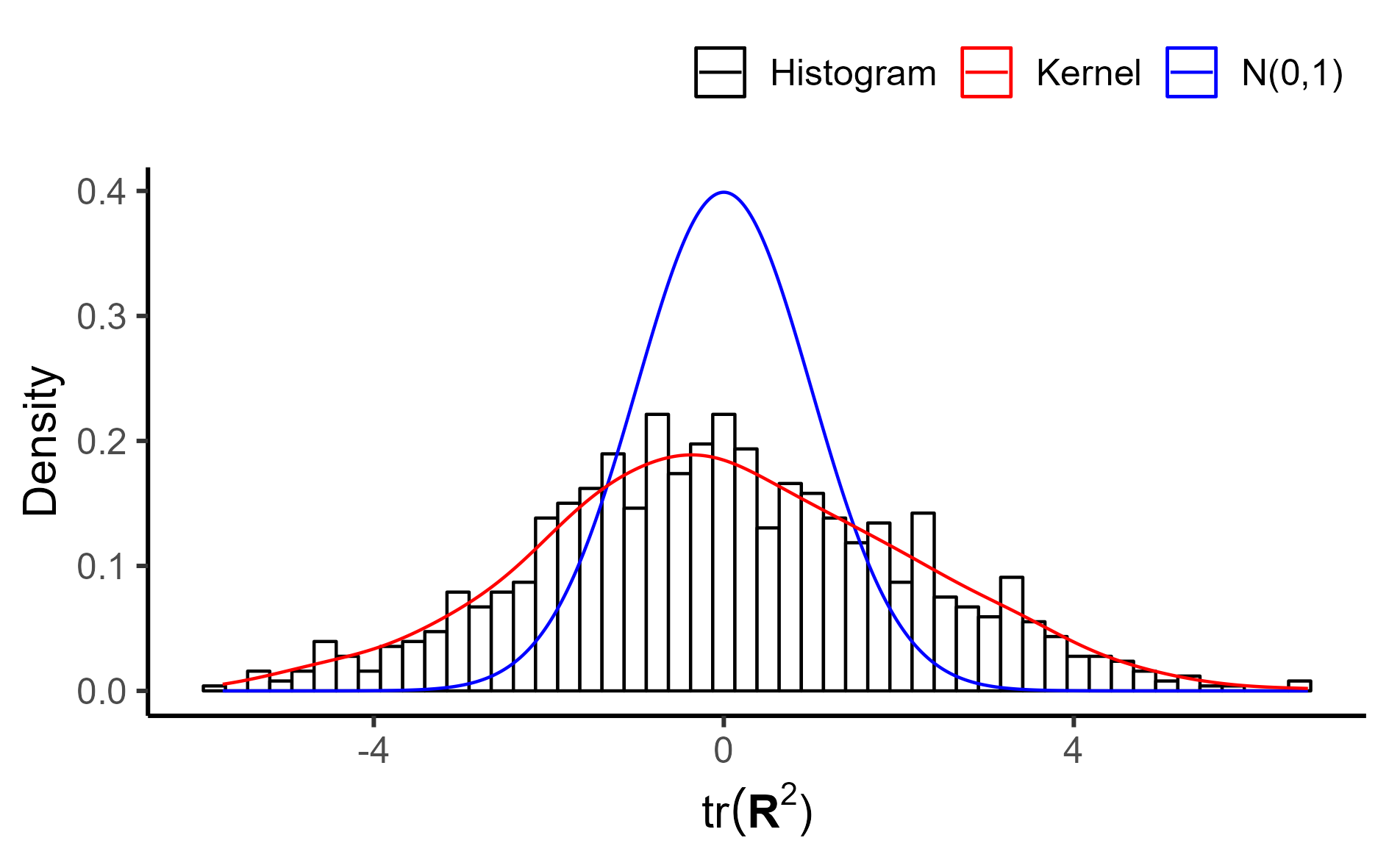}
  \caption{$\alpha = 3.2$}
\end{subfigure}
\bigskip
\begin{subfigure}{.45\textwidth}
  \centering
  \includegraphics[scale = 0.4]{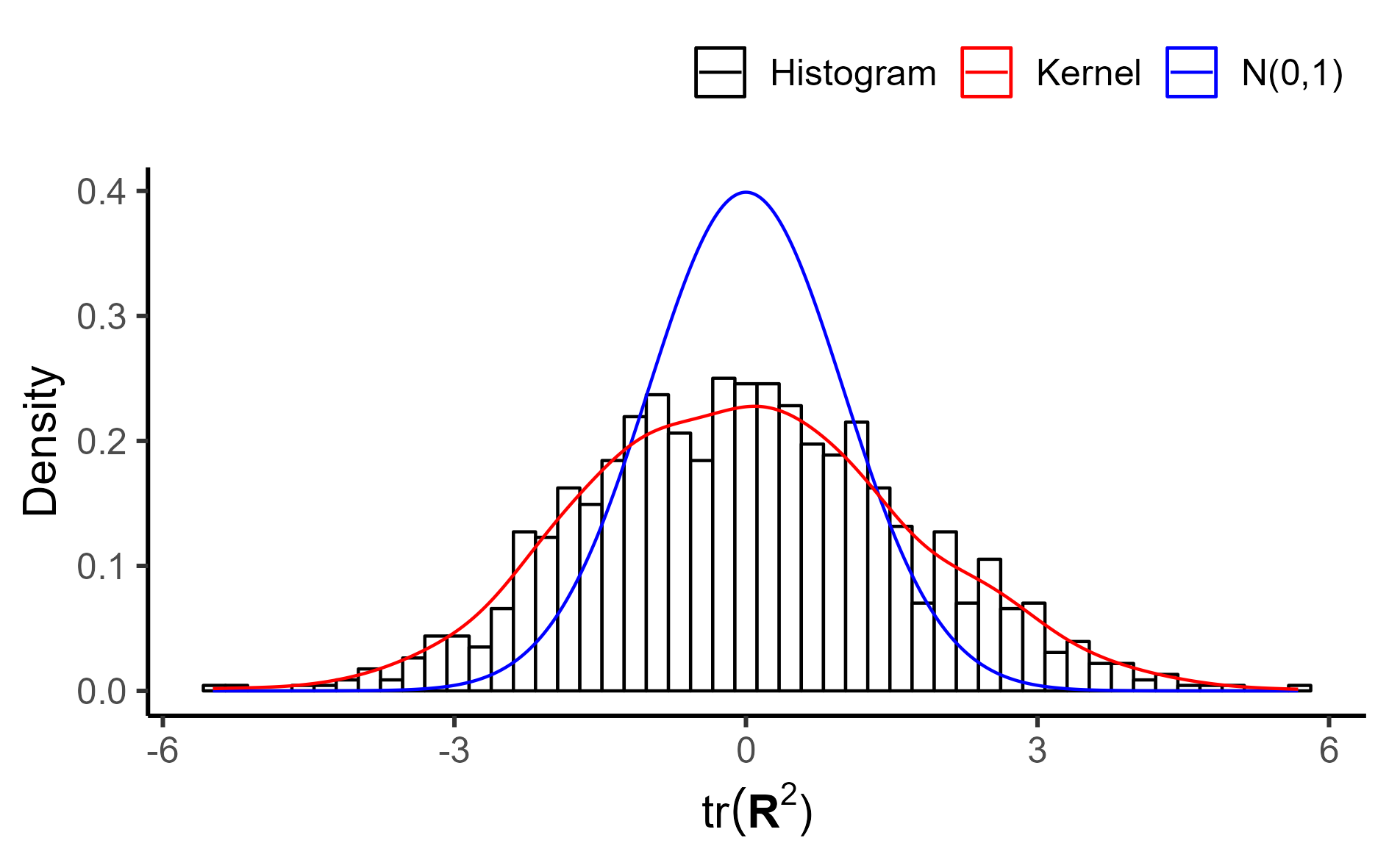}
  \caption{$\alpha = 3.5$}
\end{subfigure}%
\begin{subfigure}{.45\textwidth}
  \centering
  \includegraphics[scale = 0.4]{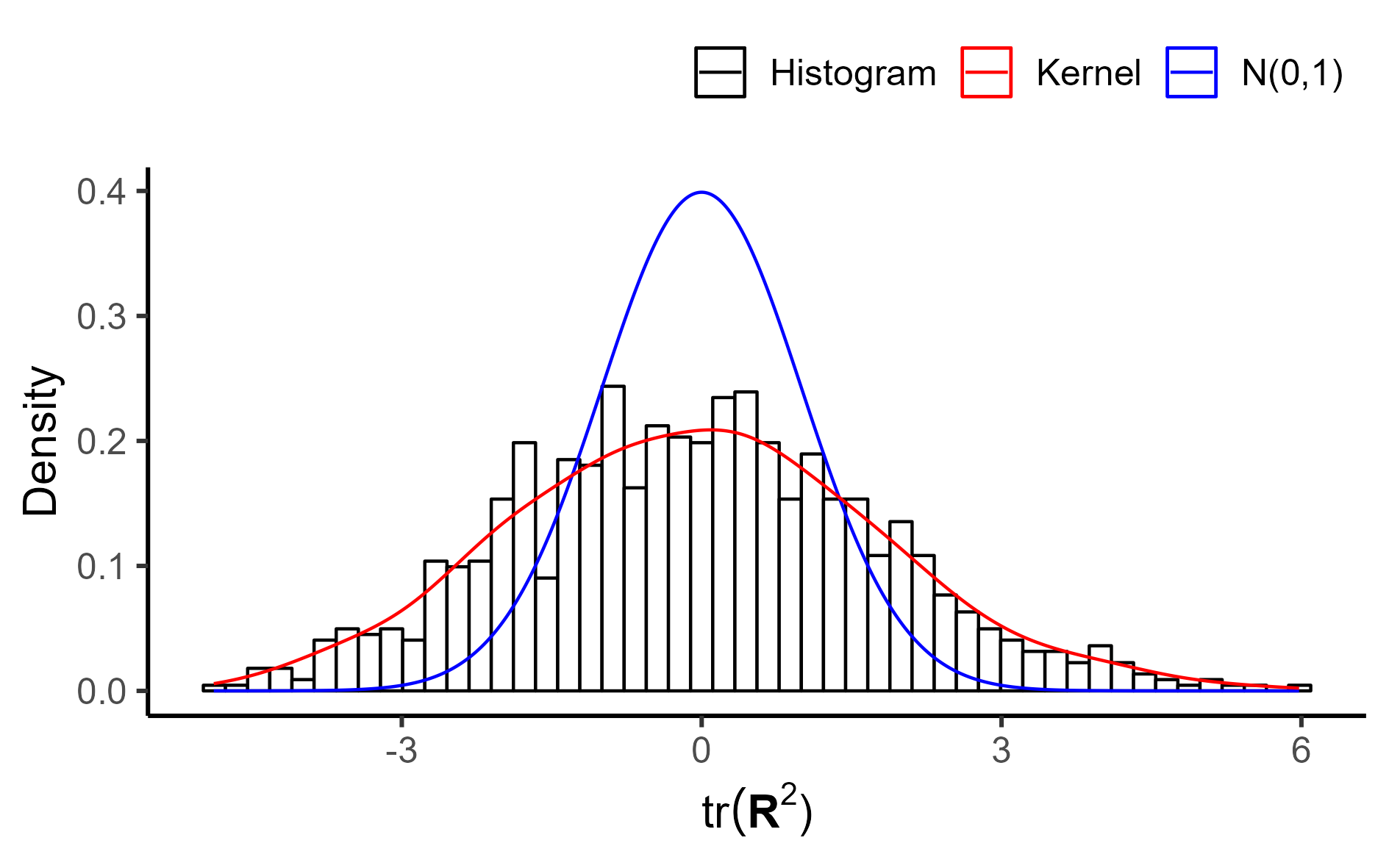}
  \caption{$\alpha = 3.5$}
\end{subfigure}
\caption{Simulations of the central limit theorem for $\tr(\bfR^2)$ for different values of $\alpha$ and $p = 400$, $n = 1000$ with 1000 repetitions. Here the plots in the left column have $X_{11} \eid Z_1 - \E[Z_1]$, where $Z_1$ is Pareto-distributed with parameter $\alpha$, while the plots in the right column have $X_{11} \eid T^2 - \E[T^2]$, where $T$ is t-distributed with $2\alpha$ degrees of freedom.}
\label{fig:nonsym_histograms}
\end{figure}
Some major differences arise when $X_{11}$ does not have a symmetric distribution. To start with, from \eqref{eq:derivmean} we get
\begin{equation*}
    \begin{aligned}
    \E[\tr(\bfR^2)]&= \mu_n + \sum_{\substack{i_1,i_2 = 1 \\ i_1 \neq i_2}}^p \sum_{\substack{t_1,t_2 = 1 \\ t_1 \neq t_2}}^n 
    \E\big[Y_{i_1 t_1} Y_{i_1 t_2} Y_{i_2 t_1} Y_{i_2 t_2}\big]\\
    &= \mu_n + p(p-1)n(n-1) (\E [ Y_{11} Y_{12} ])^2\\
    &= \mu_n + p(p-1)n(n-1) \beta^2_{1,1}\,,
    \end{aligned}
\end{equation*}
where we used the notation$\beta_{k_1,\ldots, {k_r}}:=\E[Y_{11}^{k_1}Y_{12}^{k_2} \cdots Y_{1r}^{k_r}]$ for positive integers $k_1,\ldots, k_r$. For all $y, \beta >0$, we have
\begin{equation}\label{formula_inv}
    \frac{1}{y^\beta} = \frac{1}{\Gamma(\beta)} \int\limits_0^\infty 
    \exp(-ty) t^{\beta -1} \dint t\,, 
\end{equation}
where $\Gamma$ denotes the Gamma function. Combining this representation with Fubini's theorem, we deduce
\begin{equation}\label{eq:EY1Y2}
   \beta_{1,1}  =  \int\limits_0^\infty 
    \big( \E \left[ X_{11} \exp(-sX_{11}^2) \right] \big)^2 \varphi^{n-2} (s) ds ,
\end{equation}
where $\varphi(s) = \E [ \exp(-s X_{11}^2) ], s>0$, denotes the Laplace transform of $X_{11}^2$. We conclude that $\beta_{1,1}\ge 0$ with equality if and only if the distribution of $X_{11}$ is symmetric. In general, the exact calculation of $\beta_{1,1}$ from \eqref{eq:EY1Y2} is rather involved and highly dependent on the specific non-symmetric distribution at hand.  Assuming $\E[|X_{11}|^{\eta}]<\infty$ for some $\eta>1$, Lemma~4.2 in \cite{doernemann:heiny:2025} guarantees that $\beta_{1,1} =o\big(n^{-\eta'}\big)$ for any $\eta' < \eta$. This implies that
\begin{equation*}
    \E[\tr(\bfR^2)] = \mu_n + p^2 \, n^{-2(\eta'-1)} \,o(1)\,.
\end{equation*}
In the special case $\limsup_{\nto} p/n<\infty$, we conclude that $\E[\tr(\bfR^2)]=\mu_n+o(1)$ if $\E[|X_{11}|^{\eta}]<\infty$ for some $\eta>2$. This partially explains the shift in mean in our simulations for the non-symmetric case with small values of $\alpha$.

\medskip
Finally, we would like to point out that $\Var[\tr(\bfR^2)]$ in the non-symmetric case includes many additional terms, which are determined by $\beta_{\text{odd}}$ terms. Here $\beta_{\text{odd}}$ stands for a $\beta_{k_1,\ldots, {k_r}}$ with at least one odd $k_i$. For non-symmetric distributions of $X_{11}$, we have as $\nto$ that
\begin{equation*}
    \begin{aligned}
    \Var[\tr(\bfR^2)] &\sim  \beta_{2,2,2,2} n^4 p + \beta_{2,2}^2 n^4 p^2 + 2 \beta_{1,1,1,1}^2 n^4 p^2 + 2 \beta_{1,1}^2 \beta_{2,2} n^4 p^3 + 4 \beta_{1,1}^2 \beta_{1,1,1,1} n^4 p^3 \\
    &\quad + \beta_{1,1}^4 n^4 p^4 + 6 \beta_{4,2,2} n^3 p + 4 \beta_{2,2}^2 n^3 p^2 + 12 \beta_{2,1,1}^2 n^3 p^2 + 2 \beta_4 \beta_{2,2} n^3 p^2 \\
    &\quad + 4 \beta_2 \beta_{2,2,2} n^3 p^2 + 16 \beta_{1,1} \beta_{3,2,1} n^3 p^2 + 4 \beta_{1,1} \beta_{4,1,1} n^3 p^2 + 2 \beta_4 \beta_{1,1}^2 n^3 p^3 \\
    &\quad + 2 \beta_2^2 \beta_{2,2} n^3 p^3 + 8 \beta_{1,1}^2 \beta_{2,2} n^3 p^3 + 16 \beta_{1,1}^2 \beta_{2,1,1} n^3 p^3 + 4 \beta_{1,1}^4 n^3 p^4 \\
    &\quad + 8 \beta_2 \beta_{1,1} \beta_{2,1,1} n^3 p^3 + 2 \beta_2^2 \beta_{1,1}^2 n^3 p^4 + \beta_8 n p + 3 \beta_{4,4} n^2 p + 4 \beta_{6,2} n^2 p \\
    &\quad + 3 \beta_4^2 n p^2 + \beta_4^2 n^2 p^2 + 8 \beta_{2,2}^2 n^2 p^2 + 8 \beta_{3,1}^2 n^2 p^2 + 4 \beta_2 \beta_6 n p^2 \\
    &\quad + 6 \beta_2^2 \beta_4 n p^3 + 4 \beta_4 \beta_{2,2} n^2 p^2 + 12 \beta_2 \beta_{4,2} n^2 p^2 + 8 \beta_{1,1} \beta_{3,3} n^2 p^2 \\
    &\quad + 8 \beta_{1,1} \beta_{5,1} n^2 p^2 + 2 \beta_2^2 \beta_4 n^2 p^3 + 4 \beta_4 \beta_{1,1}^2 n^2 p^3 + 8 \beta_2^2 \beta_{2,2} n^2 p^3 \\
    &\quad + 12 \beta_{1,1}^2 \beta_{2,2} n^2 p^3 + 16 \beta_2 \beta_{1,1} \beta_{3,1} n^2 p^3 + 4 \beta_2^2 \beta_{1,1}^2 n^2 p^4 + \beta_2^4 n p^4 + 2 \beta_{1,1}^4 n^2 p^4.
    \end{aligned}
\end{equation*}
In the symmetric case, these $\beta_{\text{odd}}$ terms vanish, resulting in a simpler variance structure; see \eqref{eq:defsigma} and Lemma~\ref{lem:asymp_variance} for details. However, in the non-symmetric case, they remain and can contribute in an unpredictable manner. In particular, these terms are not necessarily positive, meaning that their inclusion could lead to a reduction in the overall variance. This contrasts with the symmetric case, where all variance contributions are non-negative. Through numerical simulations and plots, we observe that in the non-symmetric setting, the variance of $\tr(\bfR^2)$ tends to decrease, highlighting the effect of these additional terms in finite samples.

\subsection{Finite-sample accuracy of the $\beta_4$ approximation}
Our asymptotic results and their application for finite-size samples hinge on $\beta_4$. For example, equations \eqref{eq:var} shows that $\Var(\tr(\bfR^2))$ only depends on $p,n$ and $\beta_4$.

In the context of the log-determinant, \cite{heiny:parolya:2024} observed in simulations that $\alpha>3$ and symmetry seem to be relevant, whereas \cite{li:pan:xie:wang:2024} argue that symmetry is not crucial and that the issue is the slow convergence of $\beta_4$ and other self-normalized moments. Both references, however, restrict themselves to the case $p/n\to \gamma \in (0,\infty)$ and do not consider a general growth of $p$ for which the task of controlling $\beta_{\text{odd}}$ terms becomes significantly more complicated. Indeed, we also observe a seemingly misspecified variance in some of our simulations. This issue could sometimes be fixed if the asymptotic version of $\beta_4$ was replaced by an empirical estimate. 

A key observation is that in the asymptotic expansion of the limiting variance, the dominant term is determined by $T_1$ when $0 < \alpha < 3$, which directly depends on the value of $\beta_4$. This suggests that the accuracy of $\beta_4$ is crucial for obtaining the correct variance. Our simulation results indicate that when $\beta_4$ is estimated empirically from the data, the standardized statistic aligns well with the normal distribution. However, when using the asymptotic expression for $\beta_4$, we observe a systematic deviation from the theoretical limiting distribution, specifically an underestimation of the variance.  

To analyze this phenomenon quantitatively, we compare the asymptotic version of $\beta_4$ based on Lemma~\ref{lem:allmoments} called $\beta_{4,\text{asy}}$ with its empirical estimate
\begin{equation*}
\hat{\beta}_4:= \frac{1}{pn} \sum_{i=1}^p \sum_{t=1}^n Y_{it}^4
\end{equation*}
for various values of $\alpha$ and $n$ to study the speed of convergence. 
More precisely, we plot the ratios $\beta_{4,\text{asy}}/\hat{\beta}_4$ for $n=10^{1.5},10^2,10^{2.5},10^3, 10^{3.5},10^4$. The distribution of $X_{11}$ is as in Section~\ref{sec:simulations} and we investigate $\alpha\in (3,4)$, $\alpha\in (2,3)$, $\alpha\in (0,2)$, respectively.
\begin{figure}
\centering
\includegraphics[scale=0.7]{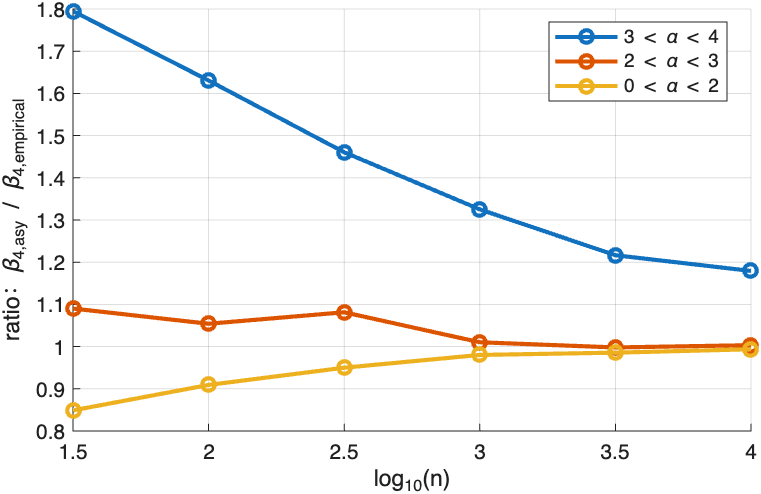}
\caption{Ratio $\beta_{4,\mathrm{asy}}/\hat{\beta}_4$ as a function of the sample size $n$ for different values of the tail index $\alpha$. The three panels correspond to the regimes $\alpha\in(3,4)$, $\alpha\in(2,3)$, and $\alpha\in(0,2)$, respectively.}
\label{fig:beta4_ratio}
\end{figure}
Figure~\ref{fig:beta4_ratio} shows that the finite-sample accuracy of the asymptotic approximation of $\beta_4$ depends strongly on the tail index $\alpha$.  For $3<\alpha<4$, the ratio $\beta_{4,\mathrm{asy}}/\hat{\beta}_4$ remains relatively far from one for moderate sample sizes and approaches one only gradually as $n$ increases, indicating slow convergence in this regime. Luckily $\beta_4$ asymptotically does not contribute to the variance $\sigma_n^2$ which alleviates the slow convergence issue. In contrast, for $2<\alpha<3$, the ratio stays close to one throughout the range of sample sizes considered, suggesting that the asymptotic approximation is already reasonably accurate.  For $0<\alpha<2$, the ratio moves steadily towards one as $n$ grows, with the discrepancy becoming small for larger sample sizes.

Overall, these results show that the convergence of $\beta_4$ to its asymptotic form can vary substantially across different tail-index regimes. They therefore highlight the importance of accounting for the finite-sample behavior of $\beta_4$ when constructing variance approximations. In regimes where the convergence is slow, replacing the asymptotic expression by the empirical estimate $\hat{\beta}_4$ can improve the agreement between the theoretical and empirical distributions. 
\FloatBarrier

\section{Proofs of the main theorems}\label{sec:mainproof}\setcounter{equation}{0}
We will need a few results about the entries of $\Y$ defined in \eqref{def:R}. For integers $k_1,\ldots,k_r$, we recall the notation  $\beta_{2k_1,\ldots, {2k_r}}:=\E [ Y_{11}^{2k_1} \cdots Y_{1r}^{2k_r} ]$ and $\tbeta_{2k_1,\ldots, {2k_r}}:=\E [ (Y_{11}^2-n^{-1})^{k_1} \cdots (Y_{1r}^2-n^{-1})^{k_r} ]$. The following lemma is a special case of \cite[Lemma~3.2]{heiny:parolya:2024}.
\begin{lemma}\label{lem:betas}
For any distribution of $X_{11}$, it holds that $\beta_2=1/n$ and
\begin{equation}
\beta_4 = \frac{1}{n}-(n-1)\beta_{2,2}.
\end{equation}
\end{lemma}
It is easy to see that $\tbeta_2=0$, and we also need the exact asymptotic behavior of  $\tbeta_{2,2}$. Using Lemma ~\ref{lem:betas}, we get $\beta_{2,2} = \frac{1 - n\beta_4}{n(n-1)}$, which yields for $\alpha\in (0,4)$
\begin{equation} \label{eq:tbeta_2_2_rewritten}
    \tbeta_{2,2} = \beta_{2,2} - n^{-2} = \frac{1}{n^2(n-1)} - \frac{\beta_4}{n-1} \sim - n^{-1} \beta_4\,, \qquad \nto\,.
\end{equation}

\subsection{Proof of Theorem~\ref{thm:clt2}}
It will be convenient to write our statistics as a sum of  martingale differences. In what follows, the notation 
\begin{equation*}
    \tx_i=(X_{i1}, \ldots, X_{in})\,, \qquad  i\in\{1,\ldots,p\}
\end{equation*}
will be helpful. We consider the filtration $(\mathcal{F}_{j})_{j\ge 0}$, where $\mathcal{F}_{j}$ is the $\sigma$-algebra generated by $\{\tx_1,\ldots,\tx_j\}$. For $j=1,\ldots, p$, define 
\begin{equation*}
    M_{j,1}:=\E[T_1|\mathcal{F}_j]- \E[T_1|\mathcal{F}_{j-1}] \quad \text { and } \quad M_{j,2}:=\E[T_2|\mathcal{F}_j]- \E[T_2|\mathcal{F}_{j-1}]
\end{equation*}
and observe that $T_i=\sum_{j=1}^p M_{j,i}$ for $i=1,2$. 

Then $(M_{j,1})_{j\ge 1}$, $(M_{j,2})_{j\ge 1}$ and $(M_{j,3})_{j\ge 1}:=(M_{j,1}+M_{j,2})_{j\ge 1}$ are martingale difference sequences with respect to the filtration $(\mathcal{F}_{j})_{j\ge 0}$. The following lemma provides explicit formulas for $M_{j,1}$ and $M_{j,2}$. 
\begin{lemma}\label{lem:martingale_differences}
   For $j=1,\ldots, p$, it holds 
    \begin{equation*}
        M_{j,1} = 2\sum_{i = 1}^{j-1} \sum_{t = 1}^n (Y_{i t}^2 - n^{-1})(Y_{j t}^2 - n^{-1}) \quad \text{ and }
        \quad
        M_{j, 2} = 4\sum_{i = 1}^{j-1} \sum_{1\le t_1 < t_2\le n} Y_{i t_1}Y_{i t_2}Y_{j t_1}Y_{j t_2}.
    \end{equation*}
\end{lemma} 
\begin{proof}
We start by considering $M_{j,1}$.  Using the definition of $T_1$, we have for $k\in\{0,\ldots,p\}$  
\begin{equation*}
    \E[T_1| \mathcal{F}_k] = 
    2\, \sum_{1\le i_1 < i_2\le p}  \sum_{t = 1}^n \E\big[(Y_{i_1 t}^2 - n^{-1})(Y_{i_2 t}^2 - n^{-1})\big| \mathcal{F}_k\big]\,.
\end{equation*}
Since $Y_{it}^2-n^{-1}$ is centered and independent of $\mathcal{F}_k$ for $k < i$, and $\mathcal{F}_k$-measurable for $k \geq i$ we have  $\E\big[(Y_{i_1 t}^2 - n^{-1})(Y_{i_2 t}^2 - n^{-1})\big| \mathcal{F}_k\big]=0$ if $i_2>k$. Therefore, we conclude that 
\begin{equation*}
    \E[T_1| \mathcal{F}_k] = 
    2\, \sum_{1\le i_1 < i_2\le k}  \sum_{t = 1}^n (Y_{i_1 t}^2 - n^{-1})(Y_{i_2 t}^2 - n^{-1})\,,
\end{equation*}
from which we easily deduce 
\begin{equation*}
    \begin{aligned}
    M_{j, 1}
    &=
    \E[T_1|\mathcal{F}_j]- \E[T_1|\mathcal{F}_{j-1}]
    \\ 
    &=
    2\, \bigg( \sum_{1\le i_1 < i_2\le j} -\sum_{1\le i_1 < i_2\le j-1} \bigg)  \sum_{t = 1}^n (Y_{i_1 t}^2 - n^{-1})(Y_{i_2 t}^2 - n^{-1})\\
    &= 2\sum_{i = 1}^{j-1} \sum_{t = 1}^n (Y_{i t}^2 - n^{-1})(Y_{j t}^2 - n^{-1})\,.
    \end{aligned}
\end{equation*}
The proof for $M_{j,2}$ is completely analogous. 
\end{proof}
We will use the following CLT for martingale differences.
\begin{lemma}[e.g. Hall and Heyde \cite{hall:heyde:1980}]\label{lem:martingaleclt}
Let $\{S_{ni}=\sum_{j=1}^{i}Z_{nj},\mathcal{F}_{ni}, 1\le i\le k_n, n\ge 1\}$ be a zero-mean, square integrable martingale array with differences $Z_{nj}$. Suppose that $\E[\max_{j=1,\ldots, k_n} Z_{nj}^2]$ is bounded in $n$ and that
\begin{equation*}
    \max_{j=1,\ldots, k_n} |Z_{nj}|\cip 0 \quad \text{ and } \quad \sum_{j=1}^{k_n} Z_{nj}^2 \cip 1\,.
\end{equation*}
Then we have $S_{nk_n}\cid N(0,1)$ as $\nto$.
\end{lemma}
For convenience of notation, we set $T_3=T_1+T_2$. In view of \eqref{remark:var_asymp}, we will apply Lemma~\ref{lem:martingaleclt} to $\sigma_n^{-1} T_i = \sum_{j=1}^p \sigma_n^{-1} M_{j,i}$ with martingale differences $\sigma_n^{-1} M_{j,i}$, by considering $i = 1$ when $\alpha < 3$, $i = 2$ when $\alpha \in (3, 4)$ and $i = 3$ when $\alpha =3$. By definition, we have $\sigma_n^{-1}\E[T_i] = 0$ and \eqref{eq:dsf245d} implies that $\sigma_n^{-2} \E[T_i^2] \to 1$, fulfilling the square integrability condition. Also note that $\E[\max_{j=1,\ldots,p} \sigma_n^{-2} M_{j,i}^2]$ is bounded since 
\begin{equation*}
    \E[\max_{j=1,\ldots,p} \sigma_n^{-2} M_{j,i}^2] 
    \leq
    \sigma_n^{-2}\sum_{j=1}^p \E[M_{j,i}^2]
    =
    \frac{\E[T_i^2]}{\sigma_n^2}
    \to
    1,\qquad \nto .
\end{equation*}
Note that in the case $\alpha \in(0,3)$, we have assumed in Theorem~\ref{thm:clt2} that $p = \omega(n^\delta)$ for some $\delta > \delta^*(\alpha)$ with $\delta^*(\alpha)$ defined in \eqref{eq:delta_star}. The two remaining conditions in Lemma~\ref{lem:martingaleclt} are
\begin{equation}\label{eq:left}
    \max_{j=1,\ldots,p} |\sigma_n^{-1} M_{j,i}| \cip 0
    \quad \text{and} \quad
    \sigma_n^{-2}\sum_{j = 1}^p M_{j,i}^2 \cip 1,
    \qquad i=1,2,3\,,
\end{equation}
whose verification is carried out in Section~\ref{sec:lefttoprove}. This concludes the proof of Theorem~\ref{thm:clt2}.

\subsection{Verification of martingale CLT conditions}\label{sec:lefttoprove}
We begin with the first condition in \eqref{eq:left}. To prove $\max_j |\sigma_n^{-1} M_{j,i}| \cip 0$, we note that by the union bound and  Markov's inequality
\begin{equation*}
    \P\Big(\max_{j=1,\ldots,p} |\sigma_n^{-1} M_{j,i}| > \varepsilon\Big)
    \leq
    \sum_{j=1}^p \P(|\sigma_n^{-1} M_{j,i}| > \varepsilon)
    \leq
    \varepsilon^{-4} \sum_{j=1}^p \frac{\E[M_{j,i}^4]}{\sigma_n^4}\,, \qquad \vep>0.
\end{equation*}
The following lemma and \eqref{eq:dfdl} assert that the \rhs~tends to zero. 
\begin{lemma} \label{lem:martingale_E[M_j,k^4]_to_0}
If $\alpha \in (0, 3]$ and $p = \omega(n^\delta)$ for some $\delta > \delta^*(\alpha)$, where $\delta^*(\alpha)$ is defined in \eqref{eq:delta_star}, we have
\begin{equation}\label{eq:sigma^(-4)E[M_j,k^4]}
    \frac{1}{\sigma_n^4}\sum_{j=1}^p\E[M_{j,1}^4]\to 0\,, \qquad \nto\, .    
\end{equation}
If $\alpha \in [3, 4)$, we have
\begin{equation}\label{eq:sigma^(-4)E[M_j,k^4]2}
    \frac{1}{\sigma_n^4}\sum_{j=1}^p\E[M_{j,2}^4]\to 0\,, \qquad \nto\, .  
\end{equation}
\end{lemma}
Note that in the case $\alpha = 3$ (corresponding to $i=3$), the convergence $ \sigma_n^{-4} \sum_{j = 1}^p \E[(M_{j, 1} + M_{j, 2})^4] \to 0$ follows directly since 
\begin{equation}\label{eq:dfdl}
    \begin{aligned}
    \frac{1}{\sigma_n^4} \sum_{j = 1}^p \E[(M_{j, 1} + M_{j, 2})^4] 
    \leq\frac{16}{\sigma_n^4} \sum_{j = 1}^p \E[M_{j, 1}^4]
    + \frac{16}{\sigma_n^4} \sum_{j = 1}^p \E[M_{j, 2}^4]
    \end{aligned}
\end{equation}
and we deduce from Lemma~\ref{lem:martingale_E[M_j,k^4]_to_0} that the above tends to zero if $\delta > \delta^*(\alpha)$ for $\alpha = 3$.
\begin{proof}[Proof of Lemma~\ref{lem:martingale_E[M_j,k^4]_to_0}]
Direct calculation yields 
\begin{equation*}
    \begin{aligned}
    \E[M_{j,1}^4]
    &=
    \E\bigg[\bigg(2\sum_{i=1}^{j-1}\sum_{t=1}^n\bar Y_{it} \bar Y_{jt}\bigg)^4\bigg]
    \\ &=
    16\!\!\! \sum_{i_1,i_2,i_3,i_4=1}^{j-1}  \sum_{t_1,t_2,t_3,t_4=1}^n 
    \E[\bar Y_{i_1t_1} \bar Y_{i_2t_2}\bar Y_{i_3t_3} \bar Y_{i_4t_4}] \E[\bar Y_{jt_1} \bar Y_{jt_2}\bar Y_{jt_3} \bar Y_{jt_4}] \,.
    \end{aligned}
\end{equation*}
To get nonzero summands we have two possibilities to pair the $i$'s. Either all are equal, $i_1 = i_2 = i_3 = i_4$, or we get two pairs, e.g., $i_1 = i_2$ and $i_3 = i_4$ but $i_1 \neq i_3$. 
\begin{table}[H]
    \caption{Possible terms of $\tbeta_{2k_1,\dots , 2k_r}$ that occur in $\E[M_{j, 1}^4]$. We include the case where all $i$'s equal ($i_1 = i_2 = i_3 = i_4$) and the case where two pairs are formed between the $i$'s (e.g. $i_1 = i_2$ and $i_3 = i_4$). Within the two pairs case, different pairings of the $t$-indices can lead to distinct factorizations. e.g. $t_1 = t_2$, $t_3 = t_4$ and $t_1 = t_3$, $t_2 = t_4$ will yield different results.}
    \label{table_2}
    \centering
    \renewcommand{\arraystretch}{1.1}
    \begin{tabular}{|c||c|c|}
    \hline
    \multicolumn{1}{|c||}{$\E[\bar{Y}_{j t_1} \bar Y_{j t_2} \bar{Y}_{j t_3} \bar{Y}_{j t_4}]$} 
    & 
    \multicolumn{1}{c|}{$\E[\Bar{Y}_{i_1 t_1} \Bar{Y}_{i_1 t_2} \bar{Y}_{i_1 t_3}\Bar{Y}_{i_1 t_4}]$}  
    &
    \multicolumn{1}{c|}{$\E[\Bar{Y}_{i_1 t_1} \Bar{Y}_{i_1 t_2}] \E[\Bar{Y}_{i_2 t_3}\Bar{Y}_{i_2 t_4}]$} \\
    \hline
    \textbf{} & \textbf{Case: All $i$'s equal} & \textbf{Case: Two pairs between $i$'s} \\
    \hline
    $\tbeta_8$ & $\tbeta_8$ & $\tbeta_4^2$  \\
    \hline
    $\tbeta_{6,2}$ & $\tbeta_{6,2}$ & $\tbeta_4 \tbeta_{2,2}$ \\
    \hline
    $\tbeta_{4,4}$ & $\tbeta_{4,4}$ & $\tbeta_{2,2}^2$ or $\tbeta_4^2$ \\
    \hline
    $\tbeta_{4,2,2}$ & $\tbeta_{4,2,2}$ & $\tbeta_4 \tbeta_{2,2}$ or $\tbeta_{2,2}^2$ \\
    \hline
    $\tbeta_{2,2,2,2}$ & $\tbeta_{2,2,2,2}$ & $\tbeta_{2,2}^2$ \\
    \hline
    \end{tabular}
\end{table}
Table~\ref{table_2} shows the possible (nonzero) terms in $\E[M_{j,1}^4]$. Each term is a product of certain $\tbeta_{2k_1,\dots ,2k_r}$  and - to get an upper bound on the order of its contribution to $\sum_{j=1}^p \E[M_{j,1}^4]$ - should be multiplied by a factor $p^{h+1} n^d$, where $d$ denotes the number of distinct $t$'s. The $p^{h+1}$ factor comes from the number $h$ of distinct $i$'s and then summing from $j = 1$ to $p$ in \eqref{eq:sigma^(-4)E[M_j,k^4]}. Thus, using Lemma~\ref{lem:allmoments} combined with Lemma~\ref{lem:tbeta} one can check that the highest order term for $\alpha \in [2, 3]$ is $\tbeta_8^2 np^2$, which (up to a slowly varying function) behaves like $n^{-\alpha + 1} p^2$.  Now using Lemma~\ref{lem:asymp_variance} for $\alpha \in [2, 3]$ we get 
\begin{equation*}
    \frac{\tbeta_8^2 np^2}{\sigma_n^4} \slv n^{\alpha - 1}p^{-2} 
\end{equation*}
which goes to zero if $p = \omega(n^\delta)$ for some $\delta > \delta^*(\alpha)$ with $\delta^*(\alpha)$ defined in \eqref{eq:delta_star}. 
Similarly, one can find for the other case with $i_1 = i_2$ and $i_3 = i_4$, but $i_1\neq i_3$, that the dominating term is $\tbeta_8 \tbeta_{4}^2 n^2 p^3$. Combining Lemmas~\ref{lem:asymp_variance} and~\ref{lem:allmoments} one can find that $\tbeta_8 \tbeta_{4}^2 n^2 p^3 / \sigma_n^4 \to 0$ if $p = \omega(n^\delta)$ for some $\delta > \alpha/2 - 1$ with $\alpha \in [2, 3]$.

Now for $\alpha \in (0, 2)$ the same terms will dominate in the case where all $i$'s are equal and here $\beta_8^2 np^2 \sim n^{-1}p^2 c_4^2(\alpha)$ and $\sigma_n^4 \sim 4p^4 n^{-2} (1 - \alpha / 2)^4$ by Lemmas~\ref{lem:allmoments} and~\ref{lem:asymp_variance}, respectively. Thus it is easy to see that $\beta_8^2 np^2/\sigma_n^4 \to 0$ if $p = \omega(n^\delta)$ for some $\delta > \delta^*(\alpha)$. In the case $i_1 = i_2$ and $i_3 = i_4$, but $i_1\neq i_3$, one of $\tbeta_8 \tbeta_4^2 n p^3$ and $\tbeta_{4,4}\tbeta_4^2 n^2 p^3$ is the dominating term but still $\sigma_n^4$ is of higher order. This completes the proof of \eqref{eq:sigma^(-4)E[M_j,k^4]}.
\medskip

Regarding \eqref{eq:sigma^(-4)E[M_j,k^4]2} we proceed similarly. By direct calculation we have
\begin{equation*}
    \begin{aligned}
    \E[M_{j,2}^4]
    &=
    \E\bigg[\bigg(4\sum_{i=1}^{j-1}\sum_{\substack{t_1,t_2=1 \\ t_1 < t_2}}^n Y_{it_1}Y_{it_2} Y_{jt_1}Y_{jt_2}\bigg)^4\bigg]
    \\ &=
    256\!\!\!\! \sum_{i_1,\ldots,i_4=1}^{j-1}\!\!   \sum_{\substack{t_1,\ldots,t_8=1 \\ t_1 < t_2, \ldots, t_7 < t_8}}^n 
    \!\!\!\!\E[Y_{i_1t_1}Y_{i_1t_2}Y_{i_2t_3}Y_{i_2t_4}
    Y_{i_3t_5}Y_{i_3t_6}Y_{i_4t_7}Y_{i_4t_8}]\\
    &\quad \times
    \E[Y_{jt_1}Y_{jt_2}Y_{jt_3}Y_{jt_4}Y_{jt_5}Y_{jt_6}Y_{jt_7}Y_{jt_8}]
    \end{aligned}
\end{equation*}
and as before in order to get nonzero expectation we want no odd powers; so either all the $i$'s are equal or we establish four pairs between the $i$'s. Table~\ref{table_3} shows the possible terms that occur in $\E[M_{j, 2}^4]$.
\begin{table}[H]
  \caption{Possible terms of $\beta_{2k_1,\dots , 2k_r}$ that occur in $\E[M_{j, 2}^4]$. Here we both include the case where all $i$'s equal ($i_1 = i_2 = i_3 = i_4$) and the case where four pairs are formed between the $i$'s (e.g. $i_1 = i_2$ and $i_3 = i_4$).}
  \label{table_3}
  \centering
  \renewcommand{\arraystretch}{1.1}
  \begin{tabular}{|c||c|c|}
    \hline
    \multicolumn{1}{|c||}{$\E[Y_{jt_1} \cdots Y_{jt_8]}$} 
    & 
    \multicolumn{1}{c|}{$\E[Y_{i_1t_1}Y_{i_1t_2} \cdots Y_{i_1t_7}Y_{i_1t_8}]$}  
    &
    \multicolumn{1}{c|}{$\E[Y_{i_1t_1}Y_{i_1t_2}Y_{i_1t_3}Y_{i_1t_4}]\E[Y_{i_2t_5}Y_{i_2t_6}Y_{i_2t_7}Y_{i_2t_8}]$} \\
    \hline
    \textbf{} & \textbf{Case: All $i$'s equal} & \textbf{Case: Four pairs between $i$'s}  \\
    \hline
    $\beta_{4,4}$ & $\beta_{4,4}$ & $\beta_{2,2}^2$  \\
    \hline
    $\beta_{4,2,2}$ & $\beta_{4,4,2}$ & $\beta_{2,2}^2$  \\
    \hline
    $\beta_{2,2,2,2}$ & $\beta_{2,2,2,2}$ & $\beta_{2,2}^2$  \\
    \hline
  \end{tabular}
\end{table}
The highest order term in the case where all $i$'s are equal is $\beta_{2,2,2,2}^2n^4 p^2$ if $\alpha \in (3, 4)$ and combining Lemmas~\ref{lem:allmoments} and~\ref{lem:asymp_variance} we get
\begin{equation*}
    \frac{\beta_{2,2,2,2}^2n^4 p^2}{\sigma_n^4}
    \sim 
    \frac{n^{-4} p^2 c_{1, 1, 1, 1}^2(\alpha)}{16p^4 n^{-4}}
    \to
    0.
\end{equation*}
In the case $\alpha = 3$ then all the terms will have the same order and the squared variance will instead become $\sigma_n^4 \sim 4p^4n^{-4}(c_2^2(\alpha) L^2(n^{1/2}) + 2)^2$, so that the same conclusion holds. Now in the case of four pairs among the $i$'s and $\alpha \in [3, 4)$, the term $\beta_{2,2,2,2}\beta_{2,2}^2 n^4 p^3 \sim n^{-4}p^3$ dominates. Using Lemma~\ref{lem:asymp_variance} we have that $\sigma_n^4  \sim 16p^4n^{-4}$ for $\alpha \in (3, 4)$ and it is obvious that $\beta_{2,2,2,2} \beta_{2,2}^2 n^4 p^3 / \sigma_n^4 \to 0$ when $\alpha \in [3, 4)$. This finishes the proof of \eqref{eq:sigma^(-4)E[M_j,k^4]2}.
\end{proof}

Now we turn to verify the second condition of \eqref{eq:left}.
\begin{lemma} \label{lem:M_jk^2_cip_1}
If $\alpha \in (0, 3)$, it holds
\begin{equation} \label{eq:M_j1^2_cip_1}
    \frac{1}{\sigma_n^2}\sum_{j=1}^p M_{j,1}^2 \cip 1\,, \qquad \nto\, , 
\end{equation}
where for $\alpha \in [2, 3)$ we additionally assume $p = \omega(n^\delta)$ for some $\delta > \alpha/2 - 1$. If $\alpha \in (3, 4)$, then it holds
\begin{equation} \label{eq:M_j2^2_cip_1} 
    \frac{1}{\sigma_n^2}\sum_{j=1}^p M_{j,2}^2 \cip 1\,, \qquad \nto .   
\end{equation}
Finally, for $\alpha = 3$ it holds
\begin{equation} \label{eq:M_j1M_j2^2_cip_1}
    \frac{1}{\sigma_n^2}\sum_{j=1}^p M_{j,3}^2 = 
    \frac{1}{\E[T_1^2]+\E[T_2^2]}\sum_{j=1}^p \big( M_{j,1}^2+M_{j,2}^2 + M_{j,1}M_{j,2} \big) \cip 1\,.    
\end{equation}
\end{lemma}
\begin{proof}
We start by first proving \eqref{eq:M_j1^2_cip_1} and \eqref{eq:M_j2^2_cip_1} for $\alpha \in (0, 3)$ and $\alpha \in (3, 4)$ respectively. Since $\sigma_n^{-2}\sum_{j=1}^p \E[M_{j,k}^2] \to 1$ for $k \in \{1,2\}$, it is sufficient to show 
\begin{equation*}
    \frac{1}{\sigma_n^4} \E\bigg[\bigg(\sum_{j = 1}^p M_{j,k}^2\bigg)^2 \bigg] \to 1, \qquad \nto\,.
\end{equation*}
Using Lemma~\ref{lem:martingale_E[M_j,k^4]_to_0} and Lemma~\ref{lem:martingale_differences}, we have
\begin{equation} \label{eq:E[(sum M_j,1^2)^2]}
    \begin{split}
    &\E\bigg[\bigg(\sum_{j = 1}^p M_{j,1}^2\bigg)^2 \bigg]
     = o(\sigma_n^4)+
  \sum_{\substack{j_1,j_2 = 1 \\ j_1 \neq j_2}}^p \E\Big[M_{j_1,1}^2 M_{j_2,1}^2 \Big]
     = o(\sigma_n^4) +\\
  &+  16\sum_{\substack{j_1,j_2 = 1 \\ j_1 \neq j_2}}^p \sum_{i_1,i_2 = 1}^{j_1 - 1} \sum_{i_3,i_4 = 1}^{j_2 - 1} \sum_{t_1,t_2 = 1}^n \sum_{t_3,t_4 = 1}^n
    \E[\bar Y_{i_1t_1} \bar Y_{i_2t_2}\bar Y_{i_3t_3}\bar Y_{i_4t_4}\bar Y_{j_1t_1}\bar Y_{j_1t_2}\bar Y_{j_2t_3}\bar Y_{j_2t_4}].
\end{split}
\end{equation}
For the above expectation to be nonzero three cases are possible (up to a permutation of the $i$ indices). Either all the $i$'s are the same (which implies that the $j$ indices take distinct values than the $i$'s), or exactly 2 pairs are formed, e.g., $i_1 = i_2\neq i_3 = i_4$ which amounts to two cases depending on if some $i$'s may coincide with a $j$ or not. The latter two cases we will call the second and third case respectively. 

In the first case where all $i$'s are equal, the possible terms are summarised in Table~\ref{table_4} using Lemma~\ref{lem:allmoments} and Lemma~\ref{lem:tbeta} when $\alpha \in [2, 3)$. Each term should also be multiplied by a factor $n^d$ where $d$ stands for the number of distinct $t$ indices in $\E[\bar Y_{i_1 t_1} \bar Y_{i_1 t_2} \bar Y_{i_1 t_3} \bar Y_{i_1 t_4}]$ in Table~\ref{table_4}. Note also that in the first case the order of \eqref{eq:E[(sum M_j,1^2)^2]} with respect to $p$ is $p^3$.
\begin{table}[H]
  \caption{Terms of $\tbeta_{2k_1,\dots , 2k_r}$ that occur in \eqref{eq:E[(sum M_j,1^2)^2]} when $i_1=i_2=i_3=i_4$ and their order with respect to $n$.}
  \label{table_4}
  \centering
  \renewcommand{\arraystretch}{1.1}
  \begin{tabular}{|c||c|c|}
    \hline
    \multicolumn{1}{|c||}{$\E[\bar Y_{i_1 t_1} \bar Y_{i_1 t_2} \bar Y_{i_1 t_3} \bar Y_{i_1 t_4}]\E[\bar Y_{j_1 t_1} \bar Y_{j_1 t_2}]\E[\bar Y_{j_2 t_3} \bar Y_{j_2 t_4}]$}  
    &
    \multicolumn{2}{c|}{\textbf{Order}} \\
    \hline
    \textbf{Term} & $\alpha \in (0, 2)$ & $\alpha \in [2, 3)$  \\
    \hline
    $\tbeta_8 \tbeta_4^2 n$ & $n^{-2}$ & $n^{-3\alpha/2 + 1}$ \\
    \hline
    $\tbeta_{4,4}\tbeta_4^2 n^2$ & $n^{-2}$ & $n^{-2\alpha + 2}$  \\
    \hline
    $\tbeta_{6,2}\tbeta_4 \tbeta_{2,2} n^2$ & $O(n^{-3})$ & $O(n^{-\alpha-1})$  \\
    \hline
    $\tbeta_{4,2,2}\tbeta_4 \tbeta_{2,2} n^3$ & $O(n^{-3})$ & $O(n^{-\alpha-1})$ \\
    \hline
    $\tbeta_{4,4}\tbeta_{2,2}^2 n^2$ & $O(n^{-4})$ & $O(n^{-\alpha-2})$ \\
    \hline
    $\tbeta_{4,2,2}\tbeta_{2,2}^2 n^3$ & $O(n^{-4})$ & $O(n^{-\alpha/2 - 3})$ \\
    \hline
    $\tbeta_{2,2,2,2}\tbeta_{2,2}^2 n^4$ & $O(n^{-4})$ & $O(n^{-4})$ \\
    \hline
  \end{tabular}
\end{table}
The highest order term when $\alpha \in [2, 3)$ is $\tbeta_8\tbeta_4^2 n p^3 \sim n^{-3\alpha/2 + 1} p^3 c_4(\alpha) c_2^2(\alpha) L^3(n^{1/2})$ and also $\tbeta_{4,4}\tbeta_4^2 n^2$ which obtains the same order when $\alpha = 2$. By Lemma~\ref{lem:asymp_variance}, $\sigma_n^4 \sim 4p^4n^{2-2\alpha} L^4(n^{1/2}) c_2^4(\alpha)$ which yields $\tbeta_8\tbeta_4^2 n p^3 / \sigma_n^4(\alpha) \to 0$ assuming $p = \omega(n^\delta)$ with $\delta > \alpha/2 - 1$. If $\alpha \in (0, 2)$, then Lemma~\ref{lem:asymp_variance} yields $\sigma_n^4 \sim 4p^4 n^{-2} (1 - \alpha / 2)^4$ which has greater order than all the terms in Table~\ref{table_4}. Thus, we have shown that $\sigma_n^{-4}$ times the contribution of the first case to \eqref{eq:E[(sum M_j,1^2)^2]} tends to zero.

Next, we turn to the second case. In this case, since we assume that the sets of $j$'s and $i$'s are disjoint, we can have either $i_1 = i_2 \neq i_3 = i_4$, or $i_1 = i_3 \neq i_2 = i_4$ that yield different terms (the case $i_1=i_4$ and $i_2 = i_3$ yields the same terms as in $i_1=i_3$ and $i_2 = i_4$ because of symmetry of the $t$'s). This means that the terms in the sum of \eqref{eq:E[(sum M_j,1^2)^2]} can take two forms: 
\begin{equation*}
    \begin{aligned}
    &\E[\bar Y_{i_1 t_1} \bar Y_{i_1 t_2}] \E[\bar Y_{i_3 t_3} \bar Y_{i_3 t_4}]
    \E[\bar Y_{j_1 t_1} \bar Y_{j_1 t_2}] \E[\bar Y_{j_2 t_3} \bar Y_{j_2 t_4}],
    \\
    &\E[\bar Y_{i_1 t_1} \bar Y_{i_1 t_3}] \E[\bar Y_{i_2 t_2} \bar Y_{i_2 t_4}]
    \E[\bar Y_{j_1 t_1} \bar Y_{j_1 t_2}] \E[\bar Y_{j_2 t_3} \bar Y_{j_2 t_4}].
    \end{aligned}
\end{equation*}
The general form of either of these products of expectations is $\tbeta_4^{k_1} \tbeta_{2,2}^{k_2}$ with $k_1 + k_2 = 4$. Since now we allow two pairs of the $i$'s we get that in this case the contribution of such terms to \eqref{eq:E[(sum M_j,1^2)^2]} will be $\tbeta_4^{k_1} \tbeta_{2,2}^{k_2} n^{k_3}p^4$ where $k_3$ stands for how many of the $t$'s that may vary freely, e.g. if $t_1 = t_2 = t_3 = t_4$ then $k_3 = 1$. Using \eqref{eq:tbeta_2_2_rewritten} we find that $\tbeta_{2,2} \sim -n^{-\alpha/2 - 1}L(n^{1/2})c_2(\alpha)$ when $\alpha \in [2, 3)$ which in conjunction with Lemma~\ref{lem:asymp_variance} yields 
\begin{equation*}
    \begin{aligned}
    \frac{\tbeta_4^{k_1} \tbeta_{2,2}^{k_2} n^{k_3}p^4}{\sigma_n^4}
    &\sim
    \frac{(-1)^{k_2} n^{-k_1 \alpha/2} n^{-k_2(\alpha/2 + 1)} n^{k_3}
    p^4 c_2(\alpha)^{k_1 + k_2} L^{k_1 + k_2}(n^{1/2})}
    {4p^4n^{2-2\alpha} L^4(n^{1/2}) c_2^4(\alpha)}\\
    &=
    \frac{(-1)^{k_2} n^{-k_2 - 2 + k_3}}{4}
    \end{aligned}
\end{equation*}
where the last equality follows by $k_1 + k_2 = 4$. If $k_3 = 4$ this means that all the $t$'s are pairwise different and hence $k_2 = 4$. If $k_3 = 3$ then two $t$'s are equal but then $k_2 \geq 2$ and if $k_3 = 1$ then it is obvious that the exponent of $n$ above is negative. Only when $k_3 = 2$ and $k_1 = 4$ with $i_1 = i_2$ and $i_3 = i_4$ is when we may get nonzero results. This results in the term $\tbeta_4^4$. In view of \eqref{eq:E[(sum M_j,1^2)^2]} the asymptotic behavior of $p$ in the second case is given by 
\begin{equation}\label{eq:Mj,1_order_p}
    \sum_{j_1,j_2=1 }^p (j_1 - 1)(j_2 - 1) - \sum_{j = 1}^p (j - 1)^2  
    =
    \frac{1}{4} (p - 1)^2 p^2 - \frac{1}{6} p (2 p^2 - 3p + 1)
    \sim
    \frac{1}{4} p^4
\end{equation}
and now we get using \eqref{eq:E[(sum M_j,1^2)^2]} and Lemma~\ref{lem:tbeta} that
\begin{equation}\label{eq:tbeta_4^4/var^2_to_1}
    16\frac{\frac{1}{4} \tbeta_4^4 n^2 p^4}{\sigma_n^4(\alpha)}
    \sim
    \frac{n^{-2\alpha + 2} p^4 L^4(n^{1/2}) c_2^4(\alpha)}{n^{-2\alpha + 2} p^4 L^4(n^{1/2}) c_2^4(\alpha)} =   1, \qquad \nto.
\end{equation}
When $\alpha \in (0, 2)$ then $\tbeta_4^{k_1} \tbeta_{2,2}^{k_2} n^{k_3}p^4/\sigma_n^4(\alpha) \lesssim n^{-2 - k_2 +k_3}$ which is the same expression analyzed before. The dominating term is again $\tbeta_4^4$ and similarly we get $\tbeta_4^4 n^{2}p^4/\sigma_n^4(\alpha) \to 1$ if $\alpha \in (0, 2)$ by Lemmas~\ref{lem:asymp_variance} and~\ref{lem:allmoments}. This concludes our analysis of the second case.

Now we turn to the third case, where we form two pairs among the $i$'s and have some $i$'s coincide with $j_1$ or $j_2$. If $i_1 = i_2$ and $i_3 = i_4$, then we must have either $j_1 = i_3$ or $j_2 = i_1$. If $j_2 = i_1$, then we get in the sum of \eqref{eq:E[(sum M_j,1^2)^2]}
\begin{equation*}
    E[\bar Y_{i_1 t_1} \bar Y_{i_2 t_2} \bar Y_{j_2 t_3} \bar Y_{j_2 t_4}]
    E[\bar Y_{i_3 t_3} \bar Y_{i_3 t_4}] E[\bar Y_{j_1 t_1} \bar Y_{j_1 t_2}] 
\end{equation*}
which essentially yields the same cases as summarized in Table~\ref{table_4} but with the same or less order in $p$. Now if $i_1 \neq i_2$ then we get nonzero only if $i_2 = i_3 = i_4$ and the cases can be summarized in Table~\ref{table_5} by use of Lemma~\ref{lem:allmoments} and Lemma~\ref{lem:tbeta}.
\begin{table}[H]
  \caption{Terms of $\tbeta_{2k_1,\dots , 2k_r}$ that occur in \eqref{eq:E[(sum M_j,1^2)^2]} when $i_1 = j_2 \neq i_2$ but $i_2=i_3=i_4$ and their order with respect to $n$.}
  \label{table_5}
  \centering
  \renewcommand{\arraystretch}{1.1}
  \begin{tabular}{|c||c|c|}
    \hline
    \multicolumn{1}{|c||}{$\E[\bar Y_{i_1 t_1} \bar Y_{i_1 t_3} \bar Y_{i_1 t_4}]\E[\bar Y_{i_2 t_2} \bar Y_{i_2 t_3} \bar Y_{i_2 t_4}] \E[\bar Y_{j_1 t_1} \bar Y_{j_1 t_2}]$}  
    &
    \multicolumn{2}{c|}{Order} \\
    \hline
    \textbf{Term} & $\alpha \in (0,2)$ & $\alpha \in [2, 3)$ \\
    \hline
    $\tbeta_6^2\tbeta_4 n$ & $n^{-2}$ & $n^{-3\alpha/2 + 1}$ \\
    \hline
    $\tbeta_6 \tbeta_{4,2} \tbeta_{2,2} n^2$ & $O(n^{-3})$ & $O(n^{-\alpha-1})$ \\
    \hline
    $\tbeta_{4,2}^2 \tbeta_{2,2} n^3$ & $O(n^{-3})$ & $O(n^{-\alpha-1})$ \\
    \hline
    $\tbeta_{4,2}^2 \tbeta_4 n^2$ & $O(n^{-3})$ & $O(n^{-3\alpha/2})$ \\
    \hline
    $\tbeta_{2,2,2}^2 \tbeta_4 n^3$ & $O(n^{-4})$ & $O(n^{-\alpha/2 - 3})$ \\
    \hline
    $\tbeta_{4,2} \tbeta_{2,2,2} \tbeta_{2,2} n^3$ & $O(n^{-4}$ & $O(n^{-\alpha/2 - 3})$ \\
    \hline
    $\tbeta_{2,2,2}^2 \tbeta_{2,2} n^4$ & $O(n^{-4}$ & $O(n^{-4})$ \\
    \hline
  \end{tabular}
\end{table}
Since now $j_2 = i_1 \neq i_2$ and $i_2 = i_3 = i_4$ the order of \eqref{eq:E[(sum M_j,1^2)^2]} in $p$ is $p^3$. Hence it is clear that when $\alpha \in (0, 2)$ every term in Table~\ref{table_5} has lower order than $\sigma_n^4 \sim 4p^4n^{-2} (1 - \alpha/2)^4$. When $\alpha \in [2, 3)$ then $\tbeta_6^2\tbeta_4n p^3$ is of highest order but $\tbeta_6^2\tbeta_4n p^3 / \sigma_n^4(\alpha) \to 0$ if $p = \omega(n^\delta)$ for some $\delta > \alpha/2 - 1$. This concludes the third and last case and hence we have completed the proof of \eqref{eq:M_j1^2_cip_1}. 

Finally, we start proving \eqref{eq:M_j2^2_cip_1} and we obtain using Lemma~\ref{lem:martingale_differences} and Lemma~\ref{lem:martingale_E[M_j,k^4]_to_0} that
\begin{equation}\label{eq:E[(sum M_j,2^2)^2]}
   \begin{aligned} 
     &\E\bigg[\bigg(\sum_{j = 1}^p M_{j,2}^2\bigg)^2 \bigg] 
     = o(\sigma_n^4) \\ 
     &+ 16\sum_{\substack{j_1,j_2 = 1 \\ j_1 \neq j_2}}^p \sum_{i_1,i_2=1}^{j_1-1} \sum_{i_3,i_4=1}^{j_2-1} \sum_{\substack{t_1,\ldots, t_8=1 \\ t_1 < t_2, \ldots, t_7 < t_8 }}^n
     \mathbb{E}[Y_{i_1t_1} Y_{i_1t_2} Y_{i_2t_3} Y_{i_2t_4} \cdots Y_{j_2t_5} Y_{j_2t_6} Y_{j_2t_7} Y_{j_2t_8}]\,.
    \end{aligned} 
\end{equation}

We have nonzero expectation in \eqref{eq:E[(sum M_j,2^2)^2]} if either all $i$'s  are equal (which implies that the $j$'s are distinct from the $i$'s) or if we form two pairs among the $i$'s (but here some $i$'s may coincide with $j$'s). Note that when the sets of $i$'s and $j$'s are disjoint, then we must have that $t_1=t_3$, $t_2 = t_4$, $t_5 = t_7$ and $t_6 = t_8$ to get nonzero. Starting with $i_1=i_2=i_3=i_4$ the terms that occur in \eqref{eq:E[(sum M_j,2^2)^2]} can be summarized in Table~\ref{table_6}.
\begin{table}[H]
  \caption{Terms of $\beta_{2k_1,\dots , 2k_r}$ that occur in \eqref{eq:E[(sum M_j,2^2)^2]} when $i_1 = i_2 = i_3 = i_4$ and their order with respect to $n$.}
  \label{table_6}
  \centering
  \renewcommand{\arraystretch}{1.1}
  \begin{tabular}{|c||c|c|}
    \hline
    \multicolumn{1}{|c||}{$\E[Y_{i_1t_1}^2 Y_{i_1t_2}^2 Y_{i_2t_3}^2 Y_{i_2t_4}^2] \E[Y_{j_1t_1}^2 Y_{j_1t_2}^2] \E[Y_{j_2t_5}^2 Y_{j_2t_6}^2]$}  
    &
    \multicolumn{1}{c|}{Order} \\
    \hline
    \textbf{Term} & $\alpha \in (3,4)$  \\
    \hline
    $\beta_{4,4} \beta_{2,2}^2 n^2$ & $n^{-\alpha - 2}$ \\
    \hline
    $\beta_{4,2,2} \beta_{2,2}^2 n^3$ & $n^{-\alpha/2 - 3}$ \\
    \hline
    $\beta_{2,2,2,2} \beta_{2,2}^2 n^4$ & $n^{-4}$ \\
    \hline
  \end{tabular}
\end{table}
Since the order in $p$ is $p^3$ for \eqref{eq:E[(sum M_j,2^2)^2]} we easily see that the terms in Table~\ref{table_6} multiplied by $p^3\sigma_n^{-4}$ go to zero in view of Lemmas~\ref{lem:allmoments} and~\ref{lem:asymp_variance} for $\alpha \in (3, 4)$. Investigating the cases where two pairs are formed, we have three cases with respect to the $i$'s. If either $i_1 = i_3$ or $i_1 = i_4$ then we would need to impose the extra conditions ($t_1 = t_5$ and $t_2 = t_6$) or ($t_1 = t_5$ and $t_2 = t_6$) respectively. This leads to a lower order in $n$ than in the pairing $i_1 = i_2$ and $i_3 = i_4$ and in this case we only get the term $\beta_{2,2}^4 n^4$. The order in $p$ is the same as \eqref{eq:Mj,1_order_p} and we get for $\alpha \in (3, 4)$ that
\begin{equation*}
    16\frac{\frac{1}{4} \beta_{2,2}^4 n^4 p^4}{\sigma_n^4}
    \sim
    \frac{4n^{-4} p^4}{4n^{-4} p^4}=    1, \qquad\ n\to \infty.
\end{equation*}
It is left to consider the overlap between some $i$'s and $j$'s in the case $i_1 = i_2$ and $i_3 = i_4$. If $i_1 = i_2 = j_2$ then we have that the expectation in \eqref{eq:E[(sum M_j,2^2)^2]} is
\begin{equation*}
    \begin{aligned}
    &\E[Y_{i_1t_1} Y_{i_1t_2} Y_{i_1t_3} Y_{i_1t_4} Y_{i_1t_5} Y_{i_1t_6} Y_{i_1t_7} Y_{i_1t_8}]
    \E[Y_{i_3t_5} Y_{i_3t_6} Y_{i_3t_7} Y_{i_3t_8}]
    \E[Y_{j_1t_1} Y_{j_1t_2} Y_{j_1t_3} Y_{j_1t_4}]\\  
    &=\E[Y_{i_1t_1}^2 Y_{i_1t_2}^2 Y_{i_1t_5}^2 Y_{i_1t_6}^2]
    \E[Y_{i_3t_5}^2 Y_{i_3t_6}^2]
    \E[Y_{j_1t_1}^2 Y_{j_1t_2}^2]
    =
    \E[Y_{i_1t_1}^2 Y_{i_1t_2}^2 Y_{i_1t_5}^2 Y_{i_1t_6}^2] \beta_{2,2}^2.
    \end{aligned}
\end{equation*}
The second equality is easiest to see by looking at the last two factors and consider the possible pairing of the $t$'s that yield nonzero results. Then we get the pairing in the first factor for free as a consequence. These are the same terms as in Table~\ref{table_6} so we get that this case goes to zero as well. Lastly, we look at $i_1 = j_2$ but $i_1 \neq i_2$. This case yields nonzero if $i_2 = i_3$ which then yields in \eqref{eq:E[(sum M_j,2^2)^2]} 
\begin{equation*}
    \E[Y_{i_1t_1} Y_{i_1t_2} Y_{i_1t_5} Y_{i_1t_6} Y_{i_1t_7} Y_{i_1t_8}]
    \E[Y_{i_2t_3} Y_{i_2t_4} Y_{i_2t_5} Y_{i_2t_6} Y_{i_2t_7} Y_{i_2t_8}]
    \E[Y_{j_1t_1} Y_{j_1t_2} Y_{j_1t_3} Y_{j_1t_4}].    
\end{equation*}
Note that we must have in the last factor above $t_1 = t_3$ and $t_2 = t_4$. For the first two factors, they can only yield the possible term $\beta_{2,2,2}$ to be nonzero each. But for this to happen we need to set $t_1$ and $t_2$ to be equal to some other $t$'s than only $t_3$ and $t_4$ respectively. This implies that an upper bound for the order would be $\beta_{2,2,2}^2 n^3 p^4 \sim n^{-5}p^4$ by Lemma~\ref{lem:allmoments}. Now Lemma~\ref{lem:asymp_variance} yields that $n^{-5}p^4 \sigma_n^{-4} \to 0$. By symmetry of indices we get the same results for $i_3 = j_1$ and we have thus shown \eqref{eq:M_j2^2_cip_1} for $\alpha \in (3, 4)$.

At this point only the case $\alpha = 3$ is left. A careful inspection of the above arguments in the case $\alpha\in (0,3)\cup(3,4)$ shows that \eqref{eq:M_j1^2_cip_1} respectively \eqref{eq:M_j2^2_cip_1} still hold for $\alpha = 3$ if $\sigma_n^2$ is replaced by $\E[T_1^2]$ and $\E[T_2^2]$, respectively.  That is, if $\alpha=3$ and $p = \omega(n^\delta)$ for some $\delta > \alpha/2 - 1$, it holds 
\begin{equation}\label{eq:trick}
    \begin{aligned}
    \frac{1}{\E[T_1^2]}\sum_{j=1}^p M_{j,1}^2 &\cip 1,\\
    \frac{1}{\E[T_2^2]}\sum_{j=1}^p M_{j,2}^2 &\cip 1.
    \end{aligned}
\end{equation}
Now we turn to the proof of \eqref{eq:M_j1M_j2^2_cip_1}. By virtue of \eqref{eq:trick}, we get
\begin{equation*}
    \begin{aligned}
    &\frac{1}{\E[T_1^2]+\E[T_2^2]}
    \sum_{j=1}^p \big( M_{j,1}^2+M_{j,2}^2 \big)\\
    &\quad =
    \frac{\E[T_1^2]}{\E[T_1^2]+\E[T_2^2]}
    \frac{1}{\E[T_1^2]} \sum_{j=1}^p  M_{j,1}^2
    +\frac{\E[T_2^2]}{\E[T_1^2]+\E[T_2^2]}
    \frac{1}{\E[T_2^2]} \sum_{j=1}^p  M_{j,2}^2\\
    &\quad =
    \frac{\E[T_1^2]}{\E[T_1^2]+\E[T_2^2]} (1+o_{\P}(1))
    +\frac{\E[T_2^2]}{\E[T_1^2]+\E[T_2^2]} (1+o_{\P}(1))\\
   &\quad = 1+o_{\P}(1)\,,
    \end{aligned}
\end{equation*}
where $o_{\P}(1)$ is a generic notation for a term that tends to zero in probability. In order to establish \eqref{eq:M_j1M_j2^2_cip_1}, it remains to show that
\begin{equation}\label{eq:ssdgg}  
    \frac{1}{\sigma_n^2}\sum_{j=1}^p M_{j,1}M_{j,2} \cip 0\,.
\end{equation}
To this end, an application of Markov's inequality yields for $\vep>0$,
\begin{equation*}
    \P\Big(\sigma_n^{-2} \sum_{j=1}^p M_{j,1} M_{j,2} > \varepsilon\Big)
    \lesssim
    \sigma_n^{-4} \sum_{j,k = 1}^p \E[M_{k,1}M_{k,2}M_{j,1}M_{j,2}].
\end{equation*}
By Lemma~\ref{lem:martingale_differences} and letting $k < j$ we have that $\E[M_{k,1}M_{k,2}M_{j,1}M_{j,2}]$ is equal to
\begin{equation*}
    \begin{aligned}
    &64 \sum_{i_1,i_2 = 1}^{k-1}\sum_{i_3,i_4 = 1}^{j-1}
    \sum_{t_1,t_4 = 1}^n \sum_{1\le t_2 < t_3\le n}
     \sum_{1\le t_5 < t_6\le n}\\
    &\quad \times
    \E[\bar Y_{i_1t_1}\bar Y_{kt_1} Y_{i_2t_2}Y_{i_2t_3}Y_{k t_2}Y_{k t_3}
    \bar Y_{i_3t_4}\bar Y_{jt_4} Y_{i_4t_5}Y_{i_4t_6}Y_{j t_5}Y_{j t_6}].
    \end{aligned}
\end{equation*}
Since $k < j$ the inner expectation can be broken into two factors whereof one is $\E[\bar Y_{jt_4}Y_{j t_5}Y_{j t_6}]$. The aforementioned factor is nonzero only if $t_5 = t_6$ which cannot happen. We conclude that
\begin{equation*}
    \begin{aligned}
    \sigma_n^{-4} \sum_{j,k = 1}^p \E[M_{k,1}M_{k,2}M_{j,1}M_{j,2}]
    &= \sigma_n^{-4} \sum_{j = 1}^p \E[M_{j,1}^2 M_{j,2}^2]\\
    &\leq \sigma_n^{-4} \sum_{j = 1}^p \bigg(\E[M_{j,1}^4] +  \E[M_{j,2}^4] \bigg)\\
    &\to 0,
    \end{aligned}
\end{equation*}
as $\nto$, by Lemma~\ref{lem:martingale_E[M_j,k^4]_to_0} if $p = \omega(n^\delta)$ for some $\delta > \delta^*(3)=1$. This establishes \eqref{eq:ssdgg}   and completes the proof of the lemma.
\end{proof}

\subsection{Proof of Theorem~\ref{thm:leading-piecewise}} 
To obtain all limiting moments of $T_1$ at the boundary \eqref{eq:pboundary}, we first express them through a graph representation.
Fix an integer $s\ge1$. Expanding $\E[T_1^s]$ yields
\begin{equation}\label{eq:T1s-expand}
\E[T_1^s] = \sum_{\substack{i_1\ne j_1,\ldots,i_s\ne j_s\\ i_\ell,j_\ell\in[p]}}\ \sum_{t_1,\ldots,t_s\in[n]} \E\!\Bigg[\prod_{\ell=1}^s \bar Y_{i_\ell t_\ell}\,\bar Y_{j_\ell t_\ell}\Bigg],
\end{equation}
where we used the notation $[p]=\{1,\ldots,p\}$ and $[n]=\{1,\ldots,n\}$. A \emph{configuration} $\omega$ is a triple sequence 
\begin{equation*}
\omega=\big((i_1,j_1,t_1),\ldots,(i_s,j_s,t_s)\big)
\end{equation*} 
with $i_\ell\ne j_\ell$. We associate each $\omega$ to a colored undirected multigraph $G(\omega)=(V(\omega),E(\omega),c_\omega)$ with vertices $V(\omega)=\{i_1,j_1,\ldots,i_s,j_s\}$, edges $E(\omega)=\{e_\ell:=\{i_\ell,j_\ell\}\}_{\ell=1}^s$ and colors $c_\omega(e_\ell)=t_\ell$. We write $a(\omega)=|V(\omega)|$ and $b(\omega)=|c_\omega(E(\omega))|$ for the number of vertices and used colors, respectively. Let $y(\omega)$ denote the number of edge-containing connected components. Thus, $\omega$ determines a multigraph with $a(\omega)$ vertices, $s$ edges, and $b(\omega)$ colors.

For each vertex $u\in V(\omega)$ and $t\in[n]$, define
\begin{itemize}
\item the degree $\deg_\omega(u):=\#\{\ell: i_\ell=u\}+\#\{\ell: j_\ell=u\}$, 
\item the multiplicity $k_{u,t}(\omega):=\#\{\ell:(i_\ell,t_\ell)=(u,t)\}+\#\{\ell:(j_\ell,t_\ell)=(u,t)\}$,
\item the set of incident colors $\mathcal C_\omega(u)=\{t:k_{u,t}(\omega)\ge1\}$ and $r_u(\omega):=|\mathcal C_\omega(u)|$.
\end{itemize}
For the remainder of this proof, we proceed as follows. Section~\ref{sec:lemmaaux} collects important auxiliary results to identify the leading order terms in \eqref{eq:T1s-expand}. Section~\ref{sec:provingtheorem} uses those results to compute the limit of \eqref{eq:T1s-expand}.

\subsubsection{Auxiliary lemmas for the proof of Theorem ~\ref{thm:leading-piecewise}}\label{sec:lemmaaux}
\begin{lemma}[Nonzero constraint]\label{lem:nonzero-constraint}
Let $\omega$ be a configuration in \eqref{eq:T1s-expand}. If $\deg_\omega(u)=1$ for some $u\in V(\omega)$, then
\begin{equation*}
    \E\!\Bigg[\prod_{\ell=1}^s \bar Y_{i_\ell t_\ell}\bar Y_{j_\ell t_\ell}\Bigg]=0.
\end{equation*}
As a consequence, only configurations with $\deg_\omega(u)\ge2$ for all $u\in V(\omega)$ can contribute to \eqref{eq:T1s-expand}.
\end{lemma}
\begin{proof}
If $\deg_\omega(u)=1$, then there exists a unique time index $t_0$ such that $k_{u,t_0}(\omega)=1$. Hence the contribution from row $u$ contains only the linear factor $\bar Y_{u,t_0}$. Since $\E[\bar Y_{u,t_0}]=0$, the desired result follows.
\end{proof}
\begin{lemma}[Row-wise factorization and beta-product reduction]
\label{lem:ru1-reduction}
Let $\omega$ be a configuration in \eqref{eq:T1s-expand}. Then
\begin{equation}\label{eq:row-factorization-compact}
    \E\!\Bigg[\prod_{\ell=1}^s \bar Y_{i_\ell t_\ell}\bar Y_{j_\ell t_\ell}\Bigg]
    =
    \prod_{u\in V(\omega)}
    \E\!\Bigg[\prod_{t\in\mathcal C_\omega(u)}\bar Y_{u,t}^{\,k_{u,t}(\omega)}\Bigg]
    =
    \prod_{u\in V(\omega)}
    \widetilde\beta_{2k_{u,t_1},\ldots,2k_{u,t_{r_u}}}.
\end{equation}
Now assume the regular variation condition with index $\alpha\in(0,4)$. If $r_u(\omega)\ge2$ for some vertex $u\in V(\omega)$, then
\begin{equation*}               \widetilde\beta_{2k_{u,t_1},\ldots,2k_{u,t_{r_u}}}=o\big(\widetilde\beta_{2\deg_\omega(u)}\big),
\end{equation*}
uniformly over all admissible multi-indices with total degree $\deg_\omega(u)\le2s$. Consequently, any configuration that is maximal at the beta-product level must satisfy
\begin{equation}\label{eq:ru1}
    r_u(\omega)=1\qquad\text{for all }u\in V(\omega).
\end{equation}
\end{lemma}
\begin{proof}
The factorization \eqref{eq:row-factorization-compact} is obtained by grouping the factors according to the row index and using the independence of the rows of $\Y$.

Now assume the regular variation condition with index $\alpha\in(0,4)$ and suppose that $r_u(\omega)\ge2$ for some vertex $u$. Applying Lemma~\ref{lem:tbeta} to the $u$-th factor in \eqref{eq:row-factorization-compact} yields
\begin{equation*}
    \E\Bigg[\prod_{t\in\mathcal C_\omega(u)}\bar Y_{u,t}^{\,k_{u,t}(\omega)}\Bigg]=\widetilde\beta_{2k_{u,t_1},\ldots,2k_{u,t_{r_u}}} = o\big(\widetilde\beta_{2\deg_\omega(u)}\big),
\end{equation*}
uniformly over all admissible multi-indices with total degree at most $2s$. Hence any configuration with $r_u(\omega)\ge2$ for some vertex $u\in V(\omega)$ is of strictly smaller order than one with the same total vertex degrees but satisfying $r_u(\omega)=1$ for all $u$. This proves \eqref{eq:ru1}.
\end{proof}
\begin{lemma}[Monochromatic components and maximal time-counting]
\label{lem:by-bound}
Assume that $r_u(\omega)=1$ for all $u\in V(\omega)$. Then every edge-containing connected component of the underlying uncolored multigraph $(V(\omega),E(\omega))$ is monochromatic. In particular, $b(\omega)\le y(\omega)$.

Moreover, the time-label summation contributes a factor $(n)_{b(\omega)}\asymp n^{b(\omega)}$, so the maximal time-counting order is attained when
\begin{equation*}
    b(\omega)=y(\omega).
\end{equation*}
\end{lemma}

\begin{proof}
If $r_u(\omega)=1$ for every vertex $u$, then each vertex is incident only to edges of a single color. Hence, along any path inside an edge-containing connected component, the color cannot change from one edge to the next. Therefore, each edge-containing connected component is monochromatic.

It follows immediately that the number of used colors cannot exceed the number of edge-containing connected components, that is,
\begin{equation*}
    b(\omega)\le y(\omega).
\end{equation*}
Since summing over the distinct time labels contributes the falling factorial
\begin{equation*}
    (n)_{b(\omega)}=n(n-1)\cdots(n-b(\omega)+1)\asymp n^{b(\omega)},
\end{equation*}
the maximal time-counting order is obtained by maximizing $b(\omega)$ subject to the constraint $b(\omega)\le y(\omega)$, namely by imposing $b(\omega)=y(\omega)$.
\end{proof}
\begin{lemma}[Vertex and component bounds for non-vanishing configurations]
\label{lem:vertex-component-bounds}
Assume that $\omega$ gives a non-vanishing contribution to \eqref{eq:T1s-expand}. Then $a(\omega)\le s$. Moreover, the number $y(\omega)$ of edge-containing connected components satisfies
\begin{equation*}
    y(\omega)\le \Big\lfloor\frac{a(\omega)}{2}\Big\rfloor,
\end{equation*}
and this bound is sharp. In particular, among configurations satisfying $r_u(\omega)=1$, the maximal time-exponent is attained when
\begin{equation}\label{eq:y-max-compact}
    y(\omega)=b(\omega)=\Big\lfloor\frac{a(\omega)}{2}\Big\rfloor.
\end{equation}
\end{lemma}
\begin{proof}
By Lemma~\ref{lem:nonzero-constraint}, any configuration with $\deg_\omega(u)=1$ for some $u$ gives zero contribution. Hence every non-vanishing configuration must satisfy
\begin{equation*}
    \deg_\omega(u) \ge 2 \qquad\text{for all }u\in V(\omega).
\end{equation*}
Using the degree-sum identity for multigraphs,
\begin{equation*}
    \sum_{u\in V(\omega)}\deg_\omega(u)=2|E(\omega)|=2s,
\end{equation*}
we obtain $2a(\omega)\le 2s$, or equivalently
\begin{equation*}
    a(\omega)\le s.
\end{equation*}

Now consider the number $y(\omega)$ of edge-containing connected components. Among loopless undirected multigraphs on $a(\omega)$ vertices with minimum degree at least $2$, the number of such components is maximized by splitting the graph into components of the smallest possible size. Hence
\begin{equation*}
    y(\omega)\le \Big\lfloor\frac{a(\omega)}{2}\Big\rfloor.
\end{equation*}
This bound is sharp: if $a(\omega)$ is even, it is attained by a disjoint union of $a(\omega)/2$ two-vertex components, each consisting of two parallel edges; if $a(\omega)$ is odd, it is attained by $(a(\omega)-1)/2$ such two-vertex components together with one three-vertex component of minimum degree at least $2$.

Finally, combining this with Lemma~\ref{lem:by-bound}, we see that the maximal time-exponent is attained precisely when
\begin{equation*}
    y(\omega)=b(\omega)=\Big\lfloor\frac{a(\omega)}{2}\Big\rfloor.
\end{equation*}
\end{proof}

\subsubsection{Proving Theorem~\ref{thm:leading-piecewise}}\label{sec:provingtheorem}
We now present the main structure of the proof of Theorem~\ref{thm:leading-piecewise}, which relies on the auxiliary lemmas from Section~\ref{sec:lemmaaux}.

\medskip
\textbf{Step 1: Finding dominant configurations.} Fix $s\ge2$. The first step is to identify the configurations contributing at the leading order in the expansion of $\E[T_1^s]$. By Lemma~\ref{lem:nonzero-constraint}, only configurations with $\deg_\omega(u)\ge2$ for all $u\in V(\omega)$ contribute to \eqref{eq:T1s-expand}.

Next, by Lemma~\ref{lem:ru1-reduction}, any configuration with $r_u(\omega)\ge2$ for some vertex $u$ is of strictly smaller order, so the dominant contribution must satisfy $r_u(\omega)=1$ for all $u\in V(\omega)$.

Under this condition, each edge-containing connected component is monochromatic, hence by Lemma~\ref{lem:by-bound} one has $b(\omega)\le y(\omega)$ and the time summation is maximized when $b(\omega)=y(\omega)$. Moreover, by Lemma~\ref{lem:vertex-component-bounds}, we know $y(\omega)\le \lfloor{a(\omega)}/{2}\rfloor$ and the bound is sharp. Therefore the dominant contribution is attained when $b(\omega)=y(\omega)=\lfloor{a(\omega)}/{2}\rfloor$.

\medskip
For each fixed $a\in\{2,\ldots,s\}$, let $\Omega_{a,s}$ denote the class of configurations $\omega$ in \eqref{eq:T1s-expand} satisfying
\begin{equation}\label{eq:Omega-constraints}
    \begin{aligned}
    |V(\omega)|=a,\quad |E(\omega)|=s, \quad b(\omega)=y(\omega)=\Big\lfloor\frac{a}{2}\Big\rfloor,\\
    \deg_\omega(u)\ge2\ \text{and}\ r_u(\omega)=1
    \quad \forall u\in V(\omega),
    \end{aligned}
\end{equation}
with distinct colors assigned to distinct edge-containing components. Let $\mathcal{H}(\omega)$ be the underlying uncolored multigraph obtained by forgetting the colors, and let $\mathcal H_{a,s}$ be the set of isomorphism types $H$ that arise from some $\omega_0\in\Omega_{a,s}$, that is, $H=\mathcal{H}(\omega_0)$. We then define $\deg_{H}(u):=\deg_{\omega_0}(u)$.

For each isomorphism type $H\in\mathcal H_{a,s}$, summing over row labels and time labels produces a factor of order $p^a n^{\lfloor a/2\rfloor}$, while the row-wise factorization contributes $\prod_{u\in V(H)}\widetilde\beta_{2\deg_H(u)}$. All remaining configurations (violating at least one constraint in \eqref{eq:Omega-constraints}) are of strictly smaller order and therefore are contained in the $o(\cdot)$ remainder in \eqref{eq:moment-dom-sum-clean}. Thus,
\begin{equation}\label{eq:moment-dom-sum-clean}
    \begin{aligned}
    \E[T_1^s]
    &= \sum_{a=2}^{s}\sum_{H\in\mathcal H_{a,s}}C_{H}\,
    \Bigg(\prod_{u\in V(H)} \widetilde\beta_{2\deg_H(u)}\Bigg)\,
    p^a\,n^{\lfloor a/2\rfloor} \\
    &+o\!\left(
    \max_{2\le a\le s}
    \sup_{H\in\mathcal H_{a,s}}
    \Bigg(\prod_{u\in V(H)} \widetilde\beta_{2\deg_H(u)}\Bigg)
    p^a\,n^{\lfloor a/2\rfloor}
    \right).
    \end{aligned}
\end{equation}
Here $C_{H}$ is a purely combinatorial constant depending only on the graph type $H$, counting the extremal configurations with this fixed uncolored structure. (Its explicit value will be determined in the second step of the proof.) By Lemmas~\ref{lem:tbeta} and~\ref{lem:allmoments}, we get $\widetilde\beta_{2k}\asymp n^{-c_\alpha}$, where $c_\alpha=\max\{1,\alpha/2\}$, and consequently the contribution of the $a$-th layer is of order
\begin{equation*}
    p^a\,n^{-ac_\alpha+\lfloor a/2\rfloor}.
\end{equation*}
On the boundary scale, using $\delta^*(\alpha)=c_\alpha-\tfrac12$, we obtain
\begin{equation*}
    p^a\,n^{-ac_\alpha+\lfloor a/2\rfloor}\asymp n^{-a/2+\lfloor a/2\rfloor}.
\end{equation*}
Hence only the even layers are of leading order, whereas the odd layers are negligible. 

It follows that the dominant configurations must satisfy
\begin{equation*}
    a(\omega)=2v,\qquad \text{ for some } 1\le v\le \lfloor s/2\rfloor.
\end{equation*}
The extremal condition $y(\omega)=a(\omega)/2$ then forces each edge-containing connected component to have the minimal possible number of row-vertices, namely two. Hence the dominant configurations are precisely those consisting of $v$ disjoint row-pairs, each assigned a distinct time label.

\medskip
\textbf{Step 2: Counting dominant configurations.}
The second step is to count such configurations. Fix a $1\le v\le \lfloor s/2\rfloor$. Let $[s]:=\{1,\ldots,s\}$ index the $s$ factors in $T_1^s$. We group indices according to the (triple) type $(\{i,j\},t)$ they realize: equivalently, we choose a set partition $\pi=\{B_1,\ldots,B_v\}$ of $[s]$ into $v$ blocks with sizes
\begin{equation*}
m_w := |B_w|\ge2, \qquad w=1,\ldots,v, \qquad m_1+\cdots+m_v = s.
\end{equation*}

For a fixed ordered composition $(m_1,\ldots,m_v)$, the number of (unordered) set partitions of $[s]$ with these block sizes is
\begin{equation*}
    \frac{s!}{m_1!\cdots m_v!\,v!}.
\end{equation*}
Next, for each block $B_w$ we select one unordered row pair $\{i_w,j_w\}$ and one time index $t_w$, with the restriction that pairs and times are distinct across different blocks. This gives a multiplicity factor
\begin{equation*}
    (p)_{2v}\,(n)_v \;\sim\; p^{2v}n^v .
\end{equation*}
Finally, inside each block $B_w$ there are $2^{m_w}$ possible orientation assignments $(i_\ell,j_\ell)=(i_w,j_w)$ or $(j_w,i_w)$ for $\ell\in B_w$, but the global flip on the entire block does not change the underlying unordered pair. Hence the number of distinct orientation patterns in block $B_w$ equals $2^{m_w-1}$, and multiplying over blocks yields
\begin{equation*}
    \prod_{w=1}^v 2^{m_w-1} \;=\; 2^{\,s-v}.
\end{equation*}

For a fixed block $B_w$ of size $m_w$, row-independence and the one-color-per-vertex structure imply that the corresponding expectation contributes the one-index moment product 
\begin{equation*}
    \E\big[\bar Y_{i_w,t_w}^{\,m_w}\bar Y_{j_w,t_w}^{\,m_w}\big]=\widetilde\beta_{2m_w}^{\,2}.
\end{equation*}
Since different blocks use disjoint row pairs and distinct time indices, the block contributions factorize, giving $\prod_{w=1}^v \widetilde\beta_{2m_w}^{\,2}$.

Putting the counting and moment factors together, the total contribution associated with a fixed $v$ and a fixed composition $(m_1,\ldots,m_v)$ is asymptotically
\begin{equation*}
    \frac{s!}{m_1!\cdots m_v!\,v!}\; 2^{\,s-v}\; (p)_{2v}(n)_v\;
    \prod_{w=1}^v \widetilde\beta_{2m_w}^{\,2}
    \;\sim\;
    \frac{2^{\,s-v}}{v!}\,\frac{s!}{m_1!\cdots m_v!}\,
    \widetilde\beta_{2m_1}^{\,2}\cdots \widetilde\beta_{2m_v}^{\,2}\, n^{v}\, p^{2v}.
\end{equation*}
Summing over all $v=1,\ldots,\lfloor s/2\rfloor$ and all compositions $m_1+\cdots+m_v=s$ with $m_w\ge2$, we have
\begin{equation}
    \E[T_1^{s}] \sim
    \sum_{v=1}^{\lfloor s/2 \rfloor} 
    \frac{2^{\,s-v}}{v!}
    \sum_{\substack{m_1,\ldots,m_v\ge2\\ m_1+\cdots+m_v=s}}
    \frac{s!}{m_1!\cdots m_v!}\,
    \widetilde\beta_{2m_1}^{\,2}\cdots 
    \widetilde\beta_{2m_v}^{\,2}\,
    n^{v}\, p^{2v}\,,\qquad \nto\,.
\end{equation}
Finally, Lemmas~\ref{lem:allmoments} and~\ref{lem:tbeta} allow us to replace $\widetilde{\beta}$ by its asymptotic equivalent, which yields the conclusion of the theorem
\begin{equation}
    \E[T_1^s] \to
    \sum_{v=1}^{\lfloor s/2 \rfloor}
    c_{\delta}^{2v} \, \frac{2^{\,s-v}}{v!}\,
    \sum_{\substack{m_1,\ldots,m_v\ge2\\ m_1+\cdots +m_v=s}}
    \frac{s!}{m_1!\cdots m_v!}\,
    \prod_{\ell=1}^v c_{m_\ell}(\alpha)^2.
    \qquad \nto\,, s\ge 2. 
\end{equation}

\subsection{Proof of Lemma~\ref{lem:carleman-nongauss}}
\begin{proof}
It is sufficient to verify the Carleman condition 
\begin{equation}\label{eq:carleman}
    \sum_{k=1}^{\infty} M_{2k}^{-\frac{1}{2k}}=\infty \,,
\end{equation}
since this condition implies that a distribution with moment sequence $(M_s)_{s\ge 1}$ is moment-determinate. Let  $k\ge 2$. From \eqref{eq:ms-def} with $s=2k$, we get
\begin{equation}\label{eq:m2k_exp_clean}
    M_{2k}=
    \sum_{v=1}^{k}
    c_\delta^{2v}\,\frac{2^{\,2k-v}}{v!}
    \sum_{\substack{m_1,\ldots,m_v\ge2\\ m_1+\cdots+m_v=2k}}
    \frac{(2k)!}{m_1!\cdots m_v!}\,
    \prod_{\ell=1}^v c_{m_\ell}(\alpha)^2.
\end{equation}
By the asymptotic expansion of the Gamma function, for fixed $\alpha$ it holds
\begin{equation*}
    \frac{\Gamma(m-\alpha/2)}{\Gamma(m)} = m^{-\alpha/2}\bigl(1+O(m^{-1})\bigr),\qquad m\to\infty.
\end{equation*}
In particular, since $\alpha>0$, the ratio $\Gamma(m-\alpha/2)/\Gamma(m)$  tends to $0$ as $m\to\infty$, and therefore
\begin{equation*}
    \sup_{m\ge2} \frac{\Gamma(m-\alpha/2)}{\Gamma(m)} <\infty.
\end{equation*}
Hence, there exists a constant $C<\infty$ such that $0\le c_m(\alpha)\le C$ for all $  m\ge2$.
In conjunction with \eqref{eq:m2k_exp_clean}, this  yields
\begin{equation}\label{eq:m2k_bd1_clean}
    M_{2k} \le
    \sum_{v=1}^{k}
    \frac{2^{\,2k}}{v!}\,
    \big(c_\delta^2 C^2\big)^v
    \sum_{\substack{m_1,\ldots,m_v\ge2\\ m_1+\cdots+m_v=2k}}
    \frac{(2k)!}{m_1!\cdots m_v!}.
\end{equation}

Dropping the constraint $m_\ell\ge2$ and summing over all $m_\ell\ge0$ gives
\begin{equation*}
    \sum_{\substack{m_1,\ldots,m_v\ge2\\ m_1+\cdots+m_v=2k}} \frac{(2k)!}{m_1!\cdots m_v!}
    \le \sum_{\substack{m_1,\ldots,m_v\ge0\\ m_1+\cdots+m_v=2k}}\frac{(2k)!}{m_1!\cdots m_v!}
    = v^{2k},
\end{equation*}
by the multinomial theorem. Substituting into \eqref{eq:m2k_bd1_clean} we obtain
\begin{equation}\label{eq:m2k_bd2_clean}
    M_{2k} \le 2^{2k} \sum_{v=1}^{k} \frac{\big(c_\delta^2 C^2\big)^v}{v!}\,v^{2k}
    \le 2^{2k}k^{2k} \sum_{v=1}^{\infty}\frac{\big(c_\delta^2 C^2\big)^v}{v!}
    = 2^{2k}k^{2k}\,\exp\!\big(c_\delta^2 C^2\big).
\end{equation}

Taking $(2k)$-th roots in \eqref{eq:m2k_bd2_clean} gives
\begin{equation*}
    M_{2k}^{1/(2k)} \le 2k\,\exp\!\Big(\frac{c_\delta^2 C^2}{2k}\Big) \le C^*\,k \qquad(k\ge1)
\end{equation*}
for some constant $C^*<\infty$. Hence $M_{2k}^{-1/(2k)}\ge 1/(C^*k)$ and therefore
\begin{equation*}
    \sum_{k=2}^{\infty} M_{2k}^{-1/(2k)} \;\ge\; \frac{1}{C}\sum_{k=2}^{\infty}\frac{1}{k} =\infty,
\end{equation*}
which establishes the Carleman condition \eqref{eq:carleman}.
\end{proof}

\subsection{Proof of Theorem~\ref{thm:boundary-cumulant}}
\begin{proof}
Starting from \eqref{eq:ms-def}, use the identity $2^{s-v}=\prod_{\ell=1}^v 2^{m_\ell-1}$ whenever $\sum_{\ell=1}^v m_\ell=s$ and absorb the factors $2^{m_\ell-1}c_\delta^2c_{m_\ell}(\alpha)^2$ into $\kappa_{m_\ell}^{(\alpha,c_\delta)}$. Since $\prod_{\ell=1}^v \kappa_{m_\ell}^{(\alpha,c_\delta)}=0$ whenever  $\min _\ell m_{\ell}=1$, this yields 
\begin{equation*}
    \begin{aligned}
    M_s
    &=
    \sum_{v=1}^{s} \frac{1}{v!}\, \sum_{\substack{m_1,\ldots,m_v\ge1\\ m_1+\cdots +m_v=s}} \frac{s!}{m_1!\cdots m_v!}\, \prod_{\ell=1}^v \kappa_{m_\ell}^{(\alpha,c_\delta)}\\
    &= B_s\big(\kappa_{1}^{(\alpha,c_\delta)},\kappa_{2}^{(\alpha,c_\delta)},\ldots,\kappa_{s}^{(\alpha,c_\delta)}\big), \qquad s\ge 1,
    \end{aligned}
\end{equation*}
by definition of the $s$-th complete exponential Bell polynomial in \eqref{eq:bellpolynomial}. Moreover, by \cite[Proposition~3.3.1]{peccati:taqqu:2011}, the cumulants of the corresponding distribution are given by $\big(\kappa_{s}^{(\alpha,c_\delta)}\big)_{s\ge 1}$.

It remains to prove \eqref{eq:Kalpha-closed-thm}.  For $\alpha\in (0,2)$ let $U_1,U_2$ be two independent $\mathrm{Beta}(1-\alpha/2,\alpha/2)$ variables.  For $\alpha\in [2,3)$ let $U_1',U_2'$ be two independent $\mathrm{Beta}(2-\alpha/2,\alpha/2)$ variables.  We know
\begin{equation*}
    \E[U_1^{m-1}]=
    \frac{\Gamma(m-\alpha/2)}{\Gamma(1-\alpha/2)\Gamma(m)},\quad
    \E[(U'_1)^{m-2}]=
    \frac{\Gamma(m-\alpha/2)}{\Gamma(2-\alpha/2)\Gamma(m)}.
\end{equation*}
Then for $m\ge2$,
\begin{equation*}
    c_m^2(\alpha)=
    \begin{cases}
        \E[U_1^{m-1}]\,\E[U_2^{m-1}] =
        \E[(U_1U_2)^{m-1}] =
        \E[W^{m-1}] &\text{if}\ \alpha\in (0,2)\,, \\
        c_2(\alpha)^2 \,\E[(U'_1 U'_2)^{m-2}] = 
        c_2(\alpha)^2 \,\E[(W')^{m-2}]    &\text{if}\ \alpha\in [2,3)\,,
    \end{cases}
\end{equation*}
and thus
\begin{equation}\label{eq:kappa-as-W}
    \kappa_m^{(\alpha,c_\delta)}=
    \begin{cases}
        2^{m-1}\,c_{\delta}^{2}\,\E[W^{m-1}] &\text{if}\ \alpha\in (0,2)\,, \\
        2^{m-1}\,c_{\delta}^{2}c_2(\alpha)^2\,\E[(W')^{m-2}] &\text{if}\ \alpha\in [2,3)\,,
    \end{cases}
    \qquad m\ge 2\,.
\end{equation}
Substituting \eqref{eq:kappa-as-W} into the definition of the cumulant generating function gives
\begin{equation}\label{eq:K-expectation-sum}
    K_{\alpha,c_\delta}(t)=
    \begin{cases}
        c_{\delta}^{2}\,\E\Bigg[\sum_{m\ge 2}2^{m-1}W^{m-1}\frac{t^m}{m!}\Bigg],\\
        c_{\delta}^{2}c_2(\alpha)^2 \,\E\Bigg[\sum_{m\ge 2}2^{m-1}(W')^{m-2}\frac{t^m}{m!}\Bigg].
    \end{cases}
\end{equation}
For fixed $w\in(0,1)$, compute the inner series
\begin{equation}\label{eq:inner-series}
    \begin{cases}
    \sum_{m\ge 2}2^{m-1}w^{m-1}\frac{t^m}{m!}
    &= \frac{1}{2w}\sum_{m\ge2}\frac{(2tw)^m}{m!}
    = \frac{e^{2tw}-1-2tw}{2w},\\
    \sum_{m\ge 2}2^{m-1}w^{m-2}\frac{t^m}{m!}
    &= \frac{1}{2w^2}\sum_{m\ge2}\frac{(2tw)^m}{m!}
    = \frac{e^{2tw}-1-2tw}{2w^2}.
    \end{cases}
\end{equation}
Combining \eqref{eq:K-expectation-sum} and \eqref{eq:inner-series} yields the closed form \eqref{eq:Kalpha-closed-thm}.
\end{proof}

\subsection{Proof of Corollary~\ref{cor:poisson}}
\begin{proof}
Fix $c_\delta>0$. As $\alpha\downarrow0$, one has $U_1, U_2\Rightarrow 1$ and hence $W\Rightarrow 1$, so
\begin{equation*}
    K_{\alpha,c_\delta}(t)
    \longrightarrow 
    K_{0}(t):= c_{\delta}^{2}\,\frac{e^{2t}-1-2t}{2}\,, \qquad t\in\R.
\end{equation*}
Therefore, the corresponding moment generating function is
\begin{equation*}
    \begin{aligned}
    M_0(t)&:=
    \exp\!\left(
    c_\delta^{2}\frac{e^{2t}-1-2t}{2}
    \right)
    =
    \exp\!\left({-c_\delta^{2} t}\right)\cdot
    \exp\!\left(\frac{c_\delta^{2}}{2}(\e^{2t}-1)\right)\\
    &= \exp\!\left({-c_\delta^{2} t}\right)\cdot \E\big[\e^{2tN}\big]=\E[\e^{t(2N-c_\delta^{2})}]
    \end{aligned}
\end{equation*}
where $N\sim\mathrm{Poisson}(c_\delta^{2}/2)$. This establishes the desired convergence to the centered Poisson distribution.
\end{proof}

\subsection{Proof of Theorem~\ref{thm:Poisson point process representation}}
\begin{proof}
For brevity, we write $\nu$ instead of $\nulim$.
For a Poisson point process 
\begin{equation*}
    N=\sum_{i=1}^\infty \delta_{\xi_i}
\end{equation*}
with intensity $\nu$ and atoms $(\xi_i)_{i\ge 1}$, 
we define the compensated Poisson integral
\begin{equation*}
    Z:=\int_{(0,2]}x\,\widehat N(dx)=\sum_{i=1}^\infty \xi_i-\int_{(0,2]}x\,\nu(dx),
\end{equation*}
where $\widehat{N}:=N-\nu$. By construction, we have $\E[Z]=0$.

Let $\kappa_m(Z)$ denote the $m$-th cumulant of $Z$. By the Laplace functional of a Poisson process ~\cite[Theorem~3.9, p.~23]{LastPenrose2017}, together with the definition of compensated Poisson integrals \cite[Eq.~(12.4), p.~112]{LastPenrose2017}, we get
\begin{equation*}
    \kappa_1(Z)=0, \qquad \kappa_m(Z)=\int_{(0,2]}x^m\,\nu(dx), \qquad m\ge2.
\end{equation*}
For $\alpha\in(0,2)$, applying \eqref{eq:nu-unified} gives
\begin{equation*}
    \int_{(0,2]}x^m\nu(dx) = c_\delta^2\int_{(0,2]}x^{m-1}\mu(dx) = c_\delta^2\E[(2W)^{m-1}] = 2^{m-1}c_\delta^2\E[W^{m-1}].
\end{equation*}
Since $W$ is the product of two independent $\mathrm{Beta}(1-\alpha/2,\alpha/2)$ random variables, we have
\begin{equation*}
    \E[W^{m-1}] = c_m(\alpha)^2.
\end{equation*}
Thus
\begin{equation*}
    \kappa_m(Z) = 2^{m-1}c_\delta^2c_m(\alpha)^2, \qquad m\ge2.
\end{equation*}

For $\alpha\in[2,3)$, again by \eqref{eq:nu-unified},
\begin{equation*}
    \int_{(0,2]}x^m\nu(dx) = 2c_\delta^2c_2(\alpha)^2 \int_{(0,2]}x^{m-2}\mu'(dx)
    = 2^{m-1}c_\delta^2c_2(\alpha)^2\E[(W')^{m-2}].
\end{equation*}
Since $W'$ is the product of two independent $\mathrm{Beta}(2-\alpha/2,\alpha/2)$ random variables,
\begin{equation*}
    c_m(\alpha)^2 = c_2(\alpha)^2\E[(W')^{m-2}],
\end{equation*}
and therefore
\begin{equation*}
    \kappa_m(Z) = 2^{m-1}c_\delta^2c_m(\alpha)^2, \qquad m\ge2.
\end{equation*}

Consequently, we have shown that, for all $\alpha\in (0,3)$,
\begin{equation*}
    \kappa_1(Z)=0, \qquad
    \kappa_m(Z) = 2^{m-1}c_\delta^2c_m(\alpha)^2 = \kappa_m^{(\alpha,c_\delta)}, \qquad m\ge2.
\end{equation*}
By Theorem~\ref{thm:boundary-cumulant}, these are precisely the cumulants of the law $\etalim$. Hence
\begin{equation*}
    \int_{(0,2]}x\,\widehat N(dx) = \sum_{i=1}^\infty \xi_i-\int_{(0,2]}x\,\nu(dx) = \sum_{i=1}^\infty \xi_i-c_{\delta}^2 \sim \etalim,
\end{equation*}
establishing the desired result.
\end{proof}

\section{Proofs for Section~\ref{sec:preliminaries}} \label{sec:proofsmoment} \setcounter{equation}{0}
\subsection{Proof of Lemma~\ref{lem:allmoments}}
\begin{proof}
For a proof of part (a), see \cite[p.~4]{albrecher:teugels:2007}. 
Regarding part (b), we remark that  \eqref{moment24} was proved in \cite[Lemma~4.1]{heiny:parolya:2024} for $\alpha\in(2,4)$. For our case let $\beta=\alpha/2$ and $X\eid X_{11}$. From \cite[p.~7]{albrecher:teugels:2007}, we have
\begin{equation}\label{eq:formulagine}
    \E[ Y_{11}^{2k_1}\cdots Y_{1r}^{2k_r}]= \frac{(-1)^k}{n \Gamma(k)} \int_0^{\infty} \Big( \tfrac{t}{n}\Big)^{k-1} \varphi^{n-r}\Big( \tfrac{t}{n}\Big) \prod_{i=1}^r \varphi^{(k_i)}\Big( \tfrac{t}{n}\Big) \dint t\,,
\end{equation}
where $\varphi(s)=\E[\e^{-sX^2}]$, $s>0$, and $\varphi^{(m)}(s)=\frac{\dint^m}{\dint s^m} \varphi(s)$. By \cite{albrecher:teugels:2007}, we have 
\begin{equation}\label{lim1}
    \lim_{\nto} \varphi^{n-r}\Big( \tfrac{t}{n}\Big)=\e^{-t}\,, \qquad t>0\,,
\end{equation}
provided that $\E[X^2]=1$. For regularly varying $|X|$ with index $\beta$, \cite[Lemma~2]{ladoucette:2007} asserts that the asymptotic behavior of $\varphi^{(m)}(s)$, $m\in \N$, at the origin is given by
\begin{equation}\label{eq:asyphi}
    (-1)^m \varphi^{(m)}(s) \sim
    \left\{\begin{array}{ll}
    \beta \Gamma(m-\beta) s^{\beta -m} L(s^{-\frac12}) \,, & \mbox{if } m>\beta, \\
    \beta \ell(s^{-1}) \,, & \mbox{if }  m= \beta \text{ and } \E[X^{2m}]=\infty,\\
    \E[X^{2m}] \,, & \mbox{if }  m\le \beta \text{ and } \E[X^{2m}]<\infty,
    \end{array}\right. \quad s \downarrow 0\,,
\end{equation}
where $\ell(x)=\int_0^x L(u^{1/2})/u \dint u$ is a slowly varying function (at infinity).

By \eqref{eq:formulagine}, Potter's theorem and the dominated convergence theorem (for more details see \cite{albrecher:teugels:2007} or \cite{fuchs:joffe:teugels:2001}), we obtain in view of \eqref{lim1} and \eqref{eq:asyphi} that, as $\nto$,
\begin{equation*}
    \begin{split}
    \E[ Y_{11}^{2k_1}\cdots Y_{1r}^{2k_r}]&= \frac{(-1)^k}{n \Gamma(k)} \int_0^{\infty} \Big( \tfrac{t}{n}\Big)^{k-1} \varphi^{n-r}\Big( \tfrac{t}{n}\Big) \Big(\varphi^{(1)}\big( \tfrac{t}{n}\big)\Big)^{N_1} \prod_{i:k_i\ge 2} \varphi^{(k_i)}\Big( \tfrac{t}{n}\Big) \dint t\\
    &\sim \frac{1}{n \Gamma(k)} \int_0^{\infty} \Big( \tfrac{t}{n}\Big)^{k-1} \e^{-t} \Big(\E[X^{2}]\Big)^{N_1} \prod_{i:k_i\ge 2} \beta \Gamma(k_i-\beta) \big( \tfrac{t}{n}\big)^{\beta -k_i} \underbrace{L\Big(\big( \tfrac{t}{n}\big)^{-1/2}\Big)}_{\sim L(n^{1/2})} \dint t\\
    &\sim \Big(\prod_{i:k_i\ge 2} \Gamma(k_i-\beta) \Big) \frac{\beta^{r-N_1}L^{r-N_1}(n^{1/2})}{n^{N_1(1-\beta)+\beta r} \Gamma(k)} \int_0^{\infty} \e^{-t} t^{N_1(1-\beta)+\beta r-1} \dint t\\
    &= \frac{L^{r-N_1}(n^{1/2})}{n^{N_1(1-\beta)+\beta r}} \frac{\beta^{r-N_1}\Gamma(N_1(1-\beta)+\beta r) \, \prod_{i:k_i\ge 2} \Gamma(k_i-\beta)}{\Gamma(k)}\,.
    \end{split}
\end{equation*}
Rearranging yields \eqref{moment24} and completes the proof of part (b). 

The proof of part (c) is very similar. To this end, note that the three lines in \eqref{eq:asyphi} only differ by a slowly varying function. In the case $\{\alpha=2$ and $\E[X_{11}^2]=\infty\}$, one has to use the middle line in \eqref{eq:asyphi} for $\varphi^{(1)}$ instead of the first one, combined with equation (14) in \cite{albrecher:teugels:2007} instead of \eqref{lim1}. In the case $\{\alpha=4, \E[X_{11}^2]=1$ and $\E[X_{11}^4]=\infty\}$, one needs to use the middle line in \eqref{eq:asyphi} for $\varphi^{(2)}$ instead of the last one. For brevity we omit details.

Regarding part (d), we analogously get, as $\nto$,
\begin{equation*}
    \begin{split}
    \E[ Y_{11}^{2k_1}\cdots Y_{1r}^{2k_r}]&= \frac{(-1)^k}{n \Gamma(k)} \int_0^{\infty} \Big( \tfrac{t}{n}\Big)^{k-1} \varphi^{n-r}\Big( \tfrac{t}{n}\Big)  \prod_{i=1}^r \varphi^{(k_i)}\Big( \tfrac{t}{n}\Big) \dint t\\
    &\sim \frac{1}{n \Gamma(k)} \int_0^{\infty} \Big( \tfrac{t}{n}\Big)^{k-1} \e^{-t} \prod_{i=1}^r \E[X^{2 k_i}] \dint t = \frac{1}{n^k} \prod_{i=1}^r \E[X^{2 k_i}]\,.
    \end{split}
\end{equation*}
\end{proof}

\subsection{Proof of Lemma~\ref{lem:tbeta}}
\begin{proof}
For any real numbers $b_1, \ldots, b_k$ we have the identity
\begin{equation*}
    \prod_{i=1}^k (b_i-\tfrac{1}{n}) = \sum_{m=0}^k \sum_{\substack{S\subseteq \{1,\ldots, k\} \\ |S|=m}} (-1)^m n^{-m} \prod_{i\in S^c} b_i\,,
\end{equation*}
where $S^c=\{1,\ldots,k\} \backslash S$ denotes the complement of $S$. Applying this identity for $k=k_1+\cdots +k_r$ with positive integers $k_1,\ldots,k_r$ and 
\begin{equation*}
    b_i =
    \begin{cases}
        Y_{11}^2 \,, & \mbox{if } i=1,\ldots, k_1 \\
        Y_{12}^2 \,, & \mbox{if } i=k_1+1,\ldots, k_1+k_2 \\
        \phantom{Y} \vdots  &  \\
        Y_{1r}^2 \,, & \mbox{if } i=k_1+\cdots+ k_{r-1}+1,\ldots, k_1+\cdots +k_r\,,
    \end{cases}
\end{equation*}
we obtain
\begin{equation*}
    \begin{aligned}
    \tbeta_{2k_1,\ldots, {2k_r}}&= \E \prod_{j=1}^r \prod_{\ell=1}^{k_j} (Y_{1j}^2-\tfrac{1}{n}) = \E \prod_{i=1}^k (b_i-\tfrac{1}{n})\\
    &= \sum_{m=0}^k \sum_{\substack{S\subseteq \{1,\ldots, k\} \\ |S|=m}} (-1)^m n^{-m}\,  \E \prod_{i\in S^c} b_i\,.
    \end{aligned}
\end{equation*}
Using the shorthand notation $b(S^c):= \E \prod_{i\in S^c} b_i$ and observing that $b(S)=\beta_{2k_1,\ldots, {2k_r}}$, we deduce that 
\begin{equation*}
    \tbeta_{2k_1,\ldots, {2k_r}} =\beta_{2k_1,\ldots, {2k_r}} +\sum_{m=1}^k (-1)^m \sum_{\substack{S\subseteq \{1,\ldots, k\} \\ |S|=m}}  n^{-m}\,b(S^c)\,.
\end{equation*}
Without loss of generality we assume that the $k_i$'s are ordered, that is, $k_1\ge \cdots \ge k_r$. If $k_r\ge 2$, one can see from Lemma~\ref{lem:allmoments} that $n^{-m}\,b(S^c) =o(\beta_{2k_1,\ldots, {2k_r}})$ for any $S\subseteq \{1,\ldots, k\}$ with cardinality $1\le m \le k$. Therefore, we conclude that 
\begin{equation*}
    \tbeta_{2k_1,\ldots, {2k_r}} \sim \beta_{2k_1,\ldots, {2k_r}}\,, \qquad \nto\,, k_r \ge 2.
\end{equation*}
If $k_r= 1$, it one can analogously get from Lemma~\ref{lem:allmoments}  that $n^{-m}\,b(S^c) =O(\beta_{2k_1,\ldots, {2k_r}})$ for any $S\subseteq \{1,\ldots, k\}$ with cardinality $1\le m \le k$. Note that, for example, for $S=\{k\}$ we have $n^{-1}\,b(S^c)= n^{-1}\,b(\{1,\ldots, k-1\})=n^{-1} \beta_{2k_1,\ldots, {2k_{r-1}}}$, which is of the same order as $\beta_{2k_1,\ldots, {2k_{r-1}},2}$. We conclude that 
\begin{equation*}
    \tbeta_{2k_1,\ldots, {2k_r}} =O(\beta_{2k_1,\ldots, {2k_r}})\,, \qquad \nto\,, k_r =1,
\end{equation*}
which completes the proof of the first assertion of Lemma~\ref{lem:tbeta}.

To prove the second assertion, let $k=k_1+\cdots+k_r$ and note that by Lemma~\ref{lem:allmoments},
\begin{equation}
    \begin{aligned}
    \beta_{2k_1,\ldots,2k_r}
    &\slv  
    \begin{cases}
    n^{-r}, & \alpha\in(0,2),\\
    n^{-N_1-(\alpha/2)(r-N_1)}, & \alpha\in[2,4),
    \end{cases}
    \\
    \beta_{2k}
    &\slv
    \begin{cases}
    n^{-1}, & \alpha\in(0,2),\\[2pt]
    n^{-\alpha/2}, & \alpha\in[2,4).
    \end{cases}
    \end{aligned}
\end{equation}

In conjunction with the first assertion of Lemma~\ref{lem:tbeta}, this yields
\begin{equation*}
    \left|\frac{\widetilde\beta_{2k_1,\ldots,2k_r}}{\widetilde\beta_{2k}}\right|
    \lesssim \frac{\beta_{2k_1,\ldots,2k_r}}{\beta_{2k}}
    \slv n^{\Delta_\alpha(r,N_1)},
\end{equation*}
where
\begin{equation*}
    \Delta_\alpha(r,N_1)=
    \begin{cases}
        -(r-1), & \alpha\in(0,2),\\[4pt]
        -\dfrac{\alpha}{2}(r-N_1-1)-N_1, & \alpha\in[2,4).
    \end{cases}
\end{equation*}
Since $N_1\le r$, $r\ge2$ and $\alpha<4$, we see that in all cases $\Delta_\alpha(r,N_1)<0$ which finishes the proof of the lemma.
\end{proof}

\subsection{Proof of Lemma~\ref{lem:variance}}
\begin{proof}
Let $T_1$ and $T_2$ be defined by \eqref{eq:def_T1_T2}. We get 
\begin{equation}\label{eq_E(T1_squared))}
    \begin{aligned}
    \E{[T_1^2]} 
    &= \E\left[\left(2\sum_{i_1 < i_2}^p \sum_{t = 1}^n \bar{Y}_{i_1, t} \bar{Y}_{i_2, t}\right)^2\right]
    \\ &= 4 \sum_{i_1 < i_2}^p \sum_{t = 1}^n \E[\bar{Y}_{i_1, t}^2]\E[\bar{Y}_{i_2, t}^2] + 8 \sum_{i_1 < i_2}^p \sum_{t_1 < t_2}^n \E[\bar{Y}_{i_1, t_1} \bar{Y}_{i_1, t_2}] \E[ \bar{Y}_{i_2, t_1} \bar{Y}_{i_2, t_2}]
    \\ &= 2p(p-1)n \left(\beta_4 - \frac{1}{n^2}\right)^2 +2p(p-1)n(n-1)\left(\beta_{2,2} - \frac{1}{n^2}\right)^2
    \\ &= 2p(p-1)\frac{n^2}{n - 1} \left(\beta_4 - \frac{1}{n^2}\right)^2.
    \end{aligned}
\end{equation}

For the last equality we used Lemma~\ref{lem:betas} to express $\beta_{2,2}$ in terms of $\beta_4$. Similarly we get for $T_2$,
\begin{equation}\label{eq_E(T2_squared))}
    \begin{aligned} 
    \E[T_2^2]
    &= \E\left[\left(\sum_{\substack{i_1,i_2 = 1 \\ i_1 \neq i_2}}^p\ \sum_{\substack{t_1,t_2 = 1 \\ t_1 \neq t_2}}^n Y_{i_1 t_1} Y_{i_1 t_2} Y_{i_2 t_1} Y_{i_2 t_2}\right)^2\right]\\
    &= 16\, \E\left[\sum_{i_1 < i_2}^p \sum_{t_1 < t_2}^n Y_{i_1 t_1}^2 Y_{i_1 t_2}^2 Y_{i_2 t_1}^2 Y_{i_2 t_2}^2 \right] 
    \\ &= 4p(p-1)n(n-1)\beta_{2,2}^2
    = 4p(p-1)\frac{n}{n-1}\left(\frac{1}{n}-\beta_4\right)^2.
    \end{aligned}
\end{equation}
\end{proof}

\subsection{Proof of Lemma~\ref{lem:asymp_variance}}
\begin{proof}
From \eqref{eq:var} we obtain, as $\nto$,  
\begin{equation*}
   \Var(\tr(\bfR^2)) \sim 2np^2\bigg(\bigg(\beta_4 - \frac{1}{n^2}\bigg)^2 + \frac{2}{n}\bigg(\beta_4 - \frac{1}{n}\bigg)^2\bigg).
\end{equation*}
We start with the case $\E[X^4] < \infty$, where Lemma~\ref{lem:allmoments} yields $\beta_4 \sim n^{-2} \E[X^4]$. It is easy to see that 
\begin{equation*}
    \bigg(\beta_4 - \frac{1}{n^2}\bigg)^2 = O(n^{-4})  ,\quad \frac{2}{n}\bigg(\beta_4 - \frac{1}{n}\bigg)^2 \sim 2 n^{-3},    
\end{equation*}
so the asymptotic behavior of $\Var(\tr(\bfR^2))$ is the same as of $\sigma_n^2$ and  
\begin{equation*}
    \sigma_n^2 \sim 2np^2\cdot 2n^{-3} \sim 4\frac{p^2}{n^2}.
\end{equation*}
For the case $\alpha \in [2, 4)$ with $\E[X_{11}^2] = 1$, Lemma~\ref{lem:allmoments} yields
\begin{equation*}
    \beta_4 \sim n^{-\alpha / 2} L(n^{1/2}) \frac{\alpha \Gamma(\alpha/2)\Gamma(2 - \alpha / 2)}{2 \Gamma(2)} = n^{-\alpha / 2} L(n^{1/2}) c_2(\alpha).
\end{equation*}
In view of the Potter bounds for the slowly varying function $L$, we have $L(n) = O(n^\varepsilon)$ for any $\varepsilon > 0$ so that $\beta_4 = O(n^{-\alpha / 2 + \varepsilon})$. If $\E[X_{11}^2] = \infty$ then we get another slowly varying function by Lemma~\ref{lem:allmoments}, which we bound analogously. Hence, we do not need to distinguish between $\E[X_{11}^2] = \infty$ and $\E[X_{11}^2] = 1$ in the case $\alpha \in [2, 4)$. Use of Lemma~\ref{lem:allmoments} now yields
\begin{equation*}
    \begin{aligned}
    \bigg(\beta_4 - \frac{1}{n^2}\bigg)^2 & \sim \beta_4^2 = O(n^{-\alpha + 2\varepsilon}),\\
    \frac{2}{n}\bigg(\beta_4 - \frac{1}{n}\bigg)^2 &= \frac{2}{n}\bigg(O(n^{-\alpha / 2 + \varepsilon}) - \frac{1}{n}\bigg)^2 \sim 2n^{-3}.   
    \end{aligned}
\end{equation*}
For $\alpha \in (3, 4)$, the $2n^{-3}$ term dominates and therefore, $\sigma_n^2 \sim 2np^2 \cdot 2n^{-3} = 4p^2n^{-2}$. While for $\alpha \in (2, 3)$ the $\beta_4^2$ term dominates and therefore,
\begin{equation*}
    \sigma_n^2 \sim 2np^2 \cdot \beta_4^2 \sim 2np^2 \cdot (n^{-\alpha / 2} L(n^{1/2}) c_2(\alpha))^2 = 2p^2n^{1-\alpha} L^2(n^{1/2}) c_2^2(\alpha).     
\end{equation*}
In the case $\alpha = 3$ both terms in \eqref{eq:var} might dominate (depending on $L$) and applying Lemma~\ref{lem:allmoments} one gets $\sigma_n^2 \sim 2p^2n^{-2}(c_2^2(\alpha) L^2(n^{1/2}) + 2)$. Finally, when $\alpha \in (0, 2)$, Lemma~\ref{lem:allmoments} yields
\begin{equation*}
    \beta_4 \sim n^{-1} \frac{\Gamma(2 - \alpha / 2)}{\Gamma(1 - \alpha / 2) \Gamma(2)} = n^{-1} \bigg(1 - \frac{\alpha}{2}\bigg),
\end{equation*}
from which we conclude
\begin{equation*}
    \begin{aligned}
    \bigg(\beta_4 - \frac{1}{n^2}\bigg)^2 &\sim \bigg(1 - \frac{\alpha}{2}\bigg)^2 n^{-2}, \\
    \frac{2}{n} \bigg(\beta_4 - \frac{1}{n}\bigg)^2 &\sim \frac{2}{n} \bigg(\frac{1 - \alpha /2}{n} - \frac{1}{n}\bigg)^2 = \frac{\alpha^2}{2} n^{-3},
    \end{aligned}
\end{equation*}
which implies $\sigma_n^2 \sim 2p^2 n^{-1} (1 - \alpha / 2)^2$. For completeness, we note that the fact that $\sigma_n^2\sim \Var(\tr(\bfR^2))$ is easily deduced from the above considerations in all cases. 
\end{proof}

\subsection{Proof of Lemma~\ref{lem:fourth_mom_asymp}}
\begin{proof}
We need to find the types of $\beta$'s that occur in $\E[T_i^4]$ for $i=1,2$. Let us start with $\E[T_1^4]$ which is given by 
\begin{equation*}
    \E{[T_1^4]} =
    16\,
    \sum_{\substack{i_1,i_2 = 1 \\ i_1 < i_2}}^p 
     \cdots
    \sum_{\substack{i_7,i_8 = 1 \\ i_7 < i_8}}^p
    \sum_{t_1, \ldots, t_4 = 1}^n
    \E \left[\bar{Y}_{i_1, t_1} \bar{Y}_{i_2, t_1} \bar{Y}_{i_3, t_2} \bar{Y}_{i_4, t_2}\bar{Y}_{i_5, t_3} \bar{Y}_{i_6, t_3} \bar{Y}_{i_7, t_4} \bar{Y}_{i_8, t_4}
    \right].
\end{equation*}
We proceed to the tedious task of finding all the combinations of indices that lead to non-zero contributions in the form of products of certain $\tbeta_{2k_1,\ldots, {2k_r}}$, whose order can be determined using Lemma~\ref{lem:tbeta} yielding that any $\tbeta_{2k_1,\ldots, {2k_r}}$ with no $k_i$ equal to $1$, will behave as $\beta_{2k_1,\ldots, {2k_r}}$ asymptotically. It is important to note that any term including $\widetilde{\beta}_2$ is zero. In a methodical manner one can check from big indices (of $\tbeta$) to small, terms with 2 up to 4 factors of $\beta$, which combinations of $\beta$'s are possible. For example, the term with largest possible index is $\beta_8^2$ and this happens only if $i_1 = i_3 = i_5 = i_7$ and $i_2 = i_4 = i_6 = i_8$ with all the $t$ indices equal. Accounting for multiplicities, the contribution of such terms to $\E[T_1^4]$ is equal to 
\begin{equation*}
    \frac{p(p-1)}{2} n \beta_8^2 \sim \frac{1}{2} p^2 n \beta_8^2.
\end{equation*}
\begin{table}[htb]
  \caption{Terms in $\E[T_1^4]$ and their orders for $\alpha \in (0,4)$.}
  \label{table6}
  \centering
  \renewcommand{\arraystretch}{1.1}
  \begin{tabular}{|c||c|c|c|}
    \hline
    \multicolumn{1}{|c||}{\textbf{Term in $\E[T_1^4]$}}&
    \multicolumn{1}{c|}{Order of $p$} &
    \multicolumn{1}{c|}{Order of $n,\,\alpha \in (0,2)$} &
    \multicolumn{1}{c|}{Order of $n,\,\alpha \in [2,4)$} \\
    \hline
    $12\,\beta_{4}^{4}\,n^{2}p^{4}$ & $p^{4}$ & $n^{-2}$ & $n^{2-2\alpha}$ \\
    \hline
    $60\,\beta_{4}^{4}\,n\,p^{4}$ & $p^{4}$ & $n^{-3}$ & $n^{1-2\alpha}$ \\
    \hline
    $24\,\beta_{4}^{2}\tbeta_{2,2}^{2}\,n^{3}p^{4}$ & $p^{4}$ & $O(n^{-3})$ & $O(n^{-\alpha-1})$ \\
    \hline
    $60\,\tbeta_{2,2}^{4}\,n^{4}p^{4}$ & $p^{4}$ & $O(n^{-4})$ & $O(n^{-4})$ \\
    \hline
    $192\,\beta_{4}\tbeta_{2,2}^{3}\,n^{3}p^{4}$ & $p^{4}$ & $O(n^{-4})$ & $O(n^{-\alpha/2-3})$ \\
    \hline
    $336\,\beta_{4}^{2}\tbeta_{2,2}^{2}\,n^{2}p^{4}$ & $p^{4}$ & $O(n^{-4})$ & $O(n^{-\alpha-2})$ \\
    \hline
    $144\,\tbeta_{2,2}^{4}\,n^{3}p^{4}$ & $p^{4}$ & $O(n^{-5})$ & $O(n^{-5})$ \\
    \hline
    $72\,\tbeta_{2,2}^{4}\,n^{2}p^{4}$ & $p^{4}$ & $O(n^{-6})$ & $O(n^{-6})$ \\
    \hline
    $48\,\beta_{4}^{2}\beta_{4,4}\,n^{2}p^{3}$ & $p^{3}$ & $n^{-2}$ & $n^{2-2\alpha}$ \\
    \hline
    $96\,\beta_{4}\beta_{6}^{2}\,n\,p^{3}$ & $p^{3}$ & $n^{-2}$ & $n^{1-3\alpha/2}$ \\
    \hline
    $48\,\beta_{4}^{2}\beta_{8}\,n\,p^{3}$ & $p^{3}$ & $n^{-2}$ & $n^{1-3\alpha/2}$ \\
    \hline
    $288\,\beta_{4}\tbeta_{4,2}^{2}\,n^{2}p^{3}$ & $p^{3}$ & $O(n^{-3})$ & $O(n^{-3\alpha/2})$ \\
    \hline
    $96\,\tbeta_{2,2}\tbeta_{4,2}^{2}\,n^{3}p^{3}$ & $p^{3}$ & $O(n^{-3})$ & $O(n^{-\alpha-1})$ \\
    \hline
    $192\,\beta_{6}\tbeta_{2,2}\tbeta_{4,2}\,n^{2}p^{3}$ & $p^{3}$ & $O(n^{-3})$ & $O(n^{-\alpha-1})$ \\
    \hline
    $192\,\beta_{4}\tbeta_{2,2}\tbeta_{6,2}\,n^{2}p^{3}$ & $p^{3}$ & $O(n^{-3})$ & $O(n^{-\alpha-1})$ \\
    \hline
    $96\,\beta_{4}\tbeta_{2,2}\tbeta_{4,2,2}\,n^{3}p^{3}$ & $p^{3}$ & $O(n^{-3})$ & $O(n^{-\alpha-1})$ \\
    \hline
    $192\,\tbeta_{2,2}\tbeta_{4,2}^{2}\,n^{2}p^{3}$ & $p^{3}$ & $O(n^{-4})$ & $O(n^{-\alpha-2})$ \\
    \hline
    $96\,\tbeta_{2,2}^{2}\beta_{4,4}\,n^{2}p^{3}$ & $p^{3}$ & $O(n^{-4})$ & $O(n^{-\alpha-2})$ \\
    \hline
    $96\,\beta_{4}\tbeta_{2,2,2}^{2}\,n^{3}p^{3}$ & $p^{3}$ & $O(n^{-4})$ & $O(n^{-\alpha/2-3})$ \\
    \hline
    $192\,\tbeta_{2,2}^{2}\tbeta_{4,2,2}\,n^{3}p^{3}$ & $p^{3}$ & $O(n^{-4})$ & $O(n^{-\alpha/2-3})$ \\
    \hline
    $384\,\tbeta_{2,2}\tbeta_{4,2}\tbeta_{2,2,2}\,n^{3}p^{3}$ & $p^{3}$ & $O(n^{-4})$ & $O(n^{-\alpha/2-3})$ \\
    \hline
    $96\,\tbeta_{2,2}\tbeta_{2,2,2}^{2}\,n^{4}p^{3}$ & $p^{3}$ & $O(n^{-4})$ & $O(n^{-4})$ \\
    \hline
    $48\,\tbeta_{2,2}^{2}\tbeta_{2,2,2,2}\,n^{4}p^{3}$ & $p^{3}$ & $O(n^{-4})$ & $O(n^{-4})$ \\
    \hline
    $8\,\beta_{8}^{2}\,n\,p^{2}$ & $p^{2}$ & $n^{-1}$ & $n^{1-\alpha}$ \\
    \hline
    $24\,\beta_{4,4}^{2}\,n^{2}p^{2}$ & $p^{2}$ & $n^{-2}$ & $n^{2-2\alpha}$ \\
    \hline
    $32\,\tbeta_{6,2}^{2}\,n^{2}p^{2}$ & $p^{2}$ & $O(n^{-2})$ & $O(n^{-\alpha})$ \\
    \hline
    $48\,\tbeta_{4,2,2}^{2}\,n^{3}p^{2}$ & $p^{2}$ & $O(n^{-3})$ & $O(n^{-\alpha-1})$ \\
    \hline
    $8\,\tbeta_{2,2,2,2}^{2}\,n^{4}p^{2}$ & $p^{2}$ & $O(n^{-4})$ & $O(n^{-4})$ \\
    \hline
  \end{tabular}
\end{table}
Since we are only interested in the asymptotic behavior of $\E[T_1^4]$, the exact multiplicities are not required. Hence, it is not necessary to know the exact number of occurring $\beta_8^2$ and similarly for all other products of $\tbeta_{2k_1,\ldots, {2k_r}}$. Then we look at how many of the $t$'s and $i$'s must not be equal respectively and hence get the order of the term by setting the corresponding amount in the exponent of $n$ and $p$. Finally, we focus on the terms with highest order.
Using the growth rate $p\asymp n^{\delta}$ (where $\delta>0$), we can replace $p$ so that each term is expressed entirely as a power of $n$. Then, we only keep the terms with the highest order $n$.

The fourth moment of $T_2$ 
\begin{equation*}
    \E[T_2^4]= 4^4 \,\E\left[\left(\sum_{i_1 < i_2}^p \sum_{t_1 < t_2}^n Y_{i_1 t_1} Y_{i_1 t_2} Y_{i_2 t_1} Y_{i_2 t_2}\right)^4\right]
\end{equation*}
is analyzed in a similar way. Tables~\ref{table6} and~\ref{table7} include all the products of $\tbeta_{2k_1,\ldots, {2k_r}}$ or $\beta_{2k_1,\ldots, {2k_r}}$ that occur in either $\E[T_1^4]$ or $\E[T_2^4]$ and give their order respective orders (utilizing Lemmas~\ref{lem:allmoments} and~\ref{lem:tbeta}) up to some slowly varying function for $\alpha \in (0, 4]$. Strictly speaking, in the case $\alpha = 2$ we need to distinguish between finite or infinite second moment of $X_{11}$; luckily both cases yield the same asymptotic results apart from a slowly varying function.
\begin{table}[htb]
  \caption{Terms in $\E[T_2^4]$ and their orders for $\alpha \in (0,4)$.}
  \label{table7}
  \centering
  \renewcommand{\arraystretch}{1.1}
  \begin{tabular}{|c||c|c|c|}
    \hline
    \multicolumn{1}{|c||}{\textbf{Term in $\E[T_2^4]$}}&
    \multicolumn{1}{c|}{Order of $p$} &
    \multicolumn{1}{c|}{Order of $n,\,\alpha \in (0,2)$} &
    \multicolumn{1}{c|}{Order of $n,\,\alpha \in [2,4)$} \\
    \hline
    $48\,\beta_{2,2}^{4}\,n^{4}p^{4}$                  & $p^{4}$ & $n^{-4}$ & $n^{-4}$ \\
    \hline
    $192\,\beta_{2,2}^{4}\,n^{3}p^{4}$                 & $p^{4}$ & $n^{-5}$ & $n^{-5}$ \\
    \hline
    $480\,\beta_{2,2}^{4}\,n^{2}p^{4}$                 & $p^{4}$ & $n^{-6}$ & $n^{-6}$ \\
    \hline
    $384\,\beta_{2,2}^{2}\beta_{4,4}\,n^{2}p^{3}$      & $p^{3}$ & $n^{-4}$ & $n^{-\alpha-2}$ \\
    \hline
    $768\,\beta_{2,2}^{2}\beta_{4,2,2}\,n^{3}p^{3}$    & $p^{3}$ & $n^{-4}$ & $n^{-\alpha/2-3}$ \\
    \hline
    $192\,\beta_{2,2}^{2}\beta_{2,2,2,2}\,n^{4}p^{3}$  & $p^{3}$ & $n^{-4}$ & $n^{-4}$ \\
    \hline
    $1536\,\beta_{2,2}\beta_{2,2,2}^{2}\,n^{3}p^{3}$   & $p^{3}$ & $n^{-5}$ & $n^{-5}$ \\
    \hline
    $64\,\beta_{4,4}^{2}\,n^{2}p^{2}$                  & $p^{2}$ & $n^{-2}$ & $n^{2-2\alpha}$ \\
    \hline
    $384\,\beta_{4,2,2}^{2}\,n^{3}p^{2}$               & $p^{2}$ & $n^{-3}$ & $n^{-\alpha-1}$ \\
    \hline
    $480\,\beta_{2,2,2,2}^{2}\,n^{4}p^{2}$             & $p^{2}$ & $n^{-4}$ & $n^{-4}$ \\
    \hline
  \end{tabular}
\end{table}
For $\E[T_1^4]$, several terms can only be controlled through upper bounds on their order based on Lemma~\ref{lem:tbeta}. We therefore proceed in two steps to find dominant terms. We first assume that these upper bounds are attained and use them to identify a list of five candidate dominant terms $\beta_{4}^4 n^2 p^4$, $\tbeta_{2,2}^{4}\,n^{4}p^{4}$, $\beta_{4}\tbeta_{2,2}^{3}\,n^{3}p^{4}$, $\beta_{4}^{2}\tbeta_{2,2}^{2}\,n^{2}p^{4}$ and $\beta_8^2 n p^2$. Since we have the actual order of $\tbeta_{2,2}$ in \eqref{eq:tbeta_2_2_rewritten}, we then compare the actual orders of all candidates and verify that only the two terms $\beta_4^4\,n^2p^4$ and $\beta_8^2\,np^2$ can be dominant, while the other candidates are of strictly smaller order. Note that the growth of $p$ relative to $n$ is an important detail that needs to be taken into consideration. We recall that our results are valid for general growth rates of $p$. 

Finding the exact multiplicities of $\beta_{4}^4$ and $\beta_{8}^2$  in $\E[T_1^4]$ is not too difficult and will yield 
\begin{equation*}
    \begin{aligned}
    \frac{1}{16}\,\E[T_1^4] 
    &\sim
    3 \bigg(\frac{p(p-1)}{2} n\bigg)^2 \beta_{4}^4  
    +
    \frac{p(p-1)}{2} pn \beta_8^2 \\ 
    &\sim 
    \frac{3}{4} n^2 p^4 \beta_{4}^4   
    +
    \frac{1}{2}\beta_8^2 n p^2.
    \end{aligned}
\end{equation*}

Now we turn to $\E[T_2^4]$. We find that the highest order term for $\alpha \in [3, 4)$ is $\beta_{2,2}^4n^4p^4$ by looking in Table~\ref{table7}. The multiplicity of $\beta_{2,2}^4n^4p^4$ is 3 which then for $\alpha \in [3, 4)$ yields 
\begin{equation*}
    \frac{1}{256}\, \E[T_2^4] \sim \frac{3}{16}\, \beta_{2,2}^4 n^4 p^4 \sim \frac{3}{16}\,  n^{-4} p^4,
\end{equation*}
completing the proof of the lemma.
\end{proof}

\subsection{Proof of Theorem~\ref{lem:completefourth}}
\begin{proof}
First, note that by the binomial theorem we have
\begin{equation*}
    \begin{aligned}
    \E\big[(\tr(\bfR^2)-\E[\tr(\bfR^2)])^4\big]
    &=
    \E\big[(T_1 + T_2)^4\big]
    \\
    &=
    \E[T_1^4] + 4\,\E[T_1^3 T_2] + 6\,\E[T_1^2 T_2^2] + 4\,\E[T_1 T_2^3] + \E[T_2^4].
    \end{aligned}
\end{equation*}
Recalling the definitions of $T_1$ and $T_2$ in \eqref{eq:def_T1_T2}, we observe that $T_2$ contains only odd powers of $Y_{it}$'s and $T_1$ only even powers and thus, we immediately see that $\E[T_1^3 T_2] = 0$. By Lemma~\ref{lem:E(T_i^4)/var^2_0}, we only need to study $\E[T_1^4]$ for $\alpha < 3$ and $\E[T_2^4]$ for $\alpha > 3$. Using Cauchy-Schwarz inequality we have 
\begin{equation}\label{eq:cauchy_scwarz}
    \frac{\E[T_1^2 T_2^2]}{\Var(\tr(\bfR^2))^2}
    \leq
    \bigg(\frac{\E[T_1^4]}{\Var(\tr(\bfR^2))^2} \bigg)^{1/2}
    \bigg(\frac{\E[T_2^4]}{\Var(\tr(\bfR^2))^2} \bigg)^{1/2},
\end{equation}
which goes to zero as $n \to \infty$ by Proposition~\ref{lem:E[T_i^4]/E[T_i^2]^2} and Lemma~\ref{lem:E(T_i^4)/var^2_0} assuming $p = \omega(n^\delta)$ for some $\delta > \delta^*(\alpha)$ with $\alpha \in (0, 3) \cup (3, 4)$. Similarly, using Hölder's inequality, we get 
\begin{equation}\label{eq:hölders_E[T_1T_2^3]}
    \frac{\big|\E[T_1 T_2^3] \big|}{\Var(\tr(\bfR^2))^2}
    \leq
    \frac{\E[T_1^4]^{1/4} \E[T_2^4]^{3/4}}{\Var(\tr(\bfR^2))^2}
    =
    \bigg(\frac{\E[T_1^4]}{\Var(\tr(\bfR^2))^2}\bigg)^{1/4}
    \bigg(\frac{\E[T_2^4]}{\Var(\tr(\bfR^2))^2}\bigg)^{3/4},
\end{equation}
which goes to 0 as $n \to \infty$ by Lemma~\ref{lem:E(T_i^4)/var^2_0} and Proposition~\ref{lem:E[T_i^4]/E[T_i^2]^2} if $p = \omega(n^\delta)$ for some $\delta > \delta^*(\alpha)$ when $\alpha \in (0, 3) \cup (3, 4)$. Another application of Proposition~\ref{lem:E[T_i^4]/E[T_i^2]^2} for $\E[T_1^4]$ and $\E[T_2^4]$ directly yields the desired result for $\alpha \in (0, 3) \cup (3, 4)$ and $\delta > \delta^*(\alpha)$. 

Finally, if $p = o(n^\delta)$ for some $\delta < \delta^*(\alpha)$ one can check that the leading order term $\E[T_1^4]/\Var(\tr(\bfR^2))^2$ tends to infinity by Proposition~\ref{lem:E[T_i^4]/E[T_i^2]^2} and Lemma~\ref{lem:E(T_i^4)/var^2_0}, which establishes the desired claim for $\delta < \delta^*(\alpha)$. 
\end{proof}

\subsection{Proof of Proposition~\ref{lem:E[T_i^4]/E[T_i^2]^2}}
\begin{proof}
From Lemmas~\ref{lem:fourth_mom_asymp} and~\ref{lem:variance} we have for $\alpha \in (0, 4)$
\begin{equation*}
    \E[T_1^4] 
    \sim
    12\beta_4^4 n^2 p^4 + 8\beta_8^2 n p^2, 
    \quad
    \E[T_1^2]^2
    \sim
    4\beta_4^4 n^2 p^4 
\end{equation*}
which yields 
\begin{equation}
    \frac{\E[T_1^4]}{\E[T_1^2]^2} \sim 3 + \frac{2 \beta_8^2}{\beta_4^4 p^2 n}.
\end{equation}
By Lemma~\ref{lem:allmoments}, we see that $\beta_4, \beta_6, \beta_8$ are of the same order up to some slowly varying function. Thus, if $\beta_4^2 p^2 n \ell(n) \to \infty$  as $\nto$ for any slowly varying function $\ell$, then $\E[T_1^4] /\E[T_1^2]^2 \sim 3$. This clearly holds if $\alpha\in(0,3]$ and $p = \omega(n^\delta)$ for some $\delta > \delta^*(\alpha)$ since by Lemma~\ref{lem:allmoments} $\beta_4$ behaves like $n^{-\max(1,\alpha/2)}$ (up to a slowly varying function). If instead $p = o(n^\delta)$ for some $\delta < \delta^*(\alpha)$, then we have
\begin{equation*}
    \limn \frac{\E[T_1^4]}{\E[T_1^2]^2} = \infty.
\end{equation*}
Finally, if $\alpha \in [3, 4)$ then by \eqref{eq_E(T2_squared))} and Lemma~\ref{lem:allmoments} we have 
\begin{equation*}
    \E[T_2^2]^2
    = \bigg(4p(p-1)\frac{n}{n-1}\left(\beta_4 - \frac{1}{n}\right)^2\bigg)^2
    \sim 16p^4n^{-4}
\end{equation*}
and we find by combining with Lemma~\ref{lem:fourth_mom_asymp} that $\E[T_2^4] / \E[T_2^2]^2 \to 3$ for $\alpha \in [3, 4)$.
\end{proof}

\subsection{Proof of Lemma~\ref{lem:E(T_i^4)/var^2_0}}
\begin{proof}
Combining Lemmas~\ref{lem:asymp_variance} and~\ref{lem:fourth_mom_asymp} we have for $\alpha \in (3, 4)$
\begin{equation*}
    \frac{\E[T_1^4] }{\Var(\tr(\bfR^2))^2} 
    \sim \frac{12\, \beta_4^4 n^2 p^4 +8\beta_8^2 n p^2}{16n^{-4}p^4}, \qquad n \to \infty.   
\end{equation*}
By Lemma~\ref{lem:allmoments}, we get 
\begin{equation*}
    \frac{12\, \beta_4^4 n^2 p^4 }{16n^{-4}p^4}
    \sim \frac{3}{4}n^{-2(\alpha - 3)}\, c_2^4(\alpha) L(n^{1/2})^4
    \to 0, \qquad n \to \infty.
\end{equation*}
Moreover, assuming $p = \omega(n^\delta)$ for some $\delta > (5-\alpha)/2$, we have that 
\begin{equation*}
    \frac{8\beta_8^2 n p^2}{16n^{-4}p^4} \sim \frac{1}{2}n^{-\alpha + 5} p^{-2} L(n^{1/2})^2 c_{4}^2(\alpha) \to 0.    
\end{equation*}
We conclude that
\begin{equation}\label{eq:dfghe}
    \frac{\E[T_1^4] }{\Var(\tr(\bfR^2))^2} \to 0, \qquad n \to \infty   
\end{equation}
for $\alpha \in (3, 4)$ assuming $p = \omega(n^\delta)$ for some $\delta > (5-\alpha)/2$. 

For $\alpha \in [2, 3)$ and $\E[X_{11}^2] = 1$, the dominant terms in $\E[T_2^4]$ are $\beta_{4,4}^2 n^2 p^2$ and $\beta_{2,2}^4 n^4 p^4$ (see Table~\ref{table7}). Using the variance in Lemma~\ref{lem:asymp_variance} combined with Lemma~\ref{lem:allmoments} we get 
\begin{equation*}
    \frac{\E[T_2^4] }{\Var(\tr(\bfR^2))^2}
    \lesssim 
    \frac{ n^{-2\alpha + 2} p^2 c_{2,2}^2(\alpha) L^4(n^{1/2}) + n^{-4} p^4 c_{1,1}^4(\alpha)}
    {4p^4n^{2-2\alpha} L^4(n^{1/2}) c_2^4(\alpha)}
    \to 0.
\end{equation*}
If for $\alpha = 2$ we instead have $\E[X_{11}^2] = \infty$, then the above still holds but with a different slowly varying function in view of Lemma~\ref{lem:allmoments}. Now for $\alpha \in (0, 2)$ using Lemmas~\ref{lem:asymp_variance},~\ref{lem:fourth_mom_asymp} and~\ref{lem:allmoments} we get 
\begin{equation*}
    \frac{\E[T_2^4] }{\Var(\tr(\bfR^2))^2}
    \sim \frac{48n^{-4}p^4}{4p^4 n^{-2} (1 - \alpha / 2)^4} \bigg(\frac{\alpha}{2}\bigg)^4
    \to 0, \qquad n \to \infty,
\end{equation*}
which means that $\frac{\E[T_2^4] }{\Var(\tr(\bfR^2))^2} \to 0$, as $\nto$, for all $\alpha \in (0, 3)$.
\end{proof}

\subsection{Proof of Lemma~\ref{lem:boundary-scaling-equivalence}}
\begin{proof}[Proof of Lemma~\ref{lem:boundary-scaling-equivalence}]
Since
\begin{equation*}
    \E[T_1^4]\sim 12\beta_4^4n^2p^4+8\beta_8^2np^2, \qquad\E[T_1^2]^2\sim 4\beta_4^4n^2p^4,
\end{equation*}
it follows that, as $\nto$,
\begin{equation}\label{eq:extra}
    \frac{\E[T_1^4]}{\E[T_1^2]^2} \sim 3+\frac{2\beta_8^2}{\beta_4^4p^2n}.
\end{equation}
Hence the boundary condition \eqref{eq:boundary-def} is equivalent to requiring that the extra term in \eqref{eq:extra} converges to a finite positive constant. For $\alpha\in(0,2)$, using \eqref{eq:extra} together with the asymptotics of $\beta_4$ and $\beta_8$, we see that
\begin{equation*}
    \begin{aligned}
    \frac{2\beta_8^2}{\beta_4^4p^2n}
    &=\frac{2n}{p^2} \,\frac{(n\beta_8)^2}{(n\beta_4)^4}\\
    &\sim \frac{2n}{p^2}\,\frac{c_4(\alpha)^2}{c_2(\alpha)^4}
    =\frac{2n}{p^2}\, \bigg[\frac{\bigl(3-\frac{\alpha}{2}\bigr)\bigl(2-\frac{\alpha}{2}\bigr)}
    {6\bigl(1-\frac{\alpha}{2}\bigr)}\bigg]^2\,.
    \end{aligned}
\end{equation*}
The \rhs~ converges to $C_\alpha\in(0,\infty)$
if and only if
\begin{equation*}
    \frac{p}{n^{1/2}}\to c_\delta\in(0,\infty).
\end{equation*}
Since $\delta^*(\alpha)=1/2$ for $\alpha\in(0,2)$, this is exactly the asserted boundary scaling. The displayed formula for $C_\alpha$ follows immediately.

For $\alpha\in(2,3)$, using \eqref{highestmoment24}, we obtain
\begin{equation*}
    \frac{2\beta_8^2}{\beta_4^4p^2n}
    \sim \frac{2 n^{\alpha-1}}{p^2L(n^{1/2})^2}\,\frac{c_4(\alpha)^2}{c_2(\alpha)^4}
    = \frac{2 n^{\alpha-1}}{p^2L(n^{1/2})^2}\, \bigg[
    \frac{\bigl(3-\frac{\alpha}{2}\bigr)\bigl(2-\frac{\alpha}{2}\bigr)}{6 \Gamma(1+\frac{\alpha}{2})\Gamma(2-\frac{\alpha}{2})}\bigg]^2.
\end{equation*}
Hence \eqref{eq:boundary-def} holds if and only if
\begin{equation*}
\frac{p\,L(n^{1/2})}{n^{(\alpha-1)/2}}
\to c_\delta\in(0,\infty).
\end{equation*}
Since $\delta^*(\alpha)=(\alpha-1)/2$ for $\alpha\in(2,3)$, this is exactly the required scaling, and the expression for $C_\alpha$ follows.
\end{proof}
\FloatBarrier

\begin{appendix}
\section*{Sums of regularly varying random variables}
\begin{lemma}[Jessen and Mikosch \cite{jessen:mikosch:2006}] \label{lem:L(x)_sym_pareto}
Assume $|X_1|$ is regularly varying with index $\alpha \geq 0$ and distribution function $F = 1 - \bar F$. Assume $X_1,...,X_n$ are random variables satisfying 
\begin{equation}\label{eq:cond_1_L(x)_pareto}
    \lim_{x \to \infty} \frac{\P(X_i > x)}{\bar F(x)} = c_i^+ \quad \text{and} \quad
    \lim_{x \to \infty} \frac{\P(X_i \leq -x)}{\bar F(x)} = c_i^- , \quad i = 1,\dots,n,
\end{equation}
for some non-negative numbers $c_i^\pm$ and
\begin{equation}
\begin{aligned}
    \lim_{x \to \infty} \frac{\P(X_i > x, X_j > x)}{\bar F(x)} 
    &=
    \lim_{x \to \infty} \frac{\P(X_i \leq -x, X_j > x)}{\bar F(x)} \\
    &=
    \lim_{x \to \infty} \frac{\P(X_i \leq -x, X_j \leq -x)}{\bar F(x)}
    =
    0, \quad i \neq j. \label{eq:cond_2_L(x)_pareto}
\end{aligned}
\end{equation}
Then 
\begin{equation*}
    \lim_{x \to \infty} \frac{\P(S_n > x)}{\bar F(x)} = c_1^+ + \dots + c_n^+
    \quad \text{and} \quad
    \lim_{x \to \infty} \frac{\P(S_n \leq x)}{\bar F(x)} = c_1^- + \dots + c_n^- , 
\end{equation*}
with $S_n = X_1 + \dots + X_n$, $n \geq 1$.  
\end{lemma}
\end{appendix}


\begin{funding}
This research was partially supported by the Swedish Research Council through VR-2023-03577 ``High-dimensional extremes and random matrix structures'' and by the Verg-Foundation.
\end{funding}
\bibliographystyle{imsart-number}
\bibliography{library}
\end{document}